\documentclass[12pt,oneside]{amsart}
\usepackage{amsmath}
\usepackage{amssymb}
\usepackage{amstext}
\usepackage{amscd}
\usepackage[all]{xy}
\usepackage{hyperref} %%Clickable reference in PDF
\usepackage{soul}
\usepackage{color} %% For colorful text
\usepackage{typearea} %%Makes better use of the page, IMHO

\usepackage{pdfcomment}

\newtheorem{teo}{Theorem}[section]
\newtheorem{prop}[teo]{Proposition}
\newtheorem{lem}[teo]{Lemma}
\newtheorem{cor}[teo]{Corollary}
\newtheorem{conj}[teo]{Conjecture}

\newtheorem{defini}[teo]{Definition}
\newtheorem{ques}[teo]{Question}
\newtheorem{rem}[teo]{Remark}
\newtheorem{contr}[teo]{Construction}
\newtheorem{conv}[teo]{Convention}
\newtheorem{notation}[teo]{Notation}

\makeatletter
\newcommand{\neutralize}[1]{\expandafter\let\csname c@#1\endcsname\count@}
\makeatother

\newenvironment{teobis}[1]
  {\renewcommand{\theteo}{\ref*{#1}$'$}%
   \neutralize{teo}\phantomsection
   \begin{teo}}
  {\end{teo}}
  
 \newenvironment{teobisbis}[1]
  {\renewcommand{\theteo}{\ref*{#1}$''$}%
   \neutralize{teo}\phantomsection
   \begin{teo}}
  {\end{teo}}

\numberwithin{equation}{section}

\newcommand{\Hom}{\mbox{Hom}}

\newcommand{\IQbar}{\overline{\mathbb{Q}}}

\newcommand{\pullbackcorner}[1][dr]{\save*!/#1-1.7pc/#1:(-1.5,1.5)@^{|-}\restore}

\newcommand{\End}{{\rm End}}
\newcommand{\GSp}{{\rm GSp}}

\newcommand{\SL}{{\rm SL}}

\newcommand{\Lie}{\operatorname{Lie} }

\newcommand{\Gal}{{\rm Gal}}

\newcommand{\der}{{\rm der}}

\newcommand{\CC}{{\mathbb C}}
\newcommand{\RR}{{\mathbb R}}
\newcommand{\ZZ}{{\mathbb Z}}
\newcommand{\QQ}{{\mathbb Q}}

\newcommand{\PP}{{\mathbb P}}

\newcommand{\GG}{{\mathbb G}}
\newcommand{\SSS}{{\mathbb S}}
\newcommand{\AAA}{{\mathbb A}}

\newcommand{\lto}{\longrightarrow}

\newcommand{\Sp}{{\rm Sp}}

\newcommand{\bx}{{\bf x}}

\newcommand{\bG}{{\bf G}}
\newcommand{\bT}{\mathbf{T}}
\newcommand{\bH}{\mathbf{H}}

\def\Fk{\mathfrak{k}}
\def\Fp{\mathfrak{p}}

\def\Fg{\mathfrak{g}}
\def\Fh{\mathfrak{h}}

\def\Fm{\mathfrak{m}}

\def\Fz{\mathfrak{z}}

\newcommand{\cH}{{\mathcal H}}

\newcommand{\cB}{{\mathcal B}}

\newcommand{\cA}{{\mathcal A}}

\newcommand{\cD}{{\mathcal D}}

\newcommand{\cW}{{\mathcal W}}

\newcommand{\cO}{{\mathcal O}}

\newcommand{\ol}{\overline}

\newcommand{\oQ}{\overline{\QQ}}

\newcommand{\diag}{{\rm diag}}

\renewcommand{\cong}{\simeq}

\newcommand{\Gm}{\mathbb{G}_{\mathrm{m}}}
\newcommand{\GmF}[1]{\mathbb{G}_{\mathrm{m,#1}}}

\title[Bi-$\IQbar$ Shimura]{Bi-$\IQbar$-structures on Hermitian symmetric spaces and quadratic relations between CM periods}
\author{Ziyang Gao, Emmanuel Ullmo, Andrei Yafaev}
\begin{document}

\address{Department of Mathematics, UCLA, Los Angeles, CA 90095, USA}
\email{ziyang.gao@math.ucla.edu}

\address{IHES, Universit\'{e} Paris-Saclay, Laboratoire Alexandre Grothendieck; 35 route de Chartres, 91440 Bures sur Yvette, France.}
\email{ullmo@ihes.fr}

\address{Department of Mathematics, University College London; Gower street, WC1E 6BT London, UK.}
\email{yafaev@math.ucl.ac.uk}

\maketitle

\begin{abstract}
In this paper, we introduce the notion of a bi-$\IQbar$-structure on the tangent space at a CM point on a locally Hermitian symmetric domain. We prove that this bi-$\IQbar$-structure decomposes into the direct sum of $1$-dimensional bi-$\IQbar$-subspaces, and make this decomposition explicit for the moduli space of abelian varieties $\mathbb{A}_g$. 

We propose an Analytic Subspace Conjecture, which is the analogue of the W\"{u}stholz's Analytic Subgroup Theorem in this context. We  show that this conjecture, applied to $\mathbb{A}_g$, implies that all quadratic $\IQbar$-relations among the holomorphic periods of CM abelian varieties arise from elementary ones.
\end{abstract}

\tableofcontents

\section{Introduction}
\subsection{Prologue and motivation} 
In all this paper, $\oQ$ will denote the algebraic closure of $\QQ$ in $\CC$. 
The uniformizing map $u \colon \CC \to \CC^*$, $z \mapsto e^{2\pi i z}$, 
 is a map between algebraic varieties defined over $\IQbar$. %If $\alpha \in \QQ$, then $\alpha \in \IQbar$ and  $u(\alpha)\in \IQbar^*$. 
This map is transcendental, and the Gelfond--Schneider theorem asserts the following: if $\alpha \in \IQbar$ and $u(\alpha)\in \IQbar^*$, then $\alpha \in \QQ$. % Notice that $u(\alpha)$ is then a root of unity. 
On the other hand, if we consider the usual exponential $\exp \colon \CC \to \CC^*$,  $z \mapsto e^z$, the situation is different. 
In this case, if $\alpha \in \IQbar$ and $e^{\alpha} \in \IQbar^*$, then $\alpha = 0$ by a theorem of Lindemann. %; see \cite[]{UllmoZP} for a proof using W\"{u}stholz's Analytic Subgroup Theorem.

We can look at this  from a different viewpoint, which we will further develop and adapt in this paper. Consider the natural $\IQbar$-structure on  $\CC^*= \Gm(\CC) $  given by $ \GmF{\CC} = \GmF{\IQbar} \otimes \CC$, and the usual exponential map $\exp \colon \CC \to \Gm(\CC) = \CC^*$. Endow $\CC$ with \textit{two} $\IQbar$-structures: one given by $W_1 := 2\pi i \IQbar \subseteq \CC$, the other given by $W_2: = \IQbar \subseteq \CC$. An $\alpha \in \CC$ is called \textit{bi-$\IQbar$} for $W_j$ if $\alpha \in W_j$ and $e^{\alpha} \in \IQbar^*$. From the previous paragraph, $\alpha \in \CC$ is bi-$\IQbar$ for $W_1$ (resp. $W_2$) if and only if $\alpha \in \QQ$ (resp. $\alpha = 0$). These two $\IQbar$-structures on $\CC$ are related by $2\pi i$ which is a transcendental number and is a period of $\GG_m$. We will call $(W_1,W_2)$ a \textit{bi-$\IQbar$-structure on $\CC$ with period $2\pi i$}.

\vskip 0.3em

%The situation in the case of %the uniformisation map of 
%abelian varieties is in general  more complicated, but in the CM case. 
The situation of bi-$\IQbar$-structure also occurs in the case of complex CM abelian varieties $A$. Let $g = \dim A$. By definition if 
$$
A\simeq A_1^{n_1}\times \dots\times A_r^{n_r}
$$
 is the isotypic decomposition of $A$, then each $A_i$ is a simple CM abelian variety. If $g_i$ denotes the dimension of $A_i$, then $E_i:=\End (A_i)\otimes \QQ$ is a CM field of dimension $2g_i$. In this situation $\End(A)\otimes \QQ=M_{n_1}(E_1)\times \dots\times M_{n_r}(E_r)$ is a CM-algebra.

 In $\mathsection$\ref{SubsectionExampleCMAV}, 
we will see that $T_0 A$ is naturally endowed with a bi-$\oQ$-structure. Moreover, we show in Proposition~\ref{PropBiQbarCMAVSplit} that this bi-$\IQbar$-structure is \textit{split}, which means  
\begin{equation}\label{TanCM}
T_0 A=\bigoplus_{j=1}^g \CC(\theta_j)
\end{equation}
where each $\CC(\theta_j)$ is a $1$-dimensional complex vector subspace endowed with the induced bi-$\oQ$-structure. The complex number $\theta_j$ measures the difference between the two $\oQ$-structures and is defined up to multiplication by some element in $\oQ^*$, and they are precisely the holomorphic periods of $A$ defined as follows. For each $1$-form $\omega \in H^1_{\mathrm{dR}}(A)$, $\int_{\gamma}\omega$ is independent of the choice of $\gamma \in H_1(A,\ZZ)$ up to $\IQbar$ and is non-zero for some $\gamma \in H_1(A,\ZZ)$ by Shimura \cite[Rmk.3.4]{ShimuraPeriod}. Moreover if $\int_{\gamma}\omega\neq 0$ and $\int_{\gamma'}\omega\neq 0$, then there exists $\theta\in \oQ^*$ such that $\int_{\gamma}\omega=\theta \int_{\gamma'}\omega$. Note that Shimura \cite[Rmk.3.4]{ShimuraPeriod} has the hypothesis that $\End(A)\otimes \QQ$ has a totally real subfield of dimension $\dim(A)$, but we can apply his result to every isotypical factors of the above decomposition of $A$.

\begin{conv}\label{SubsectionConvention}
Let $\{\omega_1,\ldots,\omega_g\}$ be an $\mathrm{End}(A)\otimes \QQ$-eigenbasis of $\Omega_A$ (holomorphic $1$-forms). We can choose $\gamma_j\in \Gamma$ such that $\theta_j := \int_{\gamma_j} \omega_j \not=0$ for each $j \in \{1,\ldots,g\}$. We call $\theta_1,\ldots,\theta_g$ the \textit{holomorphic periods} of $A$; they are well-defined up to $\IQbar^*$.
%Let $\eta_j$ be the complex conjugate of $\omega_j$. Then the Shimura relation \cite[]{} asserts that $\int_{\gamma} \eta_j \cong 2\pi i/\theta_j$ for each $j \in \{1,\ldots,g\}$, and we call them the \textit{anti-holomorphic periods} of $A$.
\end{conv}

A very general conjecture in transcendental number theory is Grothendieck's Period Conjecture and its generalizations by Andr\'{e} and Konsevich--Zagier, which predicts the polynomials relations between the periods. W\" ustholz's  famous Analytic subgroup theorem (WAST) \cite{WAST} gives a complete answer of the linear relations of this conjecture for linear algebraic groups and more generally for $1$-motives; it includes most known results in transcendental number theory. We refer to \cite{AndreMotifs},  \cite{Tretkoff} and \cite{HuberWustholz} for relevant discussions. For a CM abelian variety $A$ and its holomorphic periods defined above, WAST implies: the $\theta_j$'s are transcendental numbers, and that they are $\IQbar$-linearly independent if $A$ has no square factors.\footnote{By the Poincar\'{e} Irreducibility Theorem, $A$ is isogenous to a product of simple CM abelian varieties. We say that \textit{$A$ has no square factors} if these simple CM abelian varieties are all distinct. }

However, almost nothing is known beyond the linear relations (except if the question can be reduced to linear case, for example the transcendence of $\theta_j\theta_{j'}/\pi$). For example, little is known for quadratic relations among the $\theta_j$'s.

In a subsequent paper by the first- and second-named authors \cite{GUHodgeCycle}, we will present an example of non-trivial quadratic relation which takes the following simple form: 
\begin{equation}\label{EqElementaryNonTrivial}
\theta_j\theta_{j'}  - c_{jj'kk'}  \theta_k\theta_{k'} = 0 \quad \text{ with } \{j,j'\}\not=\{k,k'\} \text{ and  }c_{jj'kk'} \in \IQbar^*.
\end{equation}
In this example, $A$ is simple and has dimension $2^5$. 
Based on this observation, we make the following definition.

\begin{defini}\label{DefnElementaryNonTrivial}
A quadratic relation among $\theta_1,\ldots,\theta_g$ is called \textbf{elementary} if it is of the form \eqref{EqElementaryNonTrivial}.
\end{defini}
Now, a natural question is:
\begin{ques}\label{QuestionNoQuadraticRelationIntro}
Is it true that any non-trivial $\IQbar$-quadratic relation among the $\theta_j$'s, or equivalently any $\IQbar$-linear relation among the $\theta_j\theta_{j'}$'s, is a $\IQbar$-linear combination of elementary ones?

Moreover, if the Mumford--Tate group of $A$ is a maximal torus, is it true that 
\begin{equation}
\dim_{\IQbar} \sum_{1\le j \le j'\le g} \IQbar \theta_j\theta_{j'} = \frac{g(g+1)}{2}?
\end{equation}
\end{ques}

The aim of this article is to propose a possible framework to study this question. % which is more tractable than  Grothendieck's Period Conjecture. 
We will introduce a natural bi-$\IQbar$-structure on the tangent space at a CM point on a Shimura variety, prove that it decomposes into the direct sum of $1$-dimensional bi-$\IQbar$-subspaces $\CC(\alpha_j)$ (and hence yields complex numbers $\alpha_j$ which measure the differences between the two $\IQbar$-structures on each $\CC(\alpha_j)$), and compute the $\alpha_j$'s in the case of the moduli space of abelian varieties $\mathbb{A}_g$. We then propose an Analytic Subspace Conjecture, which is the analogue of WAST in this context, and show that this conjecture when applied to $\mathbb{A}_g$ gives an affirmative answer to Question~\ref{QuestionNoQuadraticRelationIntro}.

\subsection{Bi-$\IQbar$-structure associated with Shimura varieties} 
 
For Shimura varieties, we use Deligne's notations \cite{DeligneTravaux-de-Shim, De}. A quick summary which suffices for our use can be found in $\mathsection$\ref{SubsectionRecallShimura}.

Let $(\bG,X)$ be a connected Shimura datum. Here $\bG$ is a reductive group defined over $\QQ$ and $X$ is a Hermitian symmetric domain. % and let
Let $\Gamma$ be an arithmetic subgroup of $\bG(\QQ)_+$ which acts on $X$. The Baily--Borel theorem \cite{BailyCompactificatio} asserts that $S := \Gamma\backslash X$ has a unique structure of quasi-projective complex algebraic variety. Write 
\[
u \colon X \lto S
\]
for the uniformizing map. Moreover, the general theory of Shimura varieties asserts that 
%The variety $S$ is quasi-projective and defined over $\ol{\QQ}$. %For all facts, notations and some proofs
 $S$ admits a unique model over $\ol{\QQ}$ and we write $S_{\IQbar}$ for this model. % and we refer to points of $S$ 

Let $[o] \in S(\IQbar)$ and $o \in X$ such that $u(o) = [o]$. 
%Let $z$ be a point in $S(\oQ)$ and $\tau \in X$ such that $u(\tau)=z$.
 The holomorphic tangent space $T_o X$ is a finite dimensional complex vector space. Assume furthermore that $o$ is special point (or CM point).  
 We endow $T_o X$ with two $\IQbar$-structures, the \textit{arithmetic $\IQbar$-structure} and the \textit{geometric $\IQbar$-structure} as follows.%$T_{\tau} (X)$ of two $\oQ$-structures, the {\bf Arithmetic} and the {\bf Geometric} structure. 

\subsubsection{The arithmetic $\IQbar$-structure on $T_o X$ }\label{SubsectionIntroArithmeticIQbar}
 The algebraic tangent space $T_{[o]} S_{\IQbar}$ is a finite dimensional $\oQ$-vector space and we have a  $\CC$-isomorphism
$$
\mathrm{d}u \colon  T_o X\longrightarrow T_{[o]} S_{\IQbar}\otimes \CC.
$$
Then $(\mathrm{d}u)^{-1}(T_{[o]} S_{\IQbar})$ defines a $\oQ$-structure on $T_o X$. We call it the \textit{arithmetic $\IQbar$-structure}; see $\mathsection$\ref{SubsectionArithmeticIQbarShimura} for more details.

\subsubsection{The geometric $\IQbar$-structure on $T_o X$ }\label{SubsectionIntroGeomIQbar}
The complex dual $X^{\vee}$ of $X$ is a projective algebraic variety. If a faithful rational representation $\rho$ of $\bG$ on a finite dimensional $\QQ$-vector space $V$, then $X^{\vee}$ is naturally a subvariety of a flag variety defined over $\oQ$. Write $X^\vee_{\IQbar}$ for this $\IQbar$-structure on $X^\vee$. The natural inclusion $\iota \colon X\rightarrow X^{\vee}$ is a holomorphic map and $X$ is an open subset of $X^{\vee}$ in the usual topology. Since $o$ is a special point, $\iota(o)$ is a $\oQ$-point of $X^{\vee}$. So $T_{\iota(o)} X^{\vee}$ is a $\oQ$-vector space of dimension $\dim X$. We have an isomorphism 
$$
\mathrm{d}\iota \colon T_o X\longrightarrow T_{\iota(o)} X^{\vee}_{\IQbar}\otimes \CC
$$
and 
$(\mathrm{d}\iota)^{-1}(T_{\iota(o)} X^{\vee}_{\IQbar})$ defines a $\oQ$-structure on $T_o X$. We call it the \textit{geometric $\IQbar$-structure}; see 
$\mathsection$\ref{SubsectionGeometricIQbarShimura} for more details.

\vskip 0.3em
The first result of this work is the following statement.

\begin{teo}[{Theorem~\ref{ThmBiIQbarShimuraSplit}}]
The bi-$\oQ$-structure on $T_o X$ defined above is split, \textit{i.e.} we have a decomposition
$$
T_o X=\bigoplus_{i=1}^{\dim X} \CC(\alpha_j),
$$
where each $\CC(\alpha_j)$ is a $1$-dimensional complex vector space endowed with the restriction of the bi-$\oQ$-structures on $T_o X$ and $\alpha_j\in \CC/\oQ^*$ compares the two $\oQ$-structures on $\CC(\alpha_j)$.
\end{teo}

We can be much more precise when $S=\AAA_g=\Sp(2g,\ZZ)\backslash \mathfrak{H}_g$ is the moduli space of principally polarized abelian varieties of dimension $g$.

\begin{teo}[{Theorem~\ref{ThmPeriodsSiegel}}]\label{ThmPeriodsSiegelIntro}
Let $[o]$ be a CM point of $\AAA_g(\IQbar)$ corresponding to a CM abelian variety $A$ with holomorphic periods $\theta_1,\ldots,\theta_g$. 
%$$
%T_o(A)=\bigoplus_{i=1}^g \CC(\theta_i)
%$$
%according to equation \ref{TanCM}.
 Let $o\in \mathfrak{H}_g$ mapping to $[o]$. Then 
\begin{equation}\label{TanSiegel}
T_o \mathfrak{H}_g= \bigoplus_{1\leq j\leq j'\leq g}\CC\left(\frac{\theta_j\theta_{j'}}{\pi}\right).
\end{equation}
\end{teo}
Moreover, we prove that the decompositions \eqref{TanSiegel} and \eqref{TanCM} are compatible with the Kodaira--Spencer map; see Theorem~\ref{ThmKSSplitIQbarStru}. Better, we show that each $\CC\left(\frac{\theta_j\theta_{j'}}{\pi}\right)$ arises naturally as a \textit{root} space; see Theorem~\ref{ThmDecompositionIntoRootSpacesSiegel} and Corollary~\ref{CorDecompositionIntoRootSpacesSiegel} for the precise statements.

\vskip 0.3em
The following conjecture of Lang was proved by Cohen--Wolfart \cite{CW}. 
As an application of this framework of bi-$\IQbar$-structures and of Theorem~\ref{ThmPeriodsSiegelIntro}, we give a new and simpler proof.
\begin{prop}[{Proposition~\ref{PropLangQuestion}}]\label{PropLangQuestionIntro}
Let $C$ be a connected Shimura variety of dimension $1$ of genus $>1$. Suppose that the universal holomorphic covering map 
\[
\varphi \colon E_{\rho} := \{z \in \CC : |z| < \rho\} \rightarrow C^{\mathrm{an}}
\]
is normalized in such a way that $\varphi(0) \in C(\IQbar)$ is a CM point and that $\varphi'(0) \in \IQbar$. Then $\rho$ is a transcendental number.
\end{prop}

\subsection{The Hyperbolic Analytic Subspace Conjecture}

Another main objective of this article is to formulate a \textit{hyperbolic analytic subspace conjecture}, which is the analogue of WAST in this context, and to derive some of its immediate consequences. We start by recalling W\"{u}stholz's famous analytic subgroup theorem.
\begin{teo}[WAST, {\cite{WAST}}]\label{ThmWAST}
Let $G$ be a commutative algebraic group over $\overline{\QQ}$.  Let $\mathfrak{g}=\Lie G= T_0 G$ viewed as a $\IQbar$-vector space, and consider the exponential map 
\[
\exp_G \colon \mathfrak{g}_{\CC} \to G(\CC).
\]
% be the exponential map.
Let $z\in \mathfrak{g}_{\CC}$ be such that $\exp_G(z)\in G(\overline{\QQ})$. %Let $V_z$ be the smallest $\overline{\QQ}$-vector subspace of $\mathfrak{g}$ such that $z\in V_z\otimes \CC$. 
Let $V$ be a sub-vector space of the $\overline{\QQ}$-vector subspace of $\mathfrak{g}$ with $z\in V\otimes \CC$. 
Then there exists a connected algebraic subgroup $H$ of $G$ defined over $\IQbar$ which contains $\exp_G(z)$ such that $V \supseteq T_0 H$.
\end{teo}

In the context of Shimura varieties, the tangent space $T_0 G$ corresponds to the arithmetic $\IQbar$-structure on $T_o X$ defined in $\mathsection$\ref{SubsectionIntroArithmeticIQbar} and the natural $\IQbar$-structure on the uniformizing space gives rise to the geometric $\IQbar$-structure on $T_o X$ defined in $\mathsection$\ref{SubsectionIntroGeomIQbar}. Hence the analogue of WAST in this context, which we propose, is the following conjecture.
\begin{conj}[HASC; {Conjecture~\ref{ConjAnalyticSubspace}}]\label{HWASP}
Let $S=\Gamma\backslash X$ be a connected Shimura variety. 
Let $o$ be a special point of $X$.  Let $\cD \subseteq T_o X$ be  the Harish--Chandra realization of $X$  as a bounded symmetric domain centred at $o$. Let $u\colon \cD\longrightarrow S$ be the uniformizing map. % and 
Let $z \in \cD$ be such that $[z]:=u(z) \in S(\IQbar)$. %Set $V_z$ to be the smallest
Let $V$ be an arithmertic $\IQbar$-subspace of $T_o X$. %sub-vector space of $T_{[o]} S_{\IQbar}$ with $\psi(z) \in V \otimes \CC$. 
Then $V \supseteq V'$ for some bi-$\IQbar$-subspace $V' \subseteq T_o X$ which contains $z$.

If moreover $\mathrm{MT}(o)$ is a maximal torus, we can take $V' = T_{[o]} S'_{\IQbar}$ for some Shimura subvariety $S'$ of $S$ which contains $[z]$ and $[o]$.
\end{conj}

An equivalent but more precise statement of this conjecture is presented as Conjecture~\ref{ConjAnalyticSubspace}. 
Observe that if $u(z) \not= u(o)$, then in the conclusion of the conjecture we have $\dim V' > 0$ and $\dim S' > 0$. %The definition of sufficiently general special point is given in $\mathsection$\ref{SubsectionSufficientlyGeneral}. In the case $S = \AAA_g$,  a special point is sufficiently general   if and only if it parametrizes a CM abelian variety with no square factors.
 Note that the statement of the conjecture depends on the choice of a  special point of $S$ (up to Hecke operations).

It is known that WAST implies that there exists no non-trivial $\oQ$-linear relations between the holomorphic periods $\theta_1,\dots,\theta_g$ of a  CM abelian variety $A$ of dimension $g$ if $A$ has no square factors. In the same way we show that our hyperbolic counterpart HASC implies the desired result for quadratic relations. More precisely we prove:

\begin{prop}[{Proposition~\ref{PropConsequenceAnalyticSubspaceQuadraticRelations}}]
Assume that Conjecture~\ref{HWASP} holds true for $S = \mathbb{A}_g$ and the point $[o] \in \mathbb{A}_g(\IQbar)$ parametrizing $A$. Then Question~\ref{QuestionNoQuadraticRelationIntro} has an affirmative answer for $A$.
\end{prop}

We think that a proof of the Conjecture \ref{HWASP} could be quite a challenge as the classical version of WASP for commutative algebraic  groups is a major achievement. Our hope is that it could be still more tractable than  Grothendieck's Period Conjecture. 

\vskip 0.3em

\subsection{Organization of the article}
In $\mathsection$\ref{SectionBiIQbar}, we define split bi-$\IQbar$-structures and discuss the cases of algebraic tori $\mathsection$\ref{SubsectionBiIQbarAlgTori} and CM abelian varieties $\mathsection$\ref{SubsectionExampleCMAV}. We also reformulate WAST in these two cases in this language. In $\mathsection$\ref{SectionProductCMPeriodsFirstDis}, we explain why WAST applied to $A^{\natural}$, the universal vector extension of $A$, yields the transcendence of $\theta_j\theta_{j'}/\pi$ with the $\theta_j$'s the holomorphic periods of a CM abelian variety $A$. We also restate Question~\ref{QuestionNoQuadraticRelationIntro} as a motivating question for the current article.

After this prologue, we start to build up our framework and relate it to Question~\ref{QuestionNoQuadraticRelationIntro} in $\mathsection$\ref{SectionBiIQbarShimura}--\ref{SectionAnalyticSubspaceConj}.

The bi-$\IQbar$-structure on the tangent space of a Shimura variety at a special point is defined in 
$\mathsection$\ref{SectionBiIQbarShimura}, and is proved to be split in $\mathsection$\ref{SectionBiIQbarShimura2}. In $\mathsection$\ref{SectionDeRhamBettiKodaira} we discuss on the family version of the de Rham--Betti comparison and the Kodaira--Spencer map, which will be used compute the periods of the split bi-$\IQbar$-structure for $\mathbb{A}_g$ in $\mathsection$\ref{SectionPeriodSiegel}. In $\mathsection$\ref{SectionShimuraCurveHilbert}, we prove the conjecture of Lang (Proposition~\ref{PropLangQuestionIntro}) and compute the periods for Hilbert modular varieties.

Then after some preliminary discussion on CM points in $\mathsection$\ref{SectionDiscussionCM}, we propose  in $\mathsection$\ref{SectionAnalyticSubspaceConj} our hyperbolic analytic subspace conjecture (Conjecture~\ref{HWASP}), and explains how this conjecture applied to $\mathbb{A}_g$ gives a positive answer to Question~\ref{QuestionNoQuadraticRelationIntro}.

The last section $\mathsection$\ref{SectionGroPeriodConj} explains some consequences of Grothendieck's Period Conjecture on the algebraic relations between $\pi$ and holomorphic periods of CM abelian varieties. This section is rather independent of the rest of the article.

\subsection*{Acknowledgments}
Ziyang Gao has received
fundings from the European Research Council (ERC) under the
European Union's Horizon 2020 research and innovation programme (grant
agreement n$^\circ$ 945714). Both ZG and AY would like to thank the IHES for its hospitality during the preparation of this paper. Both ZG and EU would like to thank the MSRI for its hospitality during the work. We would like 
to thank Yves Andr\'{e} and Daniel Bertrand for various discussions on transcendental number theory, especially on Shimura relations, W\"{u}stholz's analytic subgroup theorem and Grothendieck's period conjecture.  
We would like to thank Martin Orr for discussion on problems related to isogeny and Hecke orbits. We would like to thank Gregorio Baldi for pointing out the reference \cite{CW} and for his comments on a previous version of the preprint. Finally, we would like to thank the reviewer for their constructive comments and careful reading of the manuscript, which helped us to improve the presentation of the paper.

\section{Bi-$\IQbar$-spaces and preliminary examples}\label{SectionBiIQbar}
The goal of this section is to define the notion of a bi-$\IQbar$-structure on a complex vector space and to discuss  examples given by algebraic tori and CM abelian varieties.

\subsection{Bi-$\oQ$-structures on a complex vector space}

Let $V$ be a complex vector space of finite dimension $n$. A $\oQ$-structure on $V$ is a $\oQ$-vector space $W$ generated by a basis $\cB=\{e_1,\dots,e_n\}$ of $V$. If $W$ is $\oQ$-structure on $V$, then $W\otimes_{\oQ}\CC\simeq V$. A complex vector subspace $V'$ of $V$ is said to be \textit{rational for the $\oQ$-structure $W$} if $V'$ admits a basis consisting 
of elements in $W$. In general if $V'$ is a vector subspace of $V$ then $V'\cap W$ is a $\oQ$-vector subspace of $W$ and 
$$
\dim_{\oQ} (V'\cap W)\leq \dim_{\CC} (V') 
$$
with equality if and only if $V'$ is $\oQ$-rational for the $\oQ$-structure $W$.

\begin{defini}
Let $V$ be a complex vector space. A {\bf bi-$\oQ$-structure} $(W_1, W_2)$ on $V$  is the data given by two $\oQ$-structures 
$W_1$ and $W_2$ on $V$. A complex vector subspace $V'$ of $V$ is said to be bi-$\oQ$-rational (or simply bi-$\oQ$)
 if $V'$ is rational for both $W_1$ and $W_2$. 
 
 If $V'$ is a bi-$\oQ$-subspace of dimension $1$, then $V'$ determines a complex number $\theta=\theta(V',W_1,W_2)$, well defined up to multiplication  by an element of $\oQ^{*}$, in the following way. Let $e_1\in W_1$ and $e_2\in W_2$ be some  bases of $V'$ defining the corresponding $\oQ$-structures. 
 Then $e_2=\theta e_1$ for some $\theta\in \CC$.  Such a $\theta$ will be called a {\bf period} of $V'$ for $( W_1, W_2)$. 
\end{defini}
 This terminology of \textit{periods} is motivated by the examples of the exponential function and of the map uniformizing CM abelian varieties. We will explain this in later sections.

To ease notation, we use the following convention. 
\begin{notation}
Let $z_1$ and $z_2$ be complex numbers. We will write $z_1\simeq z_2$ if there exists $\alpha \in \oQ^*$ such that $z_2=\alpha z_1$.
\end{notation}

\begin{defini}
A    bi-$\oQ$-structure $(W_1, W_2)$ on $V$ is said to be {\bf split} if $V$ is a direct sum of bi-{$\oQ$}-subspaces of dimension 1. In this situation
$$
V=\bigoplus_j \CC e_j=\bigoplus_j \CC f_j
$$
for a $\oQ$-basis $\{e_1,\dots,e_n\}$ of $W_1$ and a $\oQ$-basis $\{f_1,\dots,f_n\}$ of $W_2$ such that for each $j$, there exists $\theta_j\in \CC^*$ with $f_j=\theta_j e_j$. We say that $\{\theta_1,\dots,\theta_n\}$ is \textbf{a complete set of periods} of $V$ for $( W_1, W_2)$. 
\end{defini}

Let $(W_1,W_2)$ be a split bi-$\oQ$-structure on $V$ and let $\theta_1,\dots,\theta_n$ be the associated periods. These periods can be regrouped in the following way: Set $(J_1,\dots,J_r)$ to be the partition of $\{1,\dots,n\}$ such that $\theta_j\simeq \theta_{j'}$ if and only if $j, j' \in J_s$ for some  $s\in \{1,\dots,r\}$. For each $J_s$, denote by $\theta_{J_s}$ the period $\theta_j$ for any $j \in J_s$; it is well-defined up to multiplication by a number in $\IQbar^*$.

\begin{defini} 
Let $s\in \{1,\dots,r\}$. The sub-vector space  
$$
V_s:=\bigoplus_{j\in J_s} \CC e_j = \bigoplus_{j\in J_s} \CC f_j
$$
 of $V$ is called the isotypic subspace of $V$ associated to the period $\theta_{J_s}$.
\end{defini}

The following simple proposition fully describes the  bi-$\oQ$-subspaces of $V$.

\begin{prop}\label{PropBiQbarEasyProp}
The following  holds true:
\begin{enumerate}
\item[(i)] There is a decomposition 
$
V=\oplus_{s} V_s$ 
into a direct sum of isotypic subspaces.

\item[(ii)] Let $V'$ be a subspace of some isotypic subspace $V_s$. Then $V'$ is rational for $W_1$ if and only if it is rational for $W_2$. Thus $V'$ is  bi-$\oQ$ if and only if it is rational for one of two $\IQbar$-structures.

\item[(iii)] Any  bi-$\oQ$-sub-space $V'$ of $V$ has a decomposition in isotypic components
$
V'=\bigoplus V'_s
$, 
where $V'_s :=V_s\cap V'$ is a bi-$\oQ$-subpace of $V_s$.
\end{enumerate}
\end{prop}

An equivalent way of stating the proposition is to consider the category $\mathcal{C}_{\oQ,\oQ}^{split}$ of finite dimensional complex vector spaces endowed with a split bi-$\oQ$-structure. The morphism in $\mathcal{C}_{\oQ,\oQ}^{split}$ are the linear maps respecting the two $\oQ$-structures.
 Then $\mathcal{C}_{\oQ,\oQ}^{split}$ is a semi-simple category and the simple objects are the one dimensional vector spaces $(\CC; \theta)$ with period $\theta\in \CC^*/\oQ^*$. For any $n$-tuples of non zero complex numbers $(\theta_1,\dots,\theta_n)$ (modulo $\oQ^*$), denote by $(\CC^n ; \theta_1,\dots,\theta_n)$ the object in $\mathcal{C}_{\oQ,\oQ}^{split}$ of dimension $n$ with periods $(\theta_1,\dots,\theta_n)$. 
  The following proposition summarizes the two extreme cases.% for $(\CC^n,\theta_1,\dots,\theta_n)$.%, \textit{i.e.} when all the periods coincide up to $\IQbar^*$ and when these periods are $2$-by-$2$ distinct up to $\IQbar^*$.
\begin{prop}\label{PropTwoExtremeSplitBiQbar}
Consider an object $(\CC^n ; \theta_1,\dots,\theta_n)$ in $\mathcal{C}_{\oQ,\oQ}^{split}$.
\begin{enumerate}
\item[(i)] If $\theta_1\cong \theta_2 \cong \cdots \cong \theta_n$, then 
%Let $\theta\in \CC^*$ and $(\CC^n,(\theta,\dots,\theta))$ be the associated bi-$\oQ$-structure. Then 
any complex vector subspace $V'$ of $\CC^n$ which is $\oQ$-rational for one of the two $\oQ$-structures is bi-$\oQ$.

\item[(ii)] If $\theta_j\theta_{j'}^{-1}\notin \oQ^*$ for all $j \neq j'$, then there are only finitely many bi-$\oQ$-subspaces of  $(\CC^n ; \theta_1,\dots,\theta_n)$.
\end{enumerate}
\end{prop}
In case (i), $\CC^n$ is isotypic. 
In case (ii), the isotypic subspaces are all of dimension $1$, which we denote by $V_1,\ldots,V_n$. Then the bi-$\IQbar$-subspaces are precisely the $\bigoplus_{j \in J} V_j$'s, where $J$ runs over all subsets of $\{1,\ldots,n\}$. The proof of the propositions \ref{PropBiQbarEasyProp} and  \ref{PropTwoExtremeSplitBiQbar} are simple linear algebra exercises and are left to the reader.

We close this subsection by defining the \textit{Tate twist} in $\mathcal{C}_{\oQ,\oQ}^{split}$.
\begin{defini}\label{DefnTateTwist}
The \textbf{Tate twist} in $\mathcal{C}_{\oQ,\oQ}^{split}$, denoted by $\mathfrak{L}(1)$, is defined to be $(\CC; 2\pi i)$.
\end{defini}
As we will see in the next subsection, the Tate twist arises naturally from the bi-$\IQbar$-structure associated with algebraic tori.

\subsection{Example: Algebraic tori}\label{SubsectionBiIQbarAlgTori}
We look at two concrete examples of bi-$\IQbar$-structures on a complex space in the current and next subsections. The current one is in line with Proposition~\ref{PropTwoExtremeSplitBiQbar}.(i). The next one is more general, and in some cases in line with Proposition~\ref{PropTwoExtremeSplitBiQbar}.(ii).

Consider the algebraic torus $\GmF{\IQbar}^n$ defined over $\IQbar$ and the uniformization in the category of complex spaces
\[
u \colon \CC^n \to \Gm^n(\CC) = (\CC^*)^n, \quad (z_1,\ldots,z_n) \mapsto (e^{2\pi i z_1},\ldots, e^{2\pi i z_n}).
\]
We can endow $\CC^n$ with two  $\IQbar$-structures. The \textit{geometric $\IQbar$-structure} on $\CC^n$ is the one induced by the natural inclusion $\IQbar \subseteq \CC$ (\textit{i.e.} $\CC^n = \IQbar^n \otimes \CC$), and the \textit{arithmetic $\IQbar$-structure} on $\CC^n$ is the one induced by $\CC^n \cong \Lie \ \GmF{\IQbar}^n\otimes\CC$. 
A natural $\oQ$-basis for the geometric structure  is the canonical basis $(e_1,\dots,e_n)$ of $\CC^n$. Then $(2i\pi e_1,\dots,2i\pi e_n)$ is a $\oQ$-basis of the arithmetic structure.
Thus they define a bi-$\IQbar$-structure on $\CC^n$. This bi-$\IQbar$-structure is split and all the periods are $2\pi i$. This easily follows from the fact that  the derivative of $u$ defines the isomorphism
\[
\psi:= \mathrm{d}u = \mathrm{diag}(2\pi i,\ldots,2\pi i) \colon \CC^n = \IQbar^n\otimes\CC \to \Lie \ \GmF{\IQbar}^n \otimes \CC \cong \CC^n.%{We write the left hand side as $\IQbar^n\otimes\CC$ to emphasize that we are looking at the geometric $\IQbar$-structure on $\CC^n$ on this side.}
\]
Hence this bi-$\IQbar$-structure on $\CC^n$ associated with $\GmF{\IQbar}^n$ is the $n$-copy of the Tate twist $\mathfrak{L}(1)^{\oplus n}$.

As a consequence, a subspace $V'$ of $\CC^n$ is rational for the geometric $\IQbar$-structure if and only if it is rational for the arithmetic $\IQbar$-structure.

On the other hand, the arithmetic and the geometric $\IQbar$-structures on $\CC$ are different. Indeed, let $x \in \Gm^n(\IQbar)$, and fix some $\tilde{x} \in u^{-1}(x)$. By Gelfond--Schneider, $\tilde{x}$ is a $\IQbar$-point in $\CC^n$ with respect to the geometric $\IQbar$-structure if and only if $x$ is a torsion point in $\Gm^n(\IQbar)$. On the other hand by Lindemann, $\tilde{x}$ is a $\IQbar$-point in $\CC^n$ with respect to the arithmetic $\IQbar$-structure if and only if $\tilde{x}= 0$.%; see \cite[]{UllmoZP}.

\medskip

We close this subsection by a reformulation of WAST, which in this case has its roots in Baker's Theorem on $\oQ$-linear independence of $\QQ$-linearly independent logarithms of algebraic numbers.
%W\"{u}stholz's Analytic Subgroup Theorem
  The theorem concerns the 
 base change of the exponential map
\[
\exp \colon \Lie \ \GmF{\IQbar}^n \to \GmF{\IQbar}^n
\]
to $\CC$, which is precisely $u \circ \psi^{-1}$. % The \textit{arithmetic $\IQbar$-structrue} on $\CC^n$ is given by $\CC^n \cong \Lie \ \GmF{\CC}^n =  (\Lie \GmF{\IQbar}^n)\otimes \CC$.
%Here is a rephrase of the theorem.
\begin{teo}[W\"{u}stholz]\label{ThmWustholzAlgTori}
Let $z \in \CC^n$ such that $u(z) \in \Gm^n(\IQbar)$. %Set $V_z$ to be the smallest 
Let $V$ be a $\IQbar$-sub-vector space of $\Lie\   \GmF{\IQbar}^n$ such that $\psi(z) \in V\otimes \CC$. Then $V \supseteq \Lie   H$ for some algebraic subgroup $H$ of $\GmF{\IQbar}^n$ which contains $u(z)$.
\end{teo}

\subsection{Example: CM abelian varieties}\label{SubsectionExampleCMAV}
In this subsection we take a first look at CM abelian varieties defined over $\IQbar$ and explain how its uniformizing space can be endowed with a bi-$\IQbar$-structure which, under some mild assumptions, fits into Proposition~\ref{PropTwoExtremeSplitBiQbar}.(ii). More detailed discussions on CM abelian varieties will be given in later sections.

Let $A$ be a complex CM abelian variety of dimension $g$. Then $A$ is the extension to $\CC$ of an abelian variety $A_{\oQ}$ defined over $\IQbar$. The universal covering of $A^{\mathrm{an}}$ is $\CC^g$. We endow $\CC^g=\Lie (A_{\oQ})\otimes \CC$ with a bi-$\IQbar$-structure:
\begin{itemize}
\item The \textit{arithmetic $\IQbar$-structure} on $\CC^g$ is defined to be the one given by 
$$
\CC^g \cong  \Lie A_{\IQbar} \otimes \CC.
$$
\item The \textit{geometric $\IQbar$-structure} on $\CC^g$ is slightly more complicated. It is known that there exists a matrix $\tau \in \mathrm{Mat}_{g\times g}(\IQbar)$ such that $A(\CC) \cong \CC^g / (\ZZ^g + \tau \ZZ^g)$. Hence $\IQbar\cdot (\ZZ^g + \tau \ZZ^g) \subseteq \CC^g$ gives a $\IQbar$-structure on $\CC^g$, and we call it  the geometric $\IQbar$-structure.
\end{itemize}

Let $\theta_1,\ldots,\theta_g$ be the holomorphic periods of $A$  as defined in Convention~\ref{SubsectionConvention}.

\begin{prop}\label{PropBiQbarCMAVSplit}
The bi-$\oQ$-structure on $\CC^g\simeq  \Lie (A_{\oQ})\otimes \CC$ defined above is split, and is isomorphic to $(\CC^g; \theta_1,\ldots,\theta_g)$.
\end{prop}
\begin{proof}
%Without loss of generality, we may assume  that 
We start with the case where $A$ is simple. Then $\mathrm{End}^0(A):=\mathrm{End}(A)\otimes_{\ZZ}\QQ$ is a CM field $E$ of degree $2g$.
Let $\Phi$ be the CM type of $A$.
% We will identify $\CC^g = \prod_{\sigma \in \Phi}\CC = \Lie A_{\CC}$.

For the arithmetic $\IQbar$-structure, $E$ acts on $\Lie A_{\CC}$  through the embeddings $\sigma \colon E\rightarrow \CC$ belonging to $\Phi$. For each such $\sigma$, its eigenspace $(\Lie A_{\CC})_{\sigma}$ has dimension $1$. Moreover, as $\sigma$ factors through $\IQbar \subseteq \CC$, its eigenspace $(\Lie A_{\CC})_{\sigma}$ descends to $\IQbar$, \textit{i.e.} there exists $e_{\sigma} \in \Lie A_{\IQbar}$ such that the eigenspace for $\sigma$ is $\CC e_{\sigma}$. Now the action of $E$ on $\Lie A_{\CC}$ induces
\[
\CC^g = \Lie A_{\CC} = \oplus_{\sigma \in \Phi} \CC e_\sigma.
\]

For the geometric $\IQbar$-structure, denote for simplicity by $\Lambda:= \ZZ^g + \tau \ZZ^g$. Then the action of $\mathrm{End}(A)$ on $A(\CC) \cong \CC^g / \Lambda$ induces an action of $\mathrm{End}(A)$ on $\Lambda \subseteq \CC^g$. Thus tensoring $\QQ$, we obtain an action of $E$ on $\QQ \Lambda \subseteq \CC^g$ and hence on $\IQbar  \Lambda = \IQbar^g \subseteq \CC^g$. If we identify $\CC^g$ with $\prod_{\sigma \in \Phi}\CC$, then $\Lambda$ equals $\Phi(R_E)$ for some order $R_E$ in $\cO_E$ under $\Phi \colon E\otimes_{\QQ}\RR  \cong \prod_{\sigma \in \Phi}\CC$. In particular, 
the action of $E$ on $\CC^g$ is via the embeddings $\sigma \colon E \rightarrow \IQbar$ belonging to $\Phi$. Hence in the previous paragraph, we have $\CC^g = \oplus_{\sigma \in \Phi} \CC f_\sigma$ with $f_\sigma \in \IQbar\Lambda$ an eigenvector for $\sigma$.

Now that we have two actions of $E$ on $\CC^g$ via the embeddings $\sigma \colon E \rightarrow \IQbar$, the eigenspaces for $\sigma$ coincide. Hence $\CC e_\sigma = \CC f_\sigma$. Hence $f_\sigma = \theta_\sigma e_\sigma$ for some $\theta_\sigma \in \CC^*$. So in the case where $A$ is simple, the bi-$\IQbar$-structure on $\CC^g$ associated with $A$ is split.

For a general CM abelian variety $A$ of dimension $g$ defined over $\IQbar$, $A$ is isogenous to the product of simple CM abelian varieties. We use $\Phi$ to denote the (disjoint) union of all the CM types of these simple CM abelian varieties. Then we can also find $\{e_\sigma\}_{\sigma\in \Phi}$ a $\IQbar$-basis of the arithmetic $\IQbar$-structure on $\CC^g$ and $\{f_\sigma\}_{\sigma \in \Phi}$ a $\IQbar$-basis of the geometric $\IQbar$-structure on $\CC^g$. Then $f_{\sigma} = \theta_\sigma e_\sigma$ for some $\theta_\sigma \in \CC^*$. Notice that $\theta_\sigma$ is uniquely determined by $\sigma$ up to multiplication by a element in $\IQbar^*$. Hence the bi-$\IQbar$-structure on $\CC^g$ associated with $A$ is split.

It remains to show that the $\theta_\sigma$'s are precisely the holomorphic periods of $A$.

It is known that $\Omega_A$ (the $\IQbar$-vector space consisting of invariant holomorphic $1$-forms) is the dual of $\Lie A_{\IQbar}$. Set $\omega_\sigma := e_\sigma^\vee$ for each $\sigma \in \Phi$. Then $\{\omega_{\sigma}\}_{\sigma\in \Phi}$ is an $\mathrm{End}^0(A)$-eigenbasis of $\Omega_A$. Now each $f_\sigma \in \IQbar\Lambda$ can be written as $f_\sigma = \sum_{j=1}^{2g} a_{\sigma,j} \gamma_j$ 
for some $a_{\sigma,j} \in \IQbar$ and $\gamma_j \in \Lambda$. Thus
\[
1 = e_{\sigma}(\omega_\sigma) = \theta_{\sigma}^{-1} \cdot f_\sigma(\omega_\sigma) = \theta_\sigma^{-1}\cdot \sum_{j=1}^{2g} a_{\sigma,j}  \int_{\gamma_j}\omega_{\sigma} \cong \theta_\sigma^{-1}\cdot \sum_{j=1}^{2g} \int_{\gamma_j}\omega_{\sigma}.
\]
But up to $\IQbar^*$, $\int_{\gamma}\omega_{\sigma}$ is independent of the choice of $\gamma$ by Shimura \cite[Rmk.3.4]{ShimuraPeriod} (see the discussion before the convention 1.1). So $\theta_\sigma \cong \int_{\gamma}\omega_{\sigma}$. We are done.
\end{proof}

\medskip

The following result is a consequence of W\"{u}stholz's Analytic Subgroup Theorem applied to $\mathbb{G}_{\mathrm{a}}\times A$ and to $A$ respectively; see \cite[Prop.1.2, Thm.1.4]{Tretkoff}.

\begin{prop}[W\"{u}stholz]\label{PropLinearRelationPeriodCMAV}
\begin{enumerate}
\item[(i)] The periods $\theta_j$ are transcendental.
\item[(ii)] Assume $A$ has no square factors up to isogeny. Then $\dim_{\IQbar}\Sigma_{j=1}^g  \IQbar \theta_{j} = g$.
\end{enumerate}
\end{prop}

On the other hand, isogenous abelian varieties have the same periods up to $\IQbar^*$. 
From part (ii) of this proposition, we can easily deduce that the bi-$\IQbar$-structure on $\CC^g$ defined at the beginning of this subsection  satisfies the properties of  the part (ii) of Proposition~\ref{PropTwoExtremeSplitBiQbar}, if $A$ has no square factors. This holds true for example when $A$ is simple.

\medskip

We close this subsection by a reformulation of W\"{u}stholz's Analytic Subgroup Theorem for this case. Consider the uniformization
\[
u \colon \CC^g \to A(\CC) = \CC^g / \Lambda.
\]
It can be shown that: for any $x \in \CC^g$ which is $\IQbar$ in the  geometric $\IQbar$-structure, we have $u(x) \in A(\IQbar)$ if and only if $u(x)$ is a torsion point. We are thus viewing $\CC^g$ as $\IQbar\Lambda \otimes \CC$.

The map $u$ induces an isomorphism
\[
\psi := \mathrm{d}u \colon \CC^g \cong \IQbar\Lambda \otimes \CC \to \Lie \ A_{\IQbar} \otimes \CC \cong \CC^g,
\]
which, under the bases $\{f_{\sigma}\}_{\sigma\in \Phi}$ and $\{e_{\sigma}\}_{\sigma\in \Phi}$ from the proof of Proposition~\ref{PropBiQbarCMAVSplit}, is the diagonal matrix $\mathrm{diag}(\theta_{\sigma})_{\sigma \in \Phi}$.

The base change of the exponential map
\[
\exp \colon \Lie A \to A
\]
to $\CC$ is  $u \circ \psi^{-1}$. WAST %W\"{u}stholz's Analytic Subgroup Theorem
 in this case can be rephrased to be:
\begin{teo}[W\"{u}stholz]\label{ThmWustholzCMAbVar}
Let $z \in \CC^g$ be such that $u(z) \in A(\IQbar)$. 
%Set $V_z$ to be the smallest 
Let $V$ be a sub-vector space of $\Lie \ A_{\IQbar}$ such that $\psi(z) \in V\otimes \CC$. Then $V \supseteq \Lie B$ for some abelian subvariety  $B$ of $A$ which contains $u(z)$.
\end{teo}

\section{A first discussion on products of periods of CM abelian varieties}\label{SectionProductCMPeriodsFirstDis}
In $\mathsection$\ref{SubsectionExampleCMAV}, or more precisely Proposition~\ref{PropLinearRelationPeriodCMAV}, we have seen the transcendence and the linear relations between the holomorphic periods of CM abelian varieties, both as consequences of WAST. In this section, we attempt to understand and ask questions about the quadratic relations.

%We retain the notation from $\mathsection$\ref{SubsectionExampleCMAV} in a simpler way. 
Let $A$ be a CM abelian variety of dimension $g$ defined over $\IQbar$. Let $\theta_1,\ldots,\theta_g$ be its holomorphic periods as defined in Convention~\ref{SubsectionConvention}. Then up to $\IQbar^*$, each $\theta_j$ is $\int_{\gamma}\omega_j$ for an $\mathrm{End}(A)\otimes\QQ$-eigenbasis $\{\omega_1,\ldots,\omega_g\}$ of $\Omega_A$ (with $\gamma\in H_1(A,\ZZ)$).

WAST, when applied to $A^\natural$ (the universal vector extension of $A$) implies that the quotient of a holomorphic period by an anti-holomorphic period is transcendental. In the case of CM abelian varieties, this yields a transcendence result on the product of two holomorphic periods. This result is well-known to experts, but we will include its proof in the current paper for the convenience of readers.

\begin{cor}\label{PropProductPeriodTranscendental}
Each $\theta_j\theta_{j'}/\pi$ is a transcendental number for $ j, j' \in \{1,\ldots,g\}$.
\end{cor}

Before proceeding to its proof, 
at this stage a natural question to ask is what are the possible non-trivial $\IQbar$-linear relations among $\theta_j\theta_{j'}$ for all $1\le j \le j' \le g$, or equivalently what are the possible non-trivial $\IQbar$-quadratic relations between the $\theta_j$'s. 
Unlike the transcendence of $\theta_j \theta_{j'}/\pi$, this question is still widely open; see Question~\ref{QuestionNoQuadraticRelationIntro} for the prediction. 
It turns out that this question is closely related to bi-$\IQbar$-structures arising from Shimura varieties, which we will discuss from $\mathsection$\ref{SectionBiIQbarShimura}.

\medskip 
In the rest of this section we prove Corollary~\ref{PropProductPeriodTranscendental}. 
It follows immediately from the following proposition.

\begin{prop}\label{PropQbarDimPeriodHoloAntiholo}
Assume that  $A$ has no square factor up to isogeny. Then
\begin{equation}\label{EqQbarDimPeriodHoloAntiholo}
\dim_{\IQbar} \left(\IQbar 2\pi i + \IQbar \theta_1 + \cdots + \IQbar \theta_g + \IQbar \frac{2\pi i}{\theta_1} + \cdots + \IQbar \frac{2\pi i}{\theta_g}\right) = 2g+1.
\end{equation}
\end{prop}

\begin{proof}[Proof of Proposition~\ref{PropQbarDimPeriodHoloAntiholo} implying Corollary~\ref{PropProductPeriodTranscendental}]  Using equation \eqref{EqQbarDimPeriodHoloAntiholo} we conclude that $\theta_j$ and $2\pi i / \theta_{j'}$ are linearly independent over $\IQbar$ for each $j, j' \in\{1,\ldots,g\}$. Dividing both numbers by $\pi / \theta_{j'}$, we get that $\theta_j\theta_{j'}/\pi$ and $2i$ are linearly independent over $\IQbar$. But $2i \in \IQbar$. Hence $\theta_j\theta_{j'}/\pi$ is transcendental.
\end{proof}

Now let us prove Proposition~\ref{PropQbarDimPeriodHoloAntiholo}.

\begin{proof}[Proof of Proposition~\ref{PropQbarDimPeriodHoloAntiholo}] 
We may assume $A = A_1\times \cdots \times A_n$ is the product of simple CM abelian varieties which are $2$-by-$2$ non-isogenous. We start with some preparation.

The key idea is to apply W\"{u}stholz's Analytic Subgroup Theorem to $A^{\natural}$, the universal vector extension of $A$. It fits into the short exact sequence
\[
0 \to \Omega_{A^{\mathrm{t}}} \to A^\natural \to A \to 0.
\]
The natural projection $A^\natural \to A$ induces canonical isomorphisms
\[
H^1_{\mathrm{dR}}(A^\natural) = H^1_{\mathrm{dR}}(A), \quad H^1(A^\natural,\ZZ) = H^1(A,\ZZ).
\]
As a complex variety, $A^\natural(\CC) = H_1(A^\natural,\CC) / H_1(A^\natural, \ZZ)$. The Lie algebra $\Lie A^\natural$ is  $H^1_{\mathrm{dR}}(A^\natural)^\vee$.

Let $\eta_j$ be the complex conjugation of $\omega_j$. Then $\{\omega_1,\ldots,\omega_g,\eta_1,\ldots,\eta_g\}$ is a $\IQbar$-basis of $H^1_{\mathrm{dR}}(A^\natural) = H^1_{\mathrm{dR}}(A)$. Take any basis $\{\gamma_1^\vee,\ldots,\gamma_{2g}^\vee\}$ of $H^1(A^{\natural},\ZZ) = H^1(A,\ZZ)$. Then the Hodge--de Rham comparison $ H^1_{\mathrm{dR}}(A^{\natural}) \xrightarrow{\sim} H^1(A^{\natural},\CC)$ under these $\IQbar$-bases is 
$$
\beta = \begin{bmatrix} \left(\int_{\gamma_k}\omega_j \right)_{1\le  k \le 2g, ~ 1\le j \le g} & \left(\int_{\gamma_k}\eta_j \right)_{1\le  k \le 2g, ~ 1\le j \le g} \end{bmatrix}.
$$
 For each $j \in \{1,\ldots,g\}$, we have $\int_{\gamma_k}\omega_j \cong \theta_j$. Since $\omega_1,\ldots,\omega_g,\eta_1,\ldots,\eta_g$ are eigenforms for the CM action, the reciprocity law for the differential forms of the 1st and the 2nd kinds implies that $\int_{\gamma_k} \eta_j \cong 2\pi i / \theta_j$; see \cite[pp.36, equation (3)]{BertrandReciprocityDiffForm} for more details. Hence this $2g\times 2g$-matrix $\beta$ is similar to the diagonal matrix $\diag(\theta_1,\ldots, \theta_g, 2\pi i/\theta_1,\ldots,2\pi i/\theta_g)$ over $\IQbar$. For more details about this computation we refer to the last sentence of Remark~\ref{RemarkCMAbVarDiagonal} and above.

Denote for simplicity by $g_j := \dim A_j$, and $\theta^{(j)}_1,\ldots, \theta^{(j)}_{g_j}$ the holomorphic periods of $A_j$. 
Write $\exp_j \colon \Lie A_j^{\natural} \to A_j^{\natural}(\CC)$ for the exponential map, which is the composite of $\Lie A_j^\natural = H^1_{\mathrm{dR}}(A_j^\natural)^\vee \cong H_1(A_j^\natural,\CC)$ and the quotient $H_1(A_j^\natural,\CC) \rightarrow H_1(A_j^\natural,\CC) / H_1(A_j^\natural,\ZZ)$. By the conclusion of the last paragraph, under suitable $\IQbar$-bases, the isomorphism $H^1_{\mathrm{dR}}(A_j^\natural)^\vee \cong H_1(A_j^\natural,\CC)$ is represented by the diagonal matrix $\diag(\theta^{(j)}_1,\ldots, \theta^{(j)}_{g_j}, 2\pi i/\theta^{(j)}_1,\ldots,2\pi i/\theta^{(j)}_{g_j})$. Thus there exist $\lambda_1^{(j)},\ldots,\lambda_{g_j}^{(j)}, \mu_1^{(j)}, \ldots, \mu_{g_j}^{(j)} \in \IQbar^*$ such that $\exp_j(\mathbf{x}^{(j)}) = 0$ for the point 
 $\mathbf{x}^{(j)} := (\lambda_1^{(j)} \theta^{(j)}_1,\ldots, \lambda_{g_j}^{(j)} \theta^{(j)}_{g_j}, \mu_1^{(j)} \cdot 2\pi i/\theta^{(j)}_1,\ldots, \mu_{g_j}^{(j)} \cdot 2\pi i/\theta^{(j)}_{g_j}) \in \Lie A_j^{\natural}$.% \Ginlinecomment{Finish this!!!}

\vskip 0.3em

Set $\mathbf{x}:=
%(\theta_1,\ldots,\theta_g, 2\pi i/\theta_1,\ldots, 2\pi i/\theta_g)
(2\pi i, \mathbf{x}^{(1)},\ldots,\mathbf{x}^{(n)}) 
 \in \Lie (\Gm \times A^\natural)$ and  $\exp = (e^{\cdot}, \exp_1,\ldots,\exp_n) \colon \Lie (\Gm \times A^\natural) \to \Gm \times A^\natural$.
 Then $\exp(\mathbf{x})$ is the origin of $\Gm \times A^\natural$.

 Let $\ell \gg 1$ be a prime number such that $\exp_j(\mathbf{x}^{(j)}/\ell)$ is a non-zero torsion point in $A_j^\natural$ for each $j$. Let $V$ be the smallest $\IQbar$-subspace of $\Lie(\Gm\times A^\natural)$ such that $\mathbf{x}/\ell \in V_{\CC}$. Then WAST applied to $\exp \colon \Lie (\Gm \times A^\natural) \to \Gm\times A^\natural$, and the minimality assumption made on $V$,  ensure that $V = \Lie H$ for some (algebraic subgroup) $H$ of $\Gm \times A^\natural$. Since $\exp(\mathbf{x}/\ell)$ is non-zero, we have $\dim H > 0$.
 
 Write $\pi \colon \Gm \times A^\natural \rightarrow \Gm \times A$. As $\pi(H)$ is an irreducible subgroup of $\Gm \times A$ and $A$ is the product of $2$-by-$2$ non-isogenous simple  abelian varieties, we have that $\pi(H) = \Gm \times  A_1\times \cdots \times A_m \times \{0\}$ up to reordering. 

We claim that $m = n$ and $\Lie (H) = \Lie (\Gm \times A^\natural)$. Indeed, set $B:= A_1\times \cdots \times A_m$ and $B' := A_{m+1}\times \cdots \times A_n$. 
Then $A = B \times B'$. 
%we have  $B = A_1\times \cdots \times A_m$ up to reordering for some $m \in \{0,1,\ldots,n\}$. 
Since $\Gm \times B \times \{0\} =\pi(H)$, we have $\Gm \times B^\natural \times \{0\} \subseteq H$. Hence the group $H$ equals $\Gm \times B^\natural \times V'$ for some vector subgroup $V'$ of $B'$. So $\exp(\mathbf{x}/\ell) \in (\Gm \times B^\natural \times V')(\IQbar) \subseteq (\Gm \times A^\natural)(\IQbar)$. But $\exp(\mathbf{x}/\ell)$ is a torsion point, so $\exp(\mathbf{x}/\ell) \in (\Gm\times B^\natural)(\IQbar) \times \{0\}$. By minimality of $V$, we then have $\pi(H) = \Gm \times B^\natural \times \{0\}$ and $V = \Lie (\Gm \times B^\natural) \times \{0\} = \Lie \Gm \times \Lie A_1^\natural\times\cdots \times \Lie A_m^{\natural} \times \{0\}$. If $m < n$, then $\mathbf{x}^{(m+1)}/\ell$ is not $0$, and hence $\mathbf{x}/\ell = (2\pi i, \mathbf{x}^{(1)},\ldots,\mathbf{x}^{(n)})/\ell \not\in  \Lie \Gm \times \Lie A_1^\natural\times\cdots \times \Lie A_m^{\natural} \times \{0\} = V$. This contradicts the choice of $V$. So we must have $m = n$. Therefore $B = \Gm\times A$ and $V = \Lie (\Gm \times A^\natural)$.

Assume there exists $(c,a_1,\ldots,a_g,b_1,\ldots,b_g) \in \IQbar^{2g+1}$ such that
\[
c2\pi i + \sum a_k\theta_k + \sum b_l \frac{2\pi i}{\theta_l} = 0.
\]
Set $W := \{(x_0,x_1,\ldots,x_g,x_{g+1},\ldots,x_{2g}) \in \Lie (\Gm \times A^\natural) : cx_0+\sum a_k x_k + \sum b_l x_{g+l} = 0\}$. Then $W$ is defined over $\IQbar$. It contains the point $\mathbf{x}/\ell$. So $V \subseteq W$ by definition of $V$. But then $\Lie (\Gm\times A^\natural) \subseteq W$, and hence $W= \Lie (\Gm\times A^\natural)$ and
$$
c = a_1=\ldots=a_g = b_1=\ldots=b_g = 0.
$$
 This finishes the proof.
\end{proof}

\section{Bi-$\IQbar$-spaces arising from Shimura varieties: Definition}\label{SectionBiIQbarShimura}
From this section, we will build up a framework to relate Conjecture~\ref{QuestionNoQuadraticRelationIntro} to split bi-$\IQbar$-structures arising from Shimura varieties. The plan is as follows. %Let $S$ be a connected Shimura variety and let $[o] \in S(\CC)$ be a special point. 
The current $\mathsection$\ref{SectionBiIQbarShimura} defines the bi-$\IQbar$-structure in question, and $\mathsection$\ref{SectionBiIQbarShimura2} proves that this bi-$\IQbar$-structure is split. In $\mathsection$\ref{SectionPeriodSiegel} we compute the periods of this split bi-$\IQbar$-structure for $\mathbb{A}_g$, the moduli space of principally polarized abelian varieties; the computation uses the discussion on the family version of the de Rham--Betti comparison and the Kodaira--Spencer map  in $\mathsection$\ref{SectionDeRhamBettiKodaira}. Then in $\mathsection$\ref{SectionAnalyticSubspaceConj}, we propose the analogue of W\"{u}stholz's Analytic Subgroup Theorem for Shimura varieties, and explains how this conjecture applied to $\mathbb{A}_g$ implies Conjecture~\ref{QuestionNoQuadraticRelationIntro}.

\vskip 0.5em

Let $(\bG,X)$ be a connected Shimura datum,  and fix a special point $o \in X$; for definitions see $\mathsection$\ref{SubsectionRecallShimura}. It is known that $X$ is a Hermitian symmetric domain. The goals of this section are to endow $T_o X$ with a natural bi-$\IQbar$-structure. More precisely, we will endow $T_o X$ with  
\begin{itemize}
\item the \textit{arithmetic $\IQbar$-structure} in $\mathsection$\ref{SubsectionArithmeticIQbarShimura},
\item the \textit{geometric $\IQbar$-structure} in $\mathsection$\ref{SubsectionGeometricIQbarShimura}.%, and prove:
\end{itemize}
We also discuss in $\mathsection$\ref{SubsectionHC} on the $\IQbar$-structure on the Harish--Chandra (bounded) realization of $X$. This discussion is important when one relates the bi-$\IQbar$-space arising from Shimura varieties to transcendence theory.

\subsection{Shimura data and Shimura varieties}\label{SubsectionRecallShimura}
We give a quick summary of Deligne's language of Shimura varieties. References are \cite{DeligneTravaux-de-Shim, De, Milne}.

A \textit{connected Shimura datum} is a pair $(\bG,X)$ where $\bG$ is a reductive algebraic group over $\QQ$ and $X$ is the $\bG(\RR)^+$-conjugacy class of a morphism
$$
x\colon \SSS\longrightarrow \bG_{\RR}
$$
satisfying Deligne's conditions (SV1)--(SV3). It is known that $X$ is a Hermitian symmetric space, and each Hermitian symmetric space arises in this way.

Let $\Gamma \subseteq \bG(\QQ)$ be an arithmetic subgroup. The Baily--Borel theorem \cite{BailyCompactificatio} asserts that $S := \Gamma\backslash X$ has a natural structure of quasi-projective complex algebraic variety Such an $S$ is called a \textit{connected Shimura variety associated with $(\bG,X)$}.

Let $u \colon X \rightarrow S^{\mathrm{an}}$ denote the uniformization.% For each point $x \in X$, we use $[x]$ to denote $u(x)$ for simplicity.

\vskip 0.3em

Consider $x \in X$. It gives rise to a morphism $x \colon \SSS \rightarrow \bG_{\RR}$. The \textit{Mumford--Tate group} of $x$, denoted by $\mathrm{MT}(x)$, is defined to be the smallest $\QQ$-subgroup $\mathbf{H}$ of $\bG$ such that $x(\SSS) \subseteq \mathbf{H}_{\RR}$.

A point $x \in X$ is said to be \textit{special} point if $\mathrm{MT}(x)$ is a torus. A point in $S(\CC)$ is said to be \textit{special} if it is the image of a special point in $X$ under the uniformization $u$. In the particular case $S = \mathbb{A}_g$, the moduli space of principally polarized abelian varieties, special points of $S$ are precisely the points parametrizing CM abelian varieties.

\subsection{Arithmetic $\IQbar$-structure on $T_o X$}\label{SubsectionArithmeticIQbarShimura}
By general theory of Shimura varieties (see also Faltings), the Shimura variety $S$ has a canonical model over $\IQbar$ which we denote by $S_{\IQbar}$. Moreover, the special point $[o]:= u(o)$ is in $S_{\IQbar}(\IQbar)$. The differential of the uniformization $u \colon X \rightarrow S$ at $o$ is then a $\CC$-linear map %:= \mathrm{d}u
\[
\psi \colon T_o X \rightarrow T_{[o]}S.
\]
Now $T_{[o]}S$ has a natural $\IQbar$-structure $T_{[o]}S = T_{[o]}S_{\IQbar} \otimes_{\IQbar} \CC$. Thus $\psi$ defines a $\IQbar$-structure on $T_o X$ as follows: For a $\IQbar$-basis $\{e_1,\ldots,e_N\}$ of $T_{[o]}S_{\IQbar}$, the set $\{\psi^{-1}(e_1),\ldots,\psi^{-1}(e_N)\}$ is a basis of $T_o X$, and $\bigoplus_{j=1}^N \IQbar \psi^{-1}(e_j)$ defines a $\IQbar$-structure on $T_o X$. It is called the \textit{arithmetic $\IQbar$-structure} on $T_o X$.

The arithmetic $\IQbar$-structure on $T_o X$ can be characterized as follows: A subspace $W$ of $T_o X$ is rational for the arithmetic $\IQbar$-structure if and only if $\dim_{\IQbar} \left(\psi(W) \cap T_{[o]}S_{\IQbar} \right) = \dim_{\CC} \psi(W)$ (\textit{i.e.} $\psi(W)$ is rational for $T_{[o]}S_{\IQbar}$).

Using knowledge on Hecke correspondences, it is not hard to check that the arithmetic-$\IQbar$ structure on $T_o X$ does not depend on the choice of the arithmetic subgroup $\Gamma \subseteq \bG(\QQ)$.

\subsection{Geometric $\IQbar$-structure on $T_o X$}\label{SubsectionGeometricIQbarShimura}
In this subsection, we define a natural \textit{$\IQbar$-structure} on $X$ (denoted by $X_{\IQbar}$), via the Hodge theoretic interpretation of the Borel embedding theorem given by Deligne, for which $o$ lies in $X_{\IQbar}(\IQbar)$. This then endows $T_o X$ with a $\IQbar$-structure $T_o X_{\IQbar} \otimes_{\IQbar} \CC$ on $T_o X$, called the \textit{geometric $\IQbar$-structure}.

\vskip 0.3em

Let $x\colon \SSS\longrightarrow G_{\RR}$ be an element of $X$.
Let  $x_{\CC}\colon \GG_{m,\CC}\times \GG_{m,\CC}\longrightarrow G_{\CC}$ be its extension to $\CC$, and denote the associated character by  
$$
\mu_x \colon \GG_{m,\CC}\longrightarrow G_{\CC}, \quad z\mapsto x_{\CC}(z,1).
$$

For each rational representation 
${(\bf V,\rho)}$ of ${\bf G}$ and any $x \in X$, we have a Hodge structure $(V_{\RR}, \rho \circ x)$ and an associated Hodge filtration $F(x)$ on $V_{\CC}$ induced by the cocharacter $\mu_x$. %Say this filtration is given by $F(x) = (V_{\CC} \supsetneq F_x^m V_{\CC} \supsetneq  \cdots \supsetneq F_x^n V_{\CC} = 0)$. 
There exists a flag variety $\mathrm{Fl}(X,V)$ defined over $\CC$ such that each $F(x)$ is parametrized by a point in $\mathrm{Fl}(X,V)(\CC)$. Moreover, this flag variety has a natural model over $\oQ$, denoted by $\mathrm{Fl}(X,V)_{\IQbar}$, such that the $\IQbar$-points correspond to filtrations of $V_{\oQ}$ by $\oQ$-subvector spaces. Since $o \in X$ is a special point, $F(o)$ is a filtration by $\IQbar$-spaces; see \cite[Prop.3.7]{UllmoA-characterisat}. So we have $\mathrm{Fl}(X,V)_{\IQbar} \cong \mathrm{GL}(V_{\IQbar})/\mathrm{Stab}(F(o))$. 

The center $Z(\bG)(\RR)$ acts trivially on $X$. If the kernel of the representation $\rho \colon \bG \rightarrow \mathrm{GL}(\bf V)$ is contained in $Z(\bG)$, then the map
\begin{equation}\label{EqXIntoFlag}
X \rightarrow \mathrm{Fl}(X,V), \quad x \mapsto F(x)
\end{equation}
is injective. It factors through the $\bG(\CC)$-conjugacy class of $\mu_o$ in $\hom(\SSS_{\CC},\bG_{\CC})$, which by the first part of the Borel Embedding Theorem (Theorem~\ref{ThmBorelEmbedding}) can be realized as the compact dual $X^\vee$ of $X$. 
%Denote this conjugacy class by $X^\vee$.\footnote{The first part of the Borel Embedding Theorem (Theorem~\ref{ThmBorelEmbedding}) asserts that this conjugacy class can be realized as the compact dual of $X$, and hence the notation here is justified.} Later on, we will see that
Now  $\bG_{\RR} \subseteq \bG_{\CC} \xrightarrow{\rho} \mathrm{GL}(V_{\CC})$ induces injective maps
\small
\[
\begin{array}{ccccc}
X = \bG(\RR)^+/\mathrm{Stab}_{\bG(\RR)^+}(o) & \rightarrow & X^\vee = \bG(\CC)/\mathrm{Stab}_{\bG(\CC)}(\mu_o) & \rightarrow & \mathrm{Fl}(X,V) = \mathrm{GL}(V_{\CC}) / \mathrm{Stab}_{\mathrm{GL}(V_{\CC})}(F(o)) \\
x & \mapsto & \mu_x & \mapsto & F(x).
\end{array}
\]
\normalsize
Borel's embedding theorem (Theorem~\ref{ThmBorelEmbedding}) asserts that the first map realizes $X$ as an open subset (in the usual topology) of $X^\vee$. The second map makes $X^\vee$ into a projective complex algebraic subvariety of $\mathrm{Fl}(X,V)$, which furthermore descends to $\IQbar$ since $\mu_o$ descends to $\IQbar$. In other words, there exists a subvariety $X^\vee_{\IQbar}$ of $\mathrm{FL}(X,V)_{\IQbar}$, defined over $\IQbar$, such that $X^\vee = X^\vee_{\IQbar}\otimes_{\IQbar}\CC$. 

Now we are ready to define the \textit{$\IQbar$-structure} on $X$ by setting $X(\oQ):=X\cap X^{\vee}_{\IQbar}(\oQ)=X\cap \mathrm{Fl}(X,V)_{\IQbar}(\oQ)$. We have seen that $o \in X(\IQbar)$. Hence we obtain the \textit{geometric $\IQbar$-structure} on $T_o X$ as explained at the beginning of this section; it equals $T_o X^\vee_{\IQbar}$.

\vskip 0.3em
We finish this subsection with the following remark. 
Notice that $X^\vee_{\IQbar}$ does not depend on the choice of the rational representation $(\bf{V},\rho)$ (with $\ker(\rho) \subseteq Z(\bG)$). In particular since $\bG$ is a reductive group, we can take $(\bf{V},\rho)$ to be $\bf{V} = \Lie \bG^{\mathrm{ad}}$ and 
\begin{equation}\label{EqChoiceOfRepresentationForIQbarOnX}
\rho \colon \bG \rightarrow \bG^{\mathrm{ad}} = \bG/Z(\bG) \rightarrow \mathrm{GL}(\Lie \bG^{\mathrm{ad}})
\end{equation}
where the second morphism is the adjoint representation.

\subsection{Bounded realization of $X$}\label{SubsectionHC}
To make the correct analogue of W\"{u}stholz's Analytic Subgroup Theorem in the Shimura setting, it is important to work with the \textit{Harish--Chandra realization of $X$} which includes two aspects:
\begin{enumerate}
\item[(i)] identifies $T_o X$ with a certain abelian sub-Lie algebra $\Fm^+$ of $\mathfrak{g}_{\CC} := \Lie (\bG(\CC))$;
\item[(ii)] realizes $X$ as a bounded symmetric domain $\cD$ in $\Fm^+ \cong \CC^N$ such that $o$ becomes the origin.
\end{enumerate}
Here $N = \dim S = \dim X$. We will recall both aspects in the current subsection, with Lemma~\ref{LemmaTangentSpaceLieSubAlgebra} for (i) and Theorem~\ref{ThmHC} for (ii). Notice that this gives an inclusion of $X$ in $T_o X$.  
For the purpose of our paper, it is important to understand how the various $\IQbar$-structures are related under the Harish--Chandra realization. More precisely, 
we will  show that the $\IQbar$-structure on $X$ and the geometric $\IQbar$-structure on $T_o X$, both given in the previous subsection, are compatible under this inclusion; see the paragraph below Proposition~\ref{PropIQbarHC} for more details.

\vskip 0.5em

For our connected Shimura datum $(\bG, X)$, we fix the following notation. Denote by $G := \bG^{\mathrm{der}}(\RR)^+$, $G_{\RR} := \bG^{\mathrm{der}}_{\RR}$ and $G_{\CC} = \bG^{\mathrm{der}}_{\CC}$. The Lie algebras are denoted by using fractur letters: denote by $\Lie({\bf G}^{\mathrm{der}})={\bf \Fg}$,
$\Lie(G)=\Fg_{\RR}$  and $\Lie(G(\CC))=\Fg_{\CC}$.

The underlying space $X$ is a $G$-orbit because the center of $\bG(\RR)$ acts trivially on $X$.

\vskip 0.5em

Our point $o \in X$ gives rise to a morphism $o \colon \SSS \rightarrow \bG_{\RR}$. By (SV2) in the definition of Shimura data, $o(\sqrt{-1})$ defines a Cartan involution $\theta$ on $\Fg_{\RR}$. Write $\Fg_{\RR} = \Fk\oplus \Fm$ for the associated Cartan decomposition, with $\Fk$ the eigenspace of $1$ and $\Fm$ the eigenspace of $-1$. Then $G=K_{\infty}M$ with $K_{\infty}=\exp(\Fk)$ a maximal compact subgroup of $G$ and $M=\exp(\Fm)$. Moreover $X = G/K_{\infty}$ as a Riemannian symmetric space. Denote by $T_{\RR}(X)$ the real tangent space of $X$ at $o$. Then $T_{\RR}(X) = \Fm$.

The Cartan involution  $\theta$ extends to $\Fg_{\CC}$ and we have a corresponding Cartan decomposition 
$\Fg_{\CC}=\Fk_{\CC}+\Fm_{\CC}$. Let $\Fg_c:= \Fk\oplus \sqrt{-1}\Fm\subseteq \Fg_{\CC}$. Since $G$ is semi-simple of non compact type, $G_c:=\exp(\Fg_c)$ is a compact Lie group and $X^{\vee}:=G_c/K_{\infty}$ is the compact dual of the Riemannian symmetric space $X=G/K_{\infty}$.

The complex structure
on $X$ is given by an endomorphism $J$ of $T_{\RR}(X)$ such that $J^2=-\mathrm{Id}$. We have a decomposition 
$$
T_{\RR}(X) \otimes_{\RR}\CC = T^{1,0}(X)\oplus T^{0,1}(X)
$$
where $J$ acts by multiplication by $\sqrt{-1}$ on $T^{1,0}(X)$ and by $-\sqrt{-1}$ on $T^{0,1}(X)$. 
In this description $T^{1,0}(X)$ is identified with the holomorphic tangent space at $o$ and  there is a real isomorphism $T_{\RR}(X)\simeq  T^{1,0}(X)$.
We have $\Fm_{\CC}=\Fm^+\oplus \Fm^-$ where $J$ acts by multiplication by $\sqrt{-1}$ on $\Fm^+$ and by $-\sqrt{-1}$ on $\Fm^-$. It is not hard to check that $\Fm^+$ and $\Fm^-$ are abelian sub-Lie algebras, \textit{i.e.} $[\Fm^+, \Fm^+]=0$
and $[\Fm^-, \Fm^-]=0$.:

Let $M^+=\exp(\Fm^+)$, $M^-=\exp(\Fm^-)$, $K_{\CC}=\exp(\Fk_{\CC})$ and $P_{\CC}=\exp(\Fk_{\CC}+\Fm^-)= K_{\CC}M_{\CC}^-$. Notice that $P_{\CC}$ is a subgroup of $G_{\CC}$ because $\Fp:=\Fk_{\CC}\oplus \Fm^-$ is a complex sub-Lie algebra.
\begin{teo}[Borel Embedding Theorem]\label{ThmBorelEmbedding}
The embedding $G_c\rightarrow G(\CC)$ induces a biholomorphism $X^{\vee}=G_c/K_{\infty}\longrightarrow G(\CC)/P(\CC)$.  
The embedding $G\rightarrow G(\CC)$ induces an open embedding
$$
X=G/K_{\infty}\longrightarrow G(\CC)/P_{\CC}\simeq X^{\vee},
$$
realizing $X$ as a open subset (in the usual topology) of its compact dual $X^{\vee}$.
\end{teo} 

Let $T_o X$ be the holomorphic tangent space of $X$ at $o$. It is canonically isomorphic to $T_o X^\vee$ under the Borel embedding theorem, which is furthermore canonically isomorphic to $\Fg_{\CC}/ (\Fk_{\CC}+\Fm^-) = \Fm^+$. Hence we have:

\begin{lem}\label{LemmaTangentSpaceLieSubAlgebra}
Under the identification $X = G/K_{\infty}$ and the Borel embedding theorem, we have $T_o X = T_o X^\vee = \Fm^+$. 
\end{lem}

 The Harish--Chandra embedding theorem states the following.
\begin{teo}[Harish--Chandra]\label{ThmHC}
The map
$$
F \colon M^+\times K_{\CC} \times M^-\longrightarrow G_{\CC}, \qquad 
(m^+,k,m^-)\mapsto m^+km^- 
$$
is a biholomorphism  of $M^+\times K_{\CC} \times M^-$ onto an open subset of $G(\CC)$ containing $G$. As a consequence the map
\begin{equation}\label{EqHC}
\begin{array}{rccc}
\eta \colon & \Fm^+ &\longrightarrow & G(\CC)/P_{\CC} =X^{\vee} \\
&m^+ & \mapsto & \exp(m^+)P_{\CC} \nonumber
\end{array}
\end{equation}
is a biholomorphism onto a dense open subset of $X^{\vee}$ containing $X$. Furthermore $\cD := \eta^{-1}(X)$ is a bounded symmetric domain in $\Fm^+\simeq \CC^N$ and $\eta^{-1}(o) = 0$.
\end{teo}

Up to now in this subsection, we have \textit{not} used any assumption on the chosen point $o \in X$. From now on, recall that $o \in X(\IQbar)$.%Now we use the fact that $o \in X(\IQbar)$ to endow $\cD$ with a natural $\IQbar$-structure. 

We start with understanding the geometric $\IQbar$-structure on $T_o X$ given in $\mathsection$\ref{SubsectionGeometricIQbarShimura}  from the Lie algebra point of view. Notice that the geometric $\IQbar$-structure on $T_o X$ and 
the identification of $\Fm^+$ with $T_o X$ from Lemma~\ref{LemmaTangentSpaceLieSubAlgebra} together induce a natural $\IQbar$-structure on $\Fm^+$, which we denote by $\Fm^+_{\IQbar,\mathrm{geom}}$. % , the assumption that is \textit{not} used up to now.
%From the discussion above Lemma~\ref{LemmaTangentSpaceLieSubAlgebra}, we obtain immediately:

\begin{prop}\label{PropIQbarHC}
We have:
\begin{enumerate}
\item[(i)] 
$\Fm^+_{\IQbar,\mathrm{geom}} = \Fm^+ \cap \Fg_{\IQbar}$.
\item[(ii)] The map $\eta$ from \eqref{EqHC} descends to a polynomial morphism $\Fm^+_{\IQbar,\mathrm{geom}} \to X^\vee(\IQbar)$ defined over $\IQbar$.
\end{enumerate}
\end{prop}
Define the $\IQbar$-structure on $\cD$ by setting $\cD(\IQbar) :=\Fm^+_{\IQbar,\mathrm{geom}} \cap \cD$. Then by Proposition~\ref{PropIQbarHC}, $\eta$ induces a bijection between $\cD(\IQbar)$ and $X(\IQbar)$.
\begin{proof}[Proof of Proposition~\ref{PropIQbarHC}]
From the discussion above Lemma~\ref{LemmaTangentSpaceLieSubAlgebra}, the identification $\Fm^+ = T_o X$ is given by $\Fm^+ = \Fg_{\CC}/(\Fk_{\CC}+\Fm^-)$. Since $o \in X(\IQbar)$, in the Cartan decomposition $\Fk_{\CC}$, $\Fm^+$ and $\Fm^-$ descend to $\IQbar$. Hence
\[
\Fm^+_{\IQbar,\mathrm{geom}} = \Fg_{\IQbar}/((\Fk_{\CC}+\Fm^-)\cap \Fg_{\IQbar}) = \Fm^+ \cap \Fg_{\IQbar}.
\]
This proves part (i).

 As $\Fm^+$ is an abelian Lie algebra, the exponential map on $\Fm^+$ is given by a polynomial expression with coefficient in $\QQ$. Thus part (ii) holds true.
\end{proof}

\section{Bi-$\IQbar$-spaces arising from Shimura varieties: Proof of the Splitting}\label{SectionBiIQbarShimura2}

Let $(\bG,X)$ be a connected Shimura datum,  and fix a special point $o \in X$. We have endowed $T_o X$ with a bi-$\IQbar$-structure in $\mathsection$\ref{SectionBiIQbarShimura}, \textit{i.e.} the \textit{arithmetic $\IQbar$-structure} in $\mathsection$\ref{SubsectionArithmeticIQbarShimura} and the \textit{geometric $\IQbar$-structure} in $\mathsection$\ref{SubsectionGeometricIQbarShimura}. The goal of this section is to prove:

\begin{teo}\label{ThmBiIQbarShimuraSplit}
The bi-$\IQbar$-structure on $T_o X$ thus defined is split.
\end{teo}
Moreover, we will explain how the splitting of this bi-$\IQbar$-structure on $T_o X$ is obtained; see Theorem~\ref{ThmBiIQbarShimuraSplitBis}.

\vskip 0.3em

By Theorem~\ref{ThmBiIQbarShimuraSplit} and the general theory from $\mathsection$\ref{SectionBiIQbar}, we have an isomorphism of bi-$\oQ$-structure $T_o(X)\simeq (\CC^N; \alpha_1,\dots,\alpha_N)$ with $\alpha_i\in \CC^*/\oQ^*$. If $S = \mathbb{A}_g$,  then $o$ corresponds to a CM Abelian variety $A_o$ and $\Lie(A_o)\simeq (\CC^g ; \theta_1,\dots, \theta_g)$ with $\alpha_i\in \CC^*/\oQ^*$ by the results of $\mathsection$\ref{SubsectionExampleCMAV}. We will describe relations between the $\alpha_i$ and the $\theta_i$ using deformation theory in the next sections.

\subsection{Notation}\label{SubsectionNotationSplitBiIQbarShimura}

We retain the notation from $\mathsection$\ref{SubsectionHC} that $\Fg := \Lie (\bG^{\mathrm{der}})$. The natural morphism $\bG^{\mathrm{der}} \subseteq \bG \rightarrow \bG^{\mathrm{ad}} = \bG/Z(\bG)$ induces $\Fg = \Lie (\bG^{\mathrm{ad}})$.

The point $o \in X$ defines a Cartan involution $\theta$ of $\Fg_\RR$ by (SV2) of Deligne's definition of Shimura data. Write $\theta_{\CC}$ for its extension to $\Fg_{\CC}$. We have the following Cartan decompositions
$$
\Fg_\RR=\Fk\oplus \Fm, \ \Fg_c=\Fk+\sqrt{-1}\Fm \mbox{ and } \Fg_{\CC}=\Fk_{\CC}\oplus \Fm_{\CC},
$$
with $\Fm_{\CC}=\Fm^+\oplus \Fm^-$. See $\mathsection$\ref{SubsectionHC} for the notation.

Since $o$ is a special point, its Mumford--Tate group $\mathrm{MT}(o)$ is a torus. Let $\mathbf{T}$ be a maximal torus of $\bG$ which contains $\mathrm{MT}(o)$. Write $\mathbf{T}^{\mathrm{ad}}$ for its image under the morphism $\bG \rightarrow \bG^{\mathrm{ad}}$. Then $\mathbf{T}^{\mathrm{ad}}$ is a maximal torus of $\bG^{\mathrm{ad}}$, and hence $\Fh_{\RR} := \Lie \mathbf{T}^{\mathrm{ad}}(\RR)$ is a real Cartan subalgebra of $\Fg_{\RR}$. Moreover, $\Fh_{\RR} \subseteq \Fk$ since $\Fk$ is the eigenspace of $1$ for the Cartan involution.

\subsection{Decomposition of the Geometric $\IQbar$-structure on $T_o X$}\label{SubsectionDecompositionGeomShimura}
In this subsection, we will show that the $T_o X_{\IQbar}$, \textit{i.e.} the geometric $\IQbar$-structure on $T_o X$, can be decomposed into the direct sum of $1$-dimensional $\IQbar$-sub-vector spaces, each being an eigenspace for the action of a maximal torus in $\bG$ which contains $\mathrm{MT}(o)$.

\subsubsection{Statement of the result}%The Lie algebra point of view of the geometric $\IQbar$-structure}
Let $\Fz$ be the center of  $\Fk$. Then $\Fz \subseteq \Fh_{\RR}$ since $\Fh_{\RR}$ is abelian. %Let $\Fk_s=[\Fk,\Fk]$ the semi-simple part of $\Fk$. Then 
The following result can be found in \cite[pp.54, Prop.1]{Mok}.

\begin{lem}\label{LemmaCenterElementComplexStructureSV}
 There exists $z\in \Fz$ such that the $J$-operator on $\Fm$, which induces the complex structure on $\Fm$, is defined by $J.v=[z,v]$ for any $v\in \Fm$.
\end{lem}

Proposition~\ref{PropIQbarHC}.(i) gives the geometric $\IQbar$-structure on $T_o X$ a Lie algebra point of view, \textit{i.e.} it is naturally identified with $ \Fm^+_{\IQbar,\mathrm{geom}} = \Fm^+ \cap \Fg_{\IQbar}$ via Lemma~\ref{LemmaTangentSpaceLieSubAlgebra}.
\begin{prop}\label{PropGeometricIQbarLieAlgebra}
$\Fm^+_{\IQbar,\mathrm{geom}}$ can be decomposed into the direct sum of $1$-dimensional $\IQbar$-spaces.
%\end{enumerate}
\end{prop}

In fact, we will prove a more precise version, Proposition~\ref{PropDecompositionGeomIQbar}, which explains how this decomposition  is constructed. This construction is important.

\subsubsection{Complex multiplication and root space decomposition}

Let $\Fh_{\RR}^\vee$ be the dual space of $\Fh_{\RR}$. Let $\Phi \subseteq \Fh_{\CC}^\vee$ be the set of roots of $\Fh_{\CC}$ in $\Fg_{\CC}$. Then we have the root decomposition
\begin{equation}\label{EqRootDecompositionOverCC}
\Fg_{\CC} = \Fh_{\CC} \oplus \left( \bigoplus_{\phi \in \Phi} \Fg_{\CC}^{\phi} \right),
\end{equation}
where $\Fg_{\CC}^{\phi} = \{ x \in \Fg_{\CC} : [h,x] = \phi(h)x \text{ for all }h \in \Fh_{\CC}\}$ has dimension $1$. For each root $\phi \in \Phi$, let $f_{\phi}$ be a generator of the root space $\Fg_{\CC}^{\phi}$.

We say that a root $\phi$ is \textit{compact} (resp. \textit{non compact}) if $f_{\phi} \in \Fk_{\CC}$ (resp. $f_{\phi} \in \Fm_{\CC}$.  Let $\Phi_K$  be the set of compact roots and $\Phi_M$ be the set of non compact roots. We claim that  
 $\Phi=\Phi_K\cup \Phi_M$. Indeed, since $\theta$ is an involution fixing $\Fk_{\CC}$ and $\Fh_{\CC} \subseteq \Fk_{\CC}$, we have $[h,\theta(f_{\phi})] = \phi(h)\theta(f_{\phi})$. Hence $\theta(f_{\phi}) \in \Fg_{\CC}^{\phi}$, and therefore $\theta(f_{\phi}) = \lambda f_{\phi}$ for some $\lambda \in \CC^*$. If $f_{\phi} = a_K f_{\phi,K}+ a_M f_{\phi,M}$ under the Cartan decomposition with $a_K, a_M \in \CC$, then $\theta(f_{\phi}) = a_K f_{\phi,K}- a_M f_{\phi,M}$. Then either $a_K = 0$ or $a_M = 0$, and hence we conclude that $f_{\phi}$ is in $\Fk_{\CC}$ or in $\Fm_{\CC}$.
 
\begin{lem}\label{LemmaRootDecompositionTangentSpace}
There exists a choice $\Phi^+$ of positive roots, with $\Phi = \Phi^+ \coprod -\Phi^+$, satisfying the following property. For $\Phi_M^+ := \Phi_M \cap \Phi^+$, we have 
\begin{equation}\label{eqgeom}
 \Fm^+=\bigoplus_{\phi\in \Phi_M^+} \Fg_{\CC}^{\phi}. %=\Sigma_{\phi\in \Phi_M^+} \Fm^{+,\phi}.
\end{equation}
\end{lem} 
 
Before proving this lemma, let us explain how $\Phi^+$ is constructed.  
%inlinecomment!!!
Each $\phi \in \Phi$ takes real values on $\sqrt{-1}\Fh_{\RR}$, so $\Phi$ is a root system in $\sqrt{-1}\Fh_{\RR}^\vee$. 
 Since $\Fg_{\CC}$ is semi-simple, $\Phi$ spans $\sqrt{-1}\Fh_{\RR}^\vee$. From now on, we identify $\sqrt{-1}\Fh_{\RR}^\vee$ with $\sqrt{-1}\Fh_{\RR}$ using the Killing form on $\Fg_{\RR}$.\footnote{This Killing form is positive definite when restricted to $\sqrt{-1}\Fh_{\RR}$, and hence induces $\sqrt{-1}\Fh_{\RR}^\vee \cong \sqrt{-1}\Fh_{\RR}$.}. 
Now consider the Weyl chambers in $\sqrt{-1}\Fh_{\RR}$ with respect to the root system $\Phi$. Take a Weyl chamber $C$ such that $-\sqrt{-1}z \in \overline{C}$, with $z$ from Lemma~\ref{LemmaCenterElementComplexStructureSV}. For each $\gamma$ in the interior of $C$, the set $\Phi^+ := \{\phi \in \Phi : \phi(\gamma) > 0 \}$ does not depend on the choice of $\gamma$.  By theory of root systems, we have $\Phi = \Phi^+ \coprod -\Phi^+$. We will show that this is our desired $\Phi^+$.

\begin{proof}  
We start with the following observation: 
 For any non compact root $\phi$, either $f_{\phi}\in \Fm^+$ or $f_{\phi}\in \Fm^-$. Indeed  by Lemma~\ref{LemmaCenterElementComplexStructureSV}, $Jf_{\phi} = [z,f_{\phi}]$ for some $z \in \Fz$. Since $\Fh_{\CC} \subseteq \Fk_{\CC}$, we have $[h,z] = 0$ for all $h \in \Fh_{\CC}$. Hence for each $h \in \Fh_{\CC}$, we have $[h, Jf_{\phi}] = [h,[z,f_{\phi}]] = [z,[h,f_{\phi}]]$, with the last equality induced by the Jacobi identity, and further equals $[z, \phi(h)f_{\phi}] = \phi(h) [z,f_{\phi}] = \phi(h) Jf_{\phi}$. Hence $Jf_{\phi} \in \Fg_{\CC}^{\phi}$. So $Jf_{\phi}$ is a scalar of $f_{\phi}$.  
 Now that $f_{\phi} \in \Fm_{\CC} = \Fm^+ \oplus \Fm^-$, we have $f_{\phi} = a^+ m^+ + a^- m^-$ for some $m^+ \in \Fm^+$, $m^- \in \Fm^-$ and $a^+,a^- \in \CC$.  As $\Fm^+$ is the $J$-eigenspace associated with $\sqrt{-1}$ and $\Fm^-$ is the $J$-eigensapce associated with $-\sqrt{-1}$, we then have $Jf_{\phi} = \sqrt{-1}a^+ m^+ - \sqrt{-1} a^- m^-$. But $Jf_{\phi}$ is a scalar of $f_{\phi}$. So either $a^+ = 0$ or $a^- = 0$. Hence we are done for the claim.
  
This shows that $\Fm^+ = \oplus_{\phi \in \Phi_0} \Fg_{\CC}^{\phi}$ for a subset $\Phi_0 \subseteq \Phi_M$. It remains to show $\Phi_0 = \Phi_M^+$.

Assume $\phi \in \Phi_0$, \textit{i.e.} $f_{\phi} \in \Fm^+$. Then $[-\sqrt{-1}z, f_{\phi}] = -\sqrt{-1} Jf_{\phi} = f_{\phi}$. Thus $\phi(-\sqrt{-1}z) = 1 > 0$. So $\phi \in \Phi^+$ by construction. This proves $\Phi_0 \subseteq \Phi^+ \cap \Phi_M = \Phi_M^+$.

Assume $\phi \in \Phi_M^+$. Then $\phi(-\sqrt{-1}z) \ge 0$ by construction. If $\phi \not\in \Phi_0$, then $[-\sqrt{-1}z, f_{\phi}] = - \sqrt{-1} J f_{\phi} = -f_{\phi}$ and hence $\phi(-\sqrt{-1}z) = -1 < 0$, contradiction. This proves $\Phi_M^+ \subseteq \Phi_0$.

Now we are done.
 \end{proof}

\subsubsection{Decomposition over $\IQbar$}
Another way to see the root decomposition \eqref{EqRootDecompositionOverCC} is by the adjoint representation restricted to the maximal torus $\mathbf{T}$ and the character group $X^*(\mathbf{T})$. More precisely, the morphism $\bG \rightarrow \bG^{\mathrm{ad}}$ induces an inclusion $X^*(\mathbf{T}^{\mathrm{ad}}) \subseteq X^*(\mathbf{T})$. In fact, $Z(\bG) < \mathbf{T}$ and $\mathbf{T}^{\mathrm{ad}} = \mathbf{T}/Z(\bG)$, and hence
\begin{equation}
X^*(\mathbf{T}^{\mathrm{ad}}) = \{ \chi \in X^*(\mathbf{T}) : \chi|_{Z(\bG)} \equiv 1\}.
\end{equation}

The adjoint representation of $\bG^{\mathrm{ad}}$ on $\Fg$ induces
\begin{equation}\label{EqRootDecompositionGroup}
\Fg_{\CC} = \Fh_{\CC} \oplus \left( \bigoplus_{\chi \in \Phi(\mathbf{T}_{\CC},\bG_{\CC})} \Fg_{\CC}^{\chi} \right)
\end{equation}
where $\Phi(\mathbf{T}_{\CC},\bG_{\CC}) \subseteq X^*(\mathbf{T}^{\mathrm{ad}}_{\CC}) \subseteq X^*(\mathbf{T}_{\CC})$ and $\Fg_{\CC}^{\chi} = \{ x \in \Fg_{\CC} : \mathrm{Ad}(t)(x) = \chi(t)x \text{ for all }t \in \mathbf{T}(\CC)\}$ has dimension $1$.

By general theory of root decomposition for semi-simple groups, the sets of spaces $\{\Fg_{\CC}^{\phi}\}_{\phi \in \Phi}$ and $\{\Fg_{\CC}^{\chi}\}_{\chi \in \Phi(\mathbf{T}_{\CC},\bG_{\CC})}$ coincide. More precisely, there exists a bijection $\Phi \cong \Phi(\mathbf{T}_{\CC},\bG_{\CC})$, $\phi \mapsto \chi_{\phi}$, such that $\Fg_{\CC}^{\phi} = \Fg_{\CC}^{\chi_{\phi}}$. Both $\mathbf{T}$ and $\bG$ are $\QQ$-groups, so there is a natural isomorphism $X^*(\mathbf{T}_{\IQbar}) = X^*(\mathbf{T}_{\CC})$. Thus each $\Fg_{\CC}^{\chi}$ is naturally defined over $\IQbar$; more precisely  $\Fg_{\CC}^{\chi} = (\Fg_{\CC}^{\chi} \cap \Fg_{\IQbar}) \otimes_{\IQbar}\CC$.

Let $\Phi_M^+$ be from Lemma~\ref{LemmaRootDecompositionTangentSpace}. 
By abuse of notation, we still use $\Phi_M^+$ to denote the image of $\Phi_M^+ \subseteq \Phi$ under $\Phi \cong \Phi(\mathbf{T}_{\CC},\bG_{\CC}) \subseteq X^*(\mathbf{T}_{\CC}) = X^*(\mathbf{T}_{\IQbar})$.
\begin{prop}\label{PropDecompositionGeomIQbar}
Denote by $\Fm^{+,\chi}_{\IQbar,\mathrm{geom}} := \Fg_{\CC}^{\chi} \cap \Fg_{\IQbar}$ for each $\chi \in \Phi_M^+$. Then
\begin{equation}
\Fm^+_{\IQbar,\mathrm{geom}}= \Fm^+ \cap \Fg_{\IQbar} = \bigoplus_{\chi \in \Phi_M^+} \Fm^{+,\chi}_{\IQbar,\mathrm{geom}},
\end{equation}
with $\dim_{\IQbar} \Fm^{+,\chi}_{\IQbar,\mathrm{geom}} = 1$ for each $\chi \in \Phi_M^+$ (equivalently, $\Fm^{+,\chi}_{\IQbar,\mathrm{geom}}\otimes \CC = \Fg_{\CC}^{\chi}$).
\end{prop}

This proposition is a more precise version of Proposition~\ref{PropGeometricIQbarLieAlgebra}.%, with the second equality for part (i) and the first equality for part (ii).

\begin{proof}[Proof of Proposition~\ref{PropDecompositionGeomIQbar}]
The second equality follows immediately from Lemma~\ref{LemmaRootDecompositionTangentSpace}, and we have seen $\dim_{\IQbar} \Fg_{\IQbar}^{\chi} = 1$ above. To see the first equality, recall that $\Fm^+_{\IQbar,\mathrm{geom}} = T_o X^\vee_{\IQbar}$. The $\IQbar$-structure $X^\vee_{\IQbar}$ is given by the flag variety associated with the representation \eqref{EqChoiceOfRepresentationForIQbarOnX}, which is precisely the representation giving the root decomposition \eqref{EqRootDecompositionGroup}. Hence we are done.
\end{proof}

\subsection{Decomposition of the Arithmetic $\oQ$-structure on $T_o X$.}\label{SubsectionDecompositionArithShimura}
In this subsection, we turn to the decomposition of the arithmetic $\IQbar$-structure on $T_o X$ into the direct sum of $1$-dimensional $\IQbar$-space.

We retain the notation from $\mathsection$\ref{SubsectionRecallShimura} and \ref{SubsectionArithmeticIQbarShimura}. Now $S$ is a connected Shimura variety associated with the connected Shimura datum $(\bG,X)$, and $S$ has a canonical model $S_{\IQbar}$ over $\IQbar$. As a complex variety $S^{\mathrm{an}} = \Gamma \backslash X$ for some arithmetic subgroup $\Gamma \subseteq \bG(\QQ)$. 

From the uniformization $u \colon X \rightarrow S$ and the special point $o \in X$, we get a $\IQbar$-point $[o]:=u(o) \in S_{\IQbar}(\IQbar)$. The arithmetic $\IQbar$-structure on $T_o X$ is the one obtained from $\psi^{-1}(T_{[o]}S_{\IQbar})$, where $\psi$ is the differential $\mathrm{d} u \colon T_o X \rightarrow T_{[o]} S$.

The point $o \in X$ gives rise to a Cartan decomposition $\Fg_{\RR} = \Fk \oplus \Fm$, and we have $\Fm_{\CC} = \Fm^+ \oplus \Fm^-$; see $\mathsection$\ref{SubsectionHC}. By Lemma~\ref{LemmaTangentSpaceLieSubAlgebra}, $\Fm^+$ can be identified with $T_o X$. Hence the arithmetic $\IQbar$-structure on $T_o X$ defined above yields a $\IQbar$-structure on $\Fm^+$, which we denote by 
%this arithmetic $\IQbar$-structure by
 $\Fm^+_{\IQbar,\mathrm{arith}}$.

\vskip 0.3em

Let $\mathbf{T}$ be a maximal torus in $\bG$ which contains $\mathrm{MT}(o)$ such that $\Lie \mathbf{T}(\RR)$ is contained in $\Fk$ modulo $\Lie Z(\bG)(\RR)$. In particular, $\mathbf{T}(\RR)o = o$. Then any $t \in \mathbf{T}(\QQ)$ induces the following commutative diagram, and $o \in X$ is mapped 
\begin{equation}\label{EqDiagramHeckeForArithmetic}
\xymatrix{
X \ar[d]_-{u} & X \ar[l]_-{\mathrm{id}} \ar[r]^-{t\cdot } \ar[d]_-{u_t} & X \ar[d]^-{u} & o \ar@{|->}[d] & o \ar@{|->}[l] \ar@{|->}[r] \ar@{|->}[d] & o \ar@{|->}[d] \\
S & S_t := (\Gamma \cap t^{-1}\Gamma t)\backslash X \ar[l]_-{[\mathrm{id}]} \ar[r]^-{[t\cdot ]} & S & [o] & [o_t]:= u_t(o) \ar@{|->}[l] \ar@{|->}[r] & [o]
}
\end{equation}
where $u_t$ is the quotient. The theory of Shimura varieties asserts that every morphism in the bottom line is algebraic, and is furthermore defined over $\IQbar$. The morphism $[\mathrm{id}]$ induces $T_{[o_t]} S_{t,\IQbar} = T_{[o]}S_{\IQbar}$. Thus the differential of $[t\cdot]$ induces a $\IQbar$-linear map 
\[
\rho_t \colon T_{[o]}S_{\IQbar} \rightarrow T_{[o]}S_{\IQbar}.
\]
By abuse of notation, we also use $\rho_t$ to denote its base change to $\CC$.

\begin{prop}\label{PropActionTArithmeticAdjoint}
Under the identification of $T_o X$ with $\Fm^+$ from Lemma~\ref{LemmaTangentSpaceLieSubAlgebra}, we have
\begin{equation}
\rho_t = \psi \circ \mathrm{Ad}(t) \circ \psi^{-1}.
\end{equation}
\end{prop}
\begin{proof}
Differentiating \eqref{EqDiagramHeckeForArithmetic}, we get that $\psi^{-1} \circ \rho_t \circ \psi \colon T_o X \rightarrow T_o X$ is $\mathrm{d}(t\cdot)_o$.

Use the notation from Lemma~\ref{LemmaTangentSpaceLieSubAlgebra} and above. Recall our choice of $\mathbf{T}$ that $\Lie \mathbf{T}(\RR)$ is contained in $\Fk$ modulo $\Lie Z(\bG)(\RR)$. Since $[\Fk,\Fm] \subseteq \Fm$ and $[\Fk_{\CC}, \Fm^+] \subseteq \Fm^+$, we have $\mathrm{Ad}(\mathbf{T}(\RR))(\Fm)\subseteq \Fm$ and $\mathrm{Ad}(\mathbf{T}(\CC))(\Fm^+)\subseteq \Fm^+$.

For $X = G/K_{\infty}$ as a Riemannian symmetric space, the real tangent space of $X$ at $o$ is $\Fm$. 
The map $t \cdot X \rightarrow X$ is induced by $G \rightarrow G$, $g \mapsto t g t^{-1}$, and hence the differential of $t \cdot$ at $o$ is induced by $\mathrm{Ad}(t) \colon \Fm \rightarrow \Fm$. Passing from $\RR$ to $\CC$ and considering the decomposition of $\Fm_{\CC} = \Fm^+ \oplus \Fm^-$, we can conclude.
\end{proof}

Let $\Phi_M^+$ be the subset of $X^*(\mathbf{T}_{\IQbar})$ defined as above Proposition~\ref{PropDecompositionGeomIQbar}. Then $\Fm^+ = \bigoplus_{\chi\in \Phi_M^+} \Fg_{\CC}^{\chi}$, with $\Fg_{\CC}^{\chi} = \{ x \in \Fg_{\CC} : \mathrm{Ad}(t)(x) = \chi(t)x \text{ for all }t \in \mathbf{T}(\CC)\}$ of dimension $1$.
\begin{prop}\label{PropDecompositionArithIQbar}
For each $\chi\in \Phi_M^+$,
\[
\Fm^{+,\chi}_{\IQbar,\mathrm{arith}} := \Fm^+_{\IQbar,\mathrm{arith}} \cap \Fg_{\CC}^{\chi}
\]
is a $\IQbar$-space of dimension $1$ (equivalently, $\Fm^{+,\chi}_{\IQbar,\mathrm{arith}}\otimes \CC = \Fg_{\CC}^{\chi}$). 

As a consequence, $\Fm^+ = \bigoplus_{\chi\in \Phi_M^+} \Fm^{+,\chi}_{\IQbar,\mathrm{arith}}$.
\end{prop}
\begin{proof}
By Proposition~\ref{PropActionTArithmeticAdjoint} and the definition of the arithmetic $\IQbar$-structure $\Fm^+_{\IQbar,\mathrm{arith}}$ on $\Fm^+$, $\mathbf{T}(\QQ)$ acts on $\Fm^+_{\IQbar,\mathrm{arith}}$ via $t \cdot x = \mathrm{Ad}(t)(x)$. 
%\[
%\mathbf{T}(\QQ) \rightarrow \mathrm{GL}(\Fm^+_{\IQbar,\mathrm{arith}}), \qquad t \mapsto \mathrm{Ad}(t).
%\]
For each $\chi \in \Phi_M^+ \subseteq X^*(\mathbf{T}_{\IQbar})$, the weight space associated with $\chi$ (defined as $\{x \in \Fm^+_{\IQbar,\mathrm{arith}} : \mathrm{Ad}(t)(x) = \chi(t)x \text{ for all }t \in \mathbf{T}(\IQbar)\}$) is precisely $\Fm^{+,\chi}_{\IQbar,\mathrm{arith}}$. Extend this notation and write $ \Fm^{+,\chi}_{\IQbar,\mathrm{arith}}$ for the weight space associated with $\chi$ for each $\chi \in X^*(\mathbf{T}_{\IQbar})$.

The general theory of algebraic tori and their character groups then asserts that $\Fm^+_{\IQbar,\mathrm{arith}} = \bigoplus_{\chi \in X^*(\mathbf{T}_{\IQbar})} \Fm^{+,\chi}_{\IQbar,\mathrm{arith}}$. Tensoring $\CC$ on both sides, we get $\Fm^+ = \bigoplus_{\chi \in \Phi_M^+} (\Fm^{+,\chi}_{\IQbar,\mathrm{arith}}\otimes \CC) \oplus \bigoplus_{\chi\not\in \Phi_M^+} (\Fm^{+,\chi}_{\IQbar,\mathrm{arith}} \otimes \CC)$. By Lemma~\ref{LemmaRootDecompositionTangentSpace}, $\Fm^+ = \bigoplus_{\chi\in \Phi_M^+} \Fg_{\CC}^{\chi}$. Thus the conclusion follows because $\Fm^{+,\chi}_{\IQbar,\mathrm{arith}}\otimes \CC \subseteq \Fg_{\CC}^{\chi}$ for each $\chi \in X^*(\mathbf{T}_{\IQbar})$.
\end{proof}

\subsection{Conclusion for Theorem~\ref{ThmBiIQbarShimuraSplit}}
By Proposition~\ref{PropDecompositionGeomIQbar} and Proposition~\ref{PropDecompositionArithIQbar}, the bi-$\IQbar$-structure on $T_o X = \Fm^+$ is split. This establishes Theorem~\ref{ThmBiIQbarShimuraSplit}. In fact, these two propositions yield a more precise statement as follows.

Let $\Fg_{\RR} = \Fk \oplus \Fm$ be the Cartan decomposition  of $\Fg_{\RR} := \Lie \bG^{\mathrm{der}}(\RR)$ associated with the special point $o \in X$ by (SV2).  Let $\mathbf{T}$ be a maximal torus in $\bG$. % which contains $\mathrm{MT}(o)$ such that $\Lie \mathbf{T}(\RR)$ is contained in $\Fk$ modulo $\Lie Z(\bG)(\RR)$.
 By Proposition~\ref{PropDecompositionGeomIQbar} and Proposition~\ref{PropDecompositionArithIQbar} we have:

\begin{teobis}{ThmBiIQbarShimuraSplit}\label{ThmBiIQbarShimuraSplitBis}
Endow $T_o X$ with the bi-$\IQbar$-structure from $\mathsection$\ref{SubsectionArithmeticIQbarShimura} and \ref{SubsectionGeometricIQbarShimura}. Then $\mathbf{T}(\QQ)$ acts on $T_o X$ via bi-$\IQbar$-automorphisms and the splitting under this action gives the bi-$\IQbar$-splitting of $T_o X$.
\end{teobis}

\section{De Rham--Betti comparison in family and the Kodaira--Spencer map}\label{SectionDeRhamBettiKodaira}
In this section, we recall the comparison of the relative de Rham cohomology and the Betti cohomology for families of abelian varieties. These are useful in the computation of the periods of the Siegel modular varieties in the next section $\mathsection$\ref{SectionPeriodSiegel}.

Let $k \subseteq \CC$ be an algebraically closed field. 
Let $S$ be a smooth irreducible variety and let $f \colon \cA \rightarrow S$ be an abelian scheme of relative dimension $g$, all defined over $k$. Assume that $\cA$ carries a principal polarization, \textit{i.e.} an isomorphism of abelian schemes $\lambda \colon \cA \xrightarrow{\sim} \cA^{\mathrm{t}}$ where $\cA^{\mathrm{t}}$ is the dual abelian scheme of $\cA/S$. 

Let $u \colon \tilde{S} \to S^{\mathrm{an}}$ be the uniformization in the category of complex analytic spaces. Set
\begin{equation}\label{EqPullbackUnivAbVarSiegel}
\xymatrix{
\cA_{\tilde{S}} := \tilde{S} \times_S \cA \ar[r] \ar[d]_-{\tilde{f}} \pullbackcorner & \cA \ar[d]^-{f} \\
\tilde{S} \ar[r]^-{u} & S.
}
\end{equation}
Then $\cA_{\tilde{S}} / \tilde{S}$ is a family of abelian varieties which carries a principal polarization. 

For each $s \in S(\CC)$ (resp. $\tilde{s} \in \tilde{S}$), denote by $\cA_s:=f^{-1}(s)$ (resp. $\cA_{\tilde{s}} := \tilde{f}^{-1}(\tilde{s})$). Notice that for each $\tilde{s} \in \tilde{S}$, $\cA_{\tilde{s}}$ can be canonically identified with $\cA_{u(\tilde{s})}$.

We have a canonical short exact sequence
\begin{equation}\label{EqRelative1Form}
0 \to f^*\Omega_S^{1} \to \Omega_{\cA}^{1} \to \Omega_{\cA/S}^{1} \to 0.
\end{equation}
It is known that the de Rham complex $\Omega_{\cA/S}^{\bullet} = (\cO_{\cA} \to \Omega_{\cA/S}^1 \to \Omega_{\cA/S}^2 \to \cdots)$ is a resolution of $f^{-1}\cO_S$.

\subsection{De Rham cohomology and symplectic basis}\label{SubsectionBasesSiegel}

The relative de Rham bundle on $S$ is defined as follows. Let $\cH^1_{\mathrm{dR}}(\cA/S) := R^1  f_* \Omega_{\cA/S}^{\bullet}$; it is a locally free sheaf of rank $2g$, which we view as a vector bundle of rank $2g$ over $S$. We have a subbundle $\Omega_{\cA} :=  f_*\Omega_{\cA/S}^1$ of $\cH^1_{\mathrm{dR}}(\cA /S) \to S$ which has rank $g$; over $s \in S(k)$, the fiber $\Omega_{\cA,s}$ is precisely $\Omega_{\cA_s}$ and hence consists of invariant holomorphic differentials on $\cA_s$.

Set
\begin{equation}\label{EqDeRhamOverSiegel}
\cH^1_{\mathrm{dR}}(\cA_{\tilde{S}} /\tilde{S}) := u^*\cH^1_{\mathrm{dR}}(\cA /S) \quad \text{and} \quad \Omega_{\cA_{\tilde{S}}} := u^*\Omega_{\cA}.
\end{equation}
Restricted to each $\tilde{s} \in \tilde{S}$, the inclusion $ \Omega_{\cA_{\tilde{S}}} \subseteq \cH^1_{\mathrm{dR}}(\cA_{\tilde{S}} /\tilde{S})$ becomes $\Omega_{\cA_{\tilde{s}}} \subseteq H^1_{\mathrm{dR}}(\cA_{\tilde{s}})$. 

The advantage of using $\Omega_{\cA_{\tilde{S}}} \subseteq \cH^1_{\mathrm{dR}}(\cA_{\tilde{S}} /\tilde{S})$ instead of working directly on $S$ is that we can take global basis. More precisely, we have:

\begin{contr}\label{LemmaGlobalBasisSympDeRham}
There exists a \textit{global} basis $\{\omega_1,\ldots,\omega_g\}$ of $\Omega_{\cA_{\tilde{S}}}$,\footnote{Namely, the $\omega_j$'s are sections of $\Omega_{\cA_{\tilde{S}}} \to \tilde{S}$ such that $\sum_{j=1}^g \CC \omega_j(\tilde{s}) = \Omega_{\cA_{\tilde{s}}}$ for each $\tilde{s} \in \tilde{S}$.} which can be completed into a  global basis $\{\omega_1,\ldots,\omega_g, \eta_1,\ldots,\eta_g\}$ of $\cH^1_{\mathrm{dR}}(\cA_{\tilde{S}} /\tilde{S})$, with $\eta_j(\tilde{s})$ being the complex conjugate of $\omega_j(\tilde{s})$ for each $s \in \tilde{S}$.

Moreover, if $\tilde{s} \in \tilde{S}$ satisfies that $\cA_{\tilde{s}}$ is a CM abelian variety, then $E:= \mathrm{End}(\cA_{\tilde{s}})\otimes \QQ$ acts on $\Omega_{\cA_{\tilde{s}}}$ and we can make a choice such that each $\omega_i(\tilde{s})$ is an eigenvector for this action.
\end{contr}
\begin{proof}
The existence of $\{\omega_1,\ldots,\omega_g\}$ holds true since $\tilde{S}$ is simply-connected. If $\cA_{\tilde{s}}$ is a CM abelian variety, then by looking at the CM type, we can make a choice such that each $\omega_i(\tilde{s})$ is an eigenvector for the action of $E$ on $\Omega_{\cA_{\tilde{s}}}$.%The ``Moreover'' part also easily holds true by looking at the CM type.

The set $\{\eta_1,\ldots,\eta_g\}$ can be constructed as follows. Apply \eqref{EqPullbackUnivAbVarSiegel} and \eqref{EqDeRhamOverSiegel} to the dual abelian scheme $\cA^{\mathrm{t}} \to S$. Then we obtain a family of abelian varieties $\cA_{\tilde{S}}^{\mathrm{t}} = \tilde{S} \times_{S} \cA^{\mathrm{t}} \to \tilde{S}$ and a vector bundle $\Omega_{\cA_{\tilde{S}}^{\mathrm{t}}} \to \tilde{S}$. The principal polarization $\lambda \colon \cA \cong \cA^{\mathrm{t}}$ induces an isomorphism $\lambda^* \colon \Omega_{\cA_{\tilde{S}}^{\mathrm{t}}}\cong \Omega_{\cA_{\tilde{S}}}$ of vector bundles on $\tilde{S}$, and hence we obtain a global basis $\{ \lambda^*(\omega_1),\ldots,\lambda^*(\omega_g)\}$ of $\Omega_{\cA_{\tilde{S}}^{\mathrm{t}}}$. For the exact sequence
\begin{equation}\label{EqSESdRAoverS}
0 \to \Omega_{\cA_{\tilde{S}}} \to \cH^1_{\mathrm{dR}}(\cA_{\tilde{S}} /\tilde{S}) \to \Lie(\cA^{\mathrm{t}}_{\tilde{S}} / \tilde{S}) = (\Omega_{\cA_{\tilde{S}}^{\mathrm{t}}})^\vee \to 0,
\end{equation}
$\{2\pi i \lambda^*(\omega_1)^\vee,\ldots,2\pi i \lambda^*(\omega_g)^\vee\}$ lifts to a set of global sections $\{\eta_1,\ldots,\eta_g\}$ of $\cH^1_{\mathrm{dR}}(\cA_{\tilde{S}} /\tilde{S})$. 
%It is easy to check that $\{\omega_1,\ldots,\omega_g, \eta_1,\ldots,\eta_g\}$ is a symplectic global basis of $\cH^1_{\mathrm{dR}}(\cA_{\tilde{S}}/\tilde{S})$.
 For each $s \in \tilde{S}$, use $\rho$ to denote the complex conjugation on $H^1_{\mathrm{dR}}(\cA_{\tilde{s}})$. Then $\eta_j(\tilde{s}) = \rho(\omega_j(\tilde{s}))$ because $\rho(\Omega_{\cA_{\tilde{s}}}) \xrightarrow{\sim} \Lie(\cA^{\mathrm{t}}_{\tilde{s}})$, $\rho(\omega) \mapsto 2\pi i(\lambda^*\omega)^\vee$. We are done.% This is what we desire.
\end{proof}
\begin{rem}
%We close this subsection with the following remark. 
If $s \in S(k)$ and $\tilde{s} \in \tilde{S}$ lies above $s$ (\textit{i.e.} $u(\tilde{s})=s$), then the short exact sequence \eqref{EqSESdRAoverS} restricted over $\tilde{s}$ becomes $0 \to \Omega_{\cA_s} \to H^1_{\mathrm{dR}}(\cA_s) \to \Lie(\cA^{\mathrm{t}}_s) \to 0$, which is 
\textit{defined over $k$}.
\end{rem}

\subsection{Betti (co)homology and symplectic basis}
Let $\underline{\ZZ}_{\cA}$ be the locally constant sheaf of $\ZZ$ on $\cA$. 
Write $R_1 f_*\underline{\ZZ}_{\cA} := (R^1 f_*\underline{\ZZ}_{\cA})^\vee$. 
Then $R_1 f_*\underline{\ZZ}_{\cA}$ is a local system on $S$ such that $(R_1 f_*\underline{\ZZ}_{\cA})_s = H_1(\cA_s, \ZZ)$ for each $s \in S(k)$. Use $\mathbb{V}(\cA/S)$ to denote the vector bundle over $S$ associated with $R_1 f_*\underline{\ZZ}_{\cA}\otimes_{\ZZ} \cO_S$, \textit{i.e.} $\mathbb{V}(\cA/S) = \underline{\mathrm{Spec}}_S\left(\mathrm{Sym}(R_1 f_*\underline{\ZZ}_{\cA} \otimes_{\ZZ} \cO_S) \right)$.

Set
\begin{equation}
R_1\tilde{f}_*\underline{\ZZ} := u^*(R_1 f_*\underline{\ZZ}_{\cA}) \quad \text{and}\quad \mathbb{V}(\cA_{\tilde{S}}/\tilde{S}) := u^*\mathbb{V}(\cA/S).
\end{equation}
Then $R_1\tilde{f}_*\underline{\ZZ}$ is a local system on $\tilde{S}$ with $(R_1\tilde{f}_*\underline{\ZZ})_{\tilde{s}} = H_1(\cA_{\tilde{s}},\ZZ)$ for each $\tilde{s} \in \tilde{S}$, and $\mathbb{V}(\cA_{\tilde{S}}/\tilde{S}) \to \tilde{S}$ is a vector bundle whose sheaf of sections is $R_1\tilde{f}_*\underline{\ZZ} \otimes_{\ZZ} \cO_{\tilde{S}}$.

Since $\tilde{S}$ is simply-connected, we can take a set of \textit{global} sections $\{\gamma_1,\ldots,\gamma_{2g}\}$ of $R_1\tilde{f}_*\underline{\ZZ}$ which is a global basis of $ \mathbb{V}(\cA_{\tilde{S}}/\tilde{S}) \rightarrow \tilde{S}$, \textit{i.e.}  we have the following $\tilde{S}$-isomorphism
\begin{equation}\label{EqLocalSystemIsomBasis}
\CC^{2g}\times \tilde{S} \xrightarrow{\sim} \mathbb{V}(\cA_{\tilde{S}}/\tilde{S})  , \quad \left( (k_1,\ldots,k_{2g}), \tilde{s} \right) \mapsto \sum_{j=1}^{2g} k_j \gamma_j(\tilde{s}).
\end{equation}

The basis $\{\gamma_1,\ldots,\gamma_{2g}\}$ can be furthermore chosen to be \textit{symplectic} in the following sense. The principal polarization on $\cA_{\tilde{S}}$ endows, for each $\tilde{s} \in \tilde{S}$, $(R_1\tilde{f}_*\underline{\ZZ})_{\tilde{s}} = H_1(\cA_{\tilde{s}},\ZZ)$ with a symplectic form $\Psi_{\tilde{s}}$, which furthermore induces a pairing on $\ZZ^{2g}$ via $\ZZ^{2g} \xrightarrow{\sim} H_1(\cA_{\tilde{s}},\ZZ)$, $(k_1,\ldots,k_{2g}) \mapsto \sum_{j=1}^{2g}k_j \gamma_j(\tilde{s})$. 
%$(R^1\tilde{f}_*\underline{\ZZ})_{\tilde{s}} = H^1(\cA_{\tilde{s}},\ZZ)$ with a symplectic form $\Psi_{\tilde{s}}$. 
 We say that  $\{\gamma_1,\ldots,\gamma_{2g}\}$ is \textit{symplectic} if this induced pairing on $\ZZ^{2g}$ is $\begin{bmatrix} 0 & I_g \\ -I_g & 0 \end{bmatrix}$ for each $\tilde{s} \in \tilde{S}$.

Finally, notice that the natural differential map $\cO_{\tilde S} \rightarrow \Omega_{\tilde S}^1$ gives a natural map $R_1\tilde{f}_*\underline{\ZZ} \otimes \cO_{\tilde{S}} \to R_1\tilde{f}_*\underline{\ZZ} \otimes \Omega_{\tilde{S}}^1$, which in turn becomes a connection on $\mathbb{V}(\cA_{\tilde S}/\tilde S)^\vee $. Denote this. connection by  $\mathrm{d}$.

\subsection{Gauss--Manin connection}
Taking exterior powers, the canonical short exact sequence \eqref{EqRelative1Form} yields a decreasing filtration $\Omega_{\cA}^{\bullet} = \mathrm{Fil}^0 \Omega_{\cA}^{\bullet} \supseteq \mathrm{Fil}^1 \Omega_{\cA}^{\bullet} \supseteq \cdots $ on $\Omega_{\cA}^{\bullet}$ %given by
%\[
%\mathrm{Fil}^r \Omega_{\cA}^{\bullet} := \mathrm{Image}(f^*\Omega_S^r \otimes_{\cO_{\cA}} \Omega_{\cA}^{\bullet -r} \to \Omega_{\cA}^{\bullet}).
%\]
such that $\mathrm{Fil}^r \Omega_{\cA}^{\bullet}  / \mathrm{Fil}^{r+1}\Omega_{\cA}^{\bullet}  =  f^*\Omega_S^r \otimes_{\cO_{\cA}} \Omega_{\cA/S}^{\bullet -r}$.
 Thus from $0 \to \mathrm{Fil}^1/\mathrm{Fil}^2 \to \mathrm{Fil}^0 / \mathrm{Fil}^2 \to \mathrm{Fil}^0 / \mathrm{Fil}^1 \to 0$, we obtain
\[
0 \to f^*\Omega_S^1 \otimes_{\cO_{\cA}} \Omega_{\cA/S}^{\bullet -1} \to \Omega_{\cA}^{\bullet} / \mathrm{Fil}^2 \to \Omega_{\cA/S}^{\bullet} \to 0.
\]
Thus we have a connection map
\[
R^1f_* \Omega_{\cA/S}^{\bullet}  \to R^2f_* (f^*\Omega_S^1 \otimes_{\cO_{\cA}} \Omega_{\cA/S}^{\bullet -1} ) =  \Omega_S^1 \otimes R^1f_*  \Omega_{\cA/S}^{\bullet}.
\]
Recall the definition of the relative de Rham cohomology $\cH^1_{\mathrm{dR}}(\cA/S) = R^1f_* \Omega_{\cA/S}^{\bullet} $. Then the co-boundary map above is precisely the \textit{Gauss--Manin connection}
\[
\nabla_{\mathrm{GM}} \colon \cH^1_{\mathrm{dR}}(\cA/S)  \to \cH^1_{\mathrm{dR}}(\cA/S)  \otimes \Omega_S^1.
\]
By abuse of notation, we also use $\nabla_{\mathrm{GM}} $ to denote the Gauss--Manin connection $\cH^1_{\mathrm{dR}}(\cA_{\tilde{S}} /\tilde{S})  \to \cH^1_{\mathrm{dR}}(\cA_{\tilde{S}} /\tilde{S})  \otimes \Omega_{\tilde{S}}^1$.

\subsection{Comparison of de Rham and Betti cohomologies}\label{SubsectionComparisonDeRhamBetti}
%The relative de Rham cohomology $\cH^1_{\mathrm{dR}}(\cA_{\tilde{S}} /\tilde{S})$ and the family version of Betti cohomology $R^1\tilde{f}_*\underline{\ZZ}$ are linked via the Gauss--Manin connection $\nabla$.
 By the  theory of vector bundles with connections, there is an isomorphism of vector bundles over $\tilde{S}$ 
\begin{equation}\label{EqComparisonDeRhamBetti}
\beta \colon (\cH^1_{\mathrm{dR}}(\cA_{\tilde{S}} /\tilde{S}) , \nabla_{\mathrm{GM}}) \xrightarrow{\sim} ( \mathbb{V}(\cA_{\tilde{S}}/\tilde{S})^\vee , \mathrm{d}).
\end{equation}
Set $\Omega_1(\tilde{s}) = \begin{bmatrix}\int_{\gamma_l} \omega_j(\tilde{s})\end{bmatrix}_{1\le j,l\le g}$, $\Omega_2(\tilde{s}) = \begin{bmatrix}\int_{\gamma_{g+l}} \omega_j(\tilde{s})\end{bmatrix}_{1\le j,l\le g}$, $N_1(\tilde{s}) = \begin{bmatrix}\int_{\gamma_l} \eta_j(\tilde{s})\end{bmatrix}_{1 \le j,l\le g}$, and $N_2(\tilde{s}) = \begin{bmatrix}\int_{\gamma_{g+l}} \eta_j(\tilde{s})\end{bmatrix}_{1\le j,l\le g}$.
Under the global basis $\{\omega_1,\ldots,\omega_g, \eta_1,\ldots,\eta_g\}$ of $\cH^1_{\mathrm{dR}}(\cA_{\tilde{S}} /\tilde{S})$ and the global basis $\{\gamma_1^{\vee},\ldots,\gamma_{2g}^{\vee}\}$ of $\mathbb{V}(\cA_{\tilde{S}}/\tilde{S})^\vee$, the isomorphism $\beta$ is represented by the matrix 
\begin{equation}\label{EqBeta}
\beta_{\tilde{s}} := \begin{bmatrix} \Omega_1(\tilde{s}) & N_1(\tilde{s}) \\ \Omega_2(\tilde{s}) & N_2(\tilde{s})  \end{bmatrix}
\end{equation}
over each $\tilde{s} \in \tilde{S}$, \textit{i.e.} the following diagram (of $\tilde{S}$-morphisms) commutes
\begin{equation}\label{EqMatrixComparison}
\xymatrix{
\cH^1_{\mathrm{dR}}(\cA_{\tilde{S}} /\tilde{S}) \ar[r]^-{\beta}_-{\sim}  & \mathbb{V}(\cA_{\tilde{S}}/\tilde{S})^\vee \\
\CC^{2g} \times \tilde{S} \ar[r]^-{\sim}_{(\mathbf{v},\tilde{s})\mapsto (\beta_{\tilde{s}}\mathbf{v}, \tilde{s})} \ar[u]_-{\cong} & \CC^{2g} \times \tilde{S} \ar[u]^-{\cong} \\
%(\mathbf{v},\tilde{s}) \ar@{|->}[r]  & (\beta_{\tilde{s}}\mathbf{v}, \tilde{s})
}
\end{equation}
where the left isomorphism is by sending $((l_1,\ldots,l_g, l_{g+1},\ldots,l_{2g}),\tilde{s}) \mapsto \sum_{j=1}^g l_j \omega_j(\tilde{s}) + \sum_{j=1}^g l_{g+j} \eta_j(\tilde{s})$, and the right isomorphism is defined by sending $((k_1,\ldots,k_{2g}),\tilde{s}) \mapsto \sum_{j=1}^{2g} k_j\gamma_j^{\vee}(\tilde{s})$.

%\subsection{A discussion}
%This subsection is not used in later discussion.

We take a closer look at this comparison. 
Let $\tilde{s} \in \tilde{S}$. Set $\tau_{\tilde{s}}:=\Omega_2(\tilde{s})\Omega_1(\tilde{s})^{-1}$. Under the basis $\{\gamma_1^{\vee}(\tilde{s}),\ldots,\gamma_{2g}^{\vee}(\tilde{s})\}$ of $\mathbb{V}(\cA_{\tilde{S}}/\tilde{S})^\vee_{\tilde{s}} = H^1(\cA_{\tilde{s}},\CC)$, we have
%Then under the basis $\{\omega_1(\tilde{s}),\ldots,\omega_g(\tilde{s}), \eta_1(\tilde{s}),\ldots,\eta_g(\tilde{s})\}$ of $H^1_{\mathrm{dR}}(\cA_{\tilde{s}})$ and the basis $\{\gamma_1^{\vee}(\tilde{s}),\ldots,\gamma_{2g}^{\vee}(\tilde{s})\}$ of $\mathbb{V}(\cA_{\tilde{S}}/\tilde{S})^\vee_{\tilde{s}} = H^1(\cA_{\tilde{s}},\CC)$, we have 
 $\beta_{\tilde{s}}(\Omega_{\cA_{\tilde{s}}}) =\{(\Omega_1(\tilde{s}) \bx, \Omega_2(\tilde{s}) \bx) : \bx \in \CC^g\}$. It is known that $\Omega_1(\tilde{s})$ is invertible for each $\tilde{s} \in \tilde{S}$. Hence
\begin{equation}\label{EqImageOmegaUnderComparison}
\beta_{\tilde{s}}(\Omega_{\cA_{\tilde{s}}}) =\{( \bx, \Omega_2(\tilde{s}) \Omega_1(\tilde{s})^{-1}\bx) : \bx \in \CC^g\} = \{(\bx, \tau_{\tilde{s}} \cdot \bx) : \bx \in \CC^g\}.
\end{equation}
More precisely, $\beta_{\tilde{s}}(\Omega_{\cA_{\tilde{s}}})$ is the subspace of  $\bigoplus_{j=1}^{2g} \CC \gamma_j^\vee(\tilde{s})$ consisting of the vectors of the form $\begin{bmatrix} \gamma_1^{\vee}(\tilde{s}) & \cdots & \gamma_g^{\vee}(\tilde{s})  \end{bmatrix} \bx + \begin{bmatrix} \gamma_{g+1}^{\vee}(\tilde{s})  & \cdots & \gamma_{2g}^{\vee}(\tilde{s})  \end{bmatrix}  \tau_{\tilde{s}} \cdot \bx$, with $\bx$ running over all (column) vectors in $\CC^g$. Taking $\bx$ to be the vector with $1$ on the $j$-th entry and $0$ elsewhere with $j$ running over $\{1,\ldots,g\}$, we obtain a basis of $\beta_{\tilde{s}}(\Omega_{\cA_{\tilde{s}}})$.
Notice that this basis is precisely the columns vectors of $\begin{bmatrix} \gamma_1^{\vee}(\tilde{s}) & \cdots & \gamma_g^{\vee}(\tilde{s})  \end{bmatrix} + \begin{bmatrix} \gamma_{g+1}^{\vee}(\tilde{s})  & \cdots & \gamma_{2g}^{\vee}(\tilde{s})  \end{bmatrix}  \! \tau_{\tilde{s}}$. The comparison $\beta_{\tilde{s}}$ induces the following commutative diagram% (here we used the prinicpal polarization $\cA_{\tilde{s}} \cong \cA_{\tilde{s}}^{\mathrm{t}}$)
\begin{equation}\label{EqCommutativeDiagramCMAbVar}
\xymatrix{
0 \ar[r] & \Omega_{\cA_{\tilde{s}}} \ar[r] \ar[d] & H^1_{\mathrm{dR}}(\cA_{\tilde{s}}) \ar[r] \ar[d]^-{\beta_{\tilde{s}}} & \Lie(\cA^{\mathrm{t}}_{\tilde{s}}) = \Omega_{\cA^{\mathrm{t}}_{\tilde{s}}}^\vee \ar[d]^-{\bar{\beta}_{\tilde{s}}}  \ar[r] & 0 \\
0 \ar[r] & \beta_{\tilde{s}}(\Omega_{\cA_{\tilde{s}}}) \ar[r] & H^1(\cA_{\tilde{s}}, \CC) \ar[r] & H^1(\cA_{\tilde{s}}, \CC) / \beta_{\tilde{s}}(\Omega_{\cA_{\tilde{s}}}) \ar[r] & 0
}
\end{equation}
For each $j \in \{1,\ldots,2g\}$, write $\overline{\gamma}_j^\vee(\tilde{s})$ for the image of $\gamma_j^\vee(\tilde{s})$ under quotient $H^1(\cA_{\tilde{s}}, \CC) \to H^1(\cA_{\tilde{s}}, \CC) / \beta_{\tilde{s}}(\Omega_{\cA_{\tilde{s}}})$. Then $\{\overline{\gamma}_{g+1}^{\vee}(\tilde{s}) ,\ldots,\overline{\gamma}_{2g}^{\vee}(\tilde{s}) \}$ is a basis of $H^1(\cA_{\tilde{s}}, \CC) / \beta_{\tilde{s}}(\Omega_{\cA_{\tilde{s}}})$.%Then the quotient basis of $H^1(\cA_{\tilde{s}}, \CC) / \beta_{\tilde{s}}(\Omega_{\cA_{\tilde{s}}})$, with respect to the basis of $\beta_{\tilde{s}}(\Omega_{\cA_{\tilde{s}}})$ chosen above, is $\{\overline{\gamma}_{g+1}^{\vee}(\tilde{s}) ,\ldots,\overline{\gamma}_{2g}^{\vee}(\tilde{s}) \}$.

The following lemma follows from a direct computation.
\begin{lem}\label{LemmaCMAbVarDiagonal1}
We have
\begin{enumerate}
\item[(i)] Under the basis $\{\omega_1(\tilde{s}), \ldots, \omega_g(\tilde{s})\}$ of $\Omega_{\cA_{\tilde{s}}}$ and the (natural choice of) basis of $\beta_{\tilde{s}}(\Omega_{\cA_{\tilde{s}}})$ below \eqref{EqImageOmegaUnderComparison}, the matrix for $\beta_{\tilde{s}}|_{\Omega_{\cA_{\tilde{s}}}}$ is $\Omega_1(\tilde{s})$. 
\item[(ii)] Under the basis $\{\overline{\eta}_1(\tilde{s}), \ldots,\overline{\eta}_g(\tilde{s})\}$ of $\Lie(\cA^{\mathrm{t}}_s)$ and the basis $\{\overline{\gamma}_{g+1}^{\vee}(\tilde{s}) ,\ldots,\bar{\gamma}_{2g}^{\vee}(\tilde{s}) \}$ of $H^1(\cA_{\tilde{s}}, \CC) / \beta_{\tilde{s}}(\Omega_{\cA_{\tilde{s}}})$, the matrix for $\bar{\beta}_{\tilde{s}}$ is $N_2(\tilde{s}) - \Omega_2(\tilde{s})\Omega_1(\tilde{s})^{-1}N_1(\tilde{s}) = N_2(\tilde{s}) - \tau_{\tilde{s}} N_1(\tilde{s})$.
\end{enumerate}
\end{lem}
Recall from Construction~\ref{LemmaGlobalBasisSympDeRham} that $\overline{\eta}_j$ is $(\lambda^*\omega_j)^\vee$, where $\lambda \colon \cA_{\tilde{S}} \xrightarrow{\sim} \cA^{\mathrm{t}}_{\tilde{S}}$ is the principal polarization.

\begin{proof} We only prove part (i). Part (ii) follows from a similar computation.

From \eqref{EqBeta}, we have
\begin{align*}
\begin{bmatrix} \beta_{\tilde{s}}(\omega_1(\tilde{s})) & \cdots & \beta_{\tilde{s}}(\omega_g(\tilde{s}))  \end{bmatrix}  & = \begin{bmatrix} \gamma_1^\vee(\tilde{s}) & \cdots & \gamma_{2g}^\vee(\tilde{s})  \end{bmatrix} \cdot \begin{bmatrix} \Omega_1(\tilde{s}) \\ \Omega_2(\tilde{s}) \end{bmatrix} \\
& = \begin{bmatrix} \gamma_1^\vee(\tilde{s}) & \cdots & \gamma_g^\vee(\tilde{s})  \end{bmatrix} \Omega_1(\tilde{s}) + \begin{bmatrix} \gamma_{g+1}^\vee(\tilde{s}) & \cdots & \gamma_{2g}^\vee(\tilde{s})  \end{bmatrix} \Omega_2(\tilde{s}) \\
& = \begin{bmatrix} \gamma_1^\vee(\tilde{s}) & \cdots & \gamma_g^\vee(\tilde{s})  \end{bmatrix} \Omega_1(\tilde{s}) + \begin{bmatrix} \gamma_{g+1}^\vee(\tilde{s}) & \cdots & \gamma_{2g}^\vee(\tilde{s})  \end{bmatrix} \tau_{\tilde{s}} \Omega_1(\tilde{s})  \\
& = \left(\begin{bmatrix} \gamma_1^{\vee}(\tilde{s}) & \cdots & \gamma_g^{\vee}(\tilde{s})  \end{bmatrix} + \begin{bmatrix} \gamma_{g+1}^{\vee}(\tilde{s})  & \cdots & \gamma_{2g}^{\vee}(\tilde{s})  \end{bmatrix}  \! \tau_{\tilde{s}} \right) \cdot \Omega_1(\tilde{s}).
\end{align*}
We are done.
\end{proof}

\subsection{Over CM fibers}
In this subsection, we assume $k = \IQbar$. Fix a point $\tilde{s} \in \tilde{S}$ such that $\cA_{\tilde{s}}$ is a CM abelian variety. Set $E:=\mathrm{End}(\cA_{\tilde{s}})\otimes \QQ$.

All CM abelian varieties over $\CC$ are defined over $\IQbar$. So  $s:=u(\tilde{s}) \in S(\IQbar)$. Moreover, it is known that  $\tau_{\tilde{s}} \in \mathrm{Mat}_{g\times g}(\IQbar)$. The basis $\{\omega_1(\tilde{s}),\ldots,\omega_g(\tilde{s}),\eta_1(\tilde{s}),\ldots,\eta_g(\tilde{s})\}$ is by choice a $\IQbar$-eigenbasis of $H^1_{\mathrm{dR}}(\cA_{\tilde{s}})$ for the action of $E$, and $\{\gamma_1^\vee(\tilde{s}),\ldots,\gamma_{2g}^\vee(\tilde{s})\}$ is a $\IQbar$-basis of $H^1(\cA_{\tilde{s}},\IQbar)$.

Let $\theta_1,\ldots,\theta_g$ be the holomorphic periods of $\cA_{\tilde{s}}$.%Set, for each $j \in \{1,\ldots,g\}$, $\theta_j := \int_{\gamma}\omega_j(\tilde{s})$ for any non-zero $\gamma \in H_1(\cA_{\tilde{s}},\ZZ)$; then $\theta_j$ is well-defined up to $\IQbar^*$ by Shimura \cite[]{}.

Since $\tau_{\tilde{s}} \in \mathrm{Mat}_{g\times g}(\IQbar)$, the subspace  $\beta_{\tilde{s}}(\Omega_{\cA_{\tilde{s}}}) = \{(\bx, \tau_{\tilde{s}}  \bx) : \bx \in \CC^g\}$ of $ H^1(\cA_{\tilde{s}},\CC) = \oplus\CC\gamma_j^\vee(\tilde{s})$ is rational for the $\IQbar$-structure given by $H^1(\cA_{\tilde{s}},\IQbar)\otimes \CC$. Thus $H^1(\cA_{\tilde{s}},\IQbar)$ gives a $\IQbar$-structure on $H^1(\cA_{\tilde{s}}, \CC) / \beta_{\tilde{s}}(\Omega_{\cA_{\tilde{s}}}) $, for which $\{\overline{\gamma}_{g+1}^\vee(\tilde{s}),\ldots,\overline{\gamma}_{2g}^\vee(\tilde{s})\}$ is a $\IQbar$-basis. %\footnote{\label{Footnote}One can compute that the image of $H^1(\cA_{\tilde{s}},\ZZ) \subseteq H^1(\cA_{\tilde{s}},\CC) \to H^1(\cA_{\tilde{s}}, \CC) / \beta_{\tilde{s}}(\Omega_{\cA_{\tilde{s}}})$ is, under the basis $\{\overline{\gamma}_{g+1}^\vee(\tilde{s}),\ldots,\overline{\gamma}_{2g}^\vee(\tilde{s})\}$, is $\ZZ^g -\tau_{\tilde{s}}\ZZ^g = \ZZ^g + \tau_{\tilde{s}}\ZZ^g$. Hence this is the geometric $\IQbar$-structure on $\CC^g$ defined in $\mathsection$\ref{SubsectionExampleCMAV}.} 
Moreover, the action of $E$ on $H^1(\cA_{\tilde{s}},\IQbar)$ induces an action of $E$ on $H^1(\cA_{\tilde{s}}, \CC) / \beta_{\tilde{s}}(\Omega_{\cA_{\tilde{s}}}) $ which preserves its $\IQbar$-structure.  Therefore we can find a $\IQbar$-eigenbasis $\{\mathbf{f}_1,\ldots,\mathbf{f}_g\}$ of $H^1(\cA_{\tilde{s}},\CC) / \beta_{\tilde{s}}(\Omega_{\cA_{\tilde{s}}})$ with respect to the action of $E$.

\begin{lem}\label{LemmaCMAbVarDiagonal}
Under the $\IQbar$-basis $\{\overline{\eta}_1(\tilde{s}), \ldots,\overline{\eta}_g(\tilde{s})\}$ of $\Lie(\cA^{\mathrm{t}}_{\tilde{s}})$ and a choice of the $\IQbar$-eigenbasis $\{\mathbf{f}_1,\ldots,\mathbf{f}_g\}$ of $H^1(\cA_{\tilde{s}},\CC) / \beta_{\tilde{s}}(\Omega_{\cA_{\tilde{s}}})$ as above, the matrix for $\overline{\beta}_{\tilde{s}}$ is $\mathrm{diag}(\pi\theta_1^{-1}, \ldots, \pi\theta_g^{-1})$.% up to multiplication by a diagonal matrix in $\mathrm{Mat}_{g \times g}(\IQbar)$.
\end{lem}
%In view of the footnote~\ref{Footnote} above and Remark~\ref{RemarkDualAVPeriods}, this gives another proof of \eqref{EqHolomorphicPeriodsComputation}.
\begin{proof} 
Since both bases are eigenbases for the actions of $E$ and the linear map $\overline{\beta}_{\tilde{s}}$ is $E$-equivariant (because $\beta_{\tilde{s}}$ is $E$-equivariant), the matrix for $\overline{\beta}_{\tilde{s}}$ under these bases is diagonal up to reordering $\mathbf{f}_1,\ldots,\mathbf{f}_g$. Call this matrix $D$.% Hence it suffices to prove that the eigenvalues of $\overline{\beta}_{\tilde{s}}$ are $\pi \theta_1^{-1}, \ldots, \pi\theta_g^{-1}$ up to multiplications by numbers in $\IQbar$.

Let us look at the matrix $N_2(\tilde{s}) - \tau_{\tilde{s}} N_1(\tilde{s})$. By our choice of the basis, $\eta_j(\tilde{s})$ is the complex conjugate of $\omega_j(\tilde{s})$ for each $j \in \{1,\ldots,g\}$. Since $\omega_1,\ldots,\omega_g,\eta_1,\ldots,\eta_g$ are eigenforms for the CM action, the reciprocity law for the differential forms of the 1st and the 2nd spaces implies that $\int_{\gamma} \eta_j(\tilde{s}) \cong 2\pi i / \theta_j$ for any $\gamma \in H_1(\cA_{\tilde{s}},\ZZ)$; see \cite[pp.36, equation (3)]{BertrandReciprocityDiffForm} for more details. 
So both $N_1(\tilde{s})$ and $N_2(\tilde{s})$ are in $\mathrm{Mat}_{g \times g}(\IQbar) \cdot \mathrm{diag}(\pi\theta_1^{-1}, \ldots, \pi\theta_g^{-1})$. So  $N_2(\tilde{s}) - \tau_{\tilde{s}} N_1(\tilde{s})$ is in $\mathrm{Mat}_{g \times g}(\IQbar) \cdot \mathrm{diag}(\pi\theta_1^{-1}, \ldots, \pi\theta_g^{-1})$ since $\tau_{\tilde{s}} \in \mathrm{Mat}_{g \times g}(\IQbar)$, \textit{i.e.} 
\[
N_2(\tilde{s}) - \tau_{\tilde{s}}N_1(\tilde{s}) = M\cdot  \mathrm{diag}(\pi\theta_1^{-1}, \ldots, \pi\theta_g^{-1})
\]
for some $M \in \mathrm{Mat}_{g \times g}(\IQbar)$. Notice that $M$ is an invertible matrix, since $\overline{\beta}_{\tilde{s}}$ is an isomorphism of $\CC$-vector spaces.

Set  $ \begin{bmatrix} v_1(\tilde{s})  & \cdots & v_g(\tilde{s})  \end{bmatrix}  := \begin{bmatrix} \overline{\gamma}_{g+1}^{\vee}(\tilde{s})  & \cdots & \overline{\gamma}_{2g}^{\vee}(\tilde{s})  \end{bmatrix} M$. By Lemma~\ref{LemmaCMAbVarDiagonal1}.(ii), 
under the basis $\{\overline{\eta}_1(\tilde{s}), \ldots,\overline{\eta}_g(\tilde{s})\}$ of $\Lie(\cA^{\mathrm{t}}_s)$ and the basis $\{v_1(\tilde{s}) ,\ldots,v_g(\tilde{s}) \}$ of $H^1(\cA_{\tilde{s}}, \CC) / \beta_{\tilde{s}}(\Omega_{\cA_{\tilde{s}}})$, the matrix for $\bar{\beta}_{\tilde{s}}$ is $M\cdot \mathrm{diag}(\pi\theta_1^{-1}, \ldots, \pi\theta_g^{-1})$.%is similar to (and hence has the same eigenvalues as) 
%and thus its eigenvalues are $\pi\theta_1^{-1},\ldots,\pi\theta_g^{-1}$ up to multiplication by $\IQbar$. We are done.

Now we have two $\IQbar$-bases $\{\mathbf{f}_1,\ldots,\mathbf{f}_g\}$ and  $\{v_1(\tilde{s}) ,\ldots,v_g(\tilde{s}) \}$ of $H^1(\cA_{\tilde{s}}, \CC) / \beta_{\tilde{s}}(\Omega_{\cA_{\tilde{s}}})$. Let $T \in \mathrm{GL}_g(\IQbar)$ be the transition matrix. Then $D = (TM)\cdot \mathrm{diag}(\pi\theta_1^{-1}, \ldots, \pi\theta_g^{-1})$. So $TM$ is diagonal because both $D$ and $\mathrm{diag}(\pi\theta_1^{-1}, \ldots, \pi\theta_g^{-1})$ are diagonal. Thus we can conclude by replacing each $\mathbf{f}_j$ by a suitable $\IQbar$-multiple.
\end{proof}

\begin{rem}\label{RemarkCMAbVarDiagonal}
Similar to the proof of Lemma~\ref{LemmaCMAbVarDiagonal} and using Lemma~\ref{LemmaCMAbVarDiagonal1}.(i), we can prove the following assertion. Under the basis $\{\omega_1(\tilde{s}), \ldots, \omega_g(\tilde{s})\}$ of $\Omega_{\cA_{\tilde{s}}}$ and a $\IQbar$-eigenbasis $\{\mathbf{e}_1,\ldots,\mathbf{e}_g\}$ of $\beta_{\tilde{s}}(\Omega _{\cA_{\tilde{s}}})$, the matrix for $\beta_{\tilde{s}}|_{\Omega_{\cA_{\tilde{s}}}}$ is $\mathrm{diag}(\theta_1,\ldots,\theta_g)$.% up to multiplication by a diagonal matrix in $\mathrm{Mat}_{g\times g}(\IQbar)$.

Therefore the matrix for $\beta_{\tilde{s}}$ is $\mathrm{diag}(\theta_1,\ldots,\theta_g, \pi\theta_1^{-1}, \ldots, \pi\theta_g^{-1})$ under suitable $\IQbar$-bases of $H^1_{\mathrm{dR}}(\cA_{\tilde{s}})$ and $H^1(\cA_{\tilde{s}},\CC)$.
\end{rem}

\subsection{Kodaira--Spencer map}

Consider the maps defined over the field $k$
\scriptsize
\begin{equation}\label{EqDefnKS}
\Omega_{\cA} \subseteq \cH^1_{\mathrm{dR}}(\cA/S) \xrightarrow{\nabla_{\mathrm{GM}} } \cH^1_{\mathrm{dR}}(\cA/S) \otimes_{\cO_S} \Omega_S^1 \to \cH^1_{\mathrm{dR}}(\cA/S)/\Omega_{\cA} \otimes_{\cO_S} \Omega_S^1 \cong \Lie(\cA^{\mathrm{t}}/S) \otimes_{\cO_S} \Omega_S^1.
\end{equation}
\normalsize
Thus for the tangent bundle $TS = (\Omega_S^1)^\vee$ we have
\begin{equation}\label{EqKSOrig}
TS \to \Lie(\cA^{\mathrm{t}}/S) \otimes_{\cO_S} \Lie(\cA/S).
\end{equation}
%We can do better. 
Using the principal polarization $\lambda \colon \cA \cong \cA^{\mathrm{t}}$, the right hand side is isomorphic (over $\cO_S$) to $ \Lie(\cA/S)  \otimes_{\cO_S} \Lie(\cA/S)= \mathrm{Sym}^2_{\cO_S}\Lie(\cA/S) \bigoplus \mathrm{Alt}^2_{\cO_S} \Lie(\cA/S)$. We introduce the following (non-standard) notation:
\begin{equation}
\mathrm{S}^2 \Lie(\cA/S) := (\lambda_*,1)\left( \mathrm{Sym}^2_{\cO_S}\Lie(\cA/S) \right) \subseteq  \Lie(\cA^{\mathrm{t}}/S) \otimes_{\cO_S} \Lie(\cA/S).
\end{equation}
It is known that the image of \eqref{EqKSOrig} lies in $\mathrm{S}^2 \Lie(\cA/S)$, \textit{i.e.} we have
\begin{equation}\label{EqKSFamily}
\mathrm{KS} \colon TS \to \mathrm{S}^2 \Lie(\cA/S) \subseteq  \Lie(\cA^{\mathrm{t}}/S) \otimes_{\cO_S} \Lie(\cA/S) .
\end{equation}
Let $s \in S(k)$. Then \eqref{EqKSFamily} restricted over $s$ becomes
\begin{equation}\label{EqKS}
\mathrm{KS}_s \colon T_s S \to \mathrm{S}^2\Lie(\cA_s) \subseteq \Lie \cA^{\mathrm{t}}_s \otimes \Lie \cA_s.
\end{equation}

\vskip 0.3em

We close this section by the following discussion, which will be used in the computation of the periods of the Siegel modular varieties. 

For the universal covering $u \colon \tilde{S} \rightarrow S$, the sequence of maps \eqref{EqDefnKS} exists with $\cA/S$ replaced by $\cA_{\tilde{S}}/\tilde{S}$. Then combined with the isomorphism $\beta \colon (\cH^1_{\mathrm{dR}}(\cA_{\tilde{S}} /\tilde{S}) , \nabla_{\mathrm{GM}}) \xrightarrow{\sim} ( \mathbb{V}(\cA_{\tilde{S}}/\tilde{S})^\vee , \mathrm{d})$ from \eqref{EqComparisonDeRhamBetti}, we obtain
\scriptsize
\begin{equation}\label{EqKSTwoComparison1}
\xymatrix{
\Omega_{\cA_{\tilde S}} \ar@{^(->}[r] \ar[d] & \cH^1_{\mathrm{dR}}(\cA_{\tilde S}/\tilde{S}) \ar[r]^-{\nabla_{\mathrm{GM}} } \ar[d]_-{\beta} & \cH^1_{\mathrm{dR}}(\cA_{\tilde S}/\tilde{S}) \otimes \Omega_{\tilde S}^1 \ar[r] \ar[d]^-{\beta \otimes \mathrm{id}}& \cH^1_{\mathrm{dR}}(\cA_{\tilde S}/\tilde{S})/\Omega_{\cA_{\tilde S}} \otimes\Omega_{\tilde S}^1 = \Lie(\cA^{\mathrm{t}}_{\tilde S}) \otimes \Omega_{\tilde S}^1 \ar[d]_-{\bar{\beta}\otimes \mathrm{id}} \\%\otimes\mathrm{id}} \\
\beta(\Omega_{\cA_{\tilde S}}) \ar@{^(->}[r]  & \mathbb{V}(\cA_{\tilde S}/\tilde{S})^\vee \ar[r]^-{\mathrm{d}} & \mathbb{V}(\cA_{\tilde S}/\tilde{S})^\vee \otimes  \Omega_{\tilde S}^1 \ar[r] & \mathbb{V}(\cA_{\tilde S}/\tilde{S})^\vee / \beta(\Omega_{\cA_{\tilde S}}) \otimes \Omega_{\tilde S}^1,
}
\end{equation}
\normalsize
from which we get the commutative diagram
\begin{equation}\label{EqKSTwoComparison2}
\xymatrix{
T\tilde{S} \ar[r]^-{\mathrm{KS}\circ \mathrm{d}u} \ar[d]_{=} & \Lie(\cA_{\tilde{S}}^{\mathrm{t}}) \otimes \Lie(\cA_{\tilde{S}}) \ar[d]^-{\bar{\beta}\otimes \beta^\vee} \\
T\tilde{S} \ar[r] & \frac{\mathbb{V}(\cA_{\tilde S}/\tilde{S})^\vee}{\beta(\Omega_{\cA_{\tilde S}})} \otimes \beta(\Omega_{\cA_{\tilde S}})^\vee.
}
\end{equation}
Take $\tilde{s} \in u^{-1}(s)$. Then \eqref{EqKSTwoComparison2} yields, after the natural identification $\cA_{\tilde{s}} = \cA_s$ under $u$, 
\begin{equation}\label{EqKSTwoComparison}
\xymatrix{
T_s S \ar[r]^-{\mathrm{KS}_s}  & \Lie(\cA_{\tilde{s}}^{\mathrm{t}}) \otimes \Lie(\cA_{\tilde{s}}) \ar[d]^-{\bar{\beta}_{\tilde{s}} \otimes \beta_{\tilde{s}}^\vee} \\
T_{\tilde{s}} \tilde{S} \ar[u]^{\mathrm{d}u}_{\cong} \ar[r] & \frac{H^1(\cA_{\tilde{s}},\CC)}{\beta_{\tilde{s}}(\Omega_{\cA_{\tilde s}})} \otimes \beta_{\tilde{s}}(\Omega_{\cA_{\tilde s}})^\vee.
}
\end{equation}

\section{Periods of the Siegel modular variety}\label{SectionPeriodSiegel}
In this section, we turn to the Siegel case. 
Let $\mathbb{A}_g$ be the moduli space of principally polarized abelian varieties of dimension $g$. The associated connected Shimura datum is $(\GSp_{2g}, \mathfrak{H}_g)$, where $\mathfrak{H}_g $ is the Siegel upper half space $\left\{ \tau \in \mathrm{Mat}_{g\times g}(\CC) : \tau = \tau^{\!{\intercal}}, ~ \mathrm{Im}\tau > 0 \right\}$.% with a level-$\ell$-structure.

Let $o \in \mathfrak{H}_g$ be a CM point. The goal of this section is to compute the periods of the (split) bi-$\IQbar$-structure on $T_o \mathfrak{H}_g$ defined in $\mathsection$\ref{SectionBiIQbarShimura}, or more precisely in $\mathsection$\ref{SubsectionArithmeticIQbarShimura} and $\mathsection$\ref{SubsectionGeometricIQbarShimura}. 

%To state the conclusion, we introduce the following notation. 
The abelian variety $A_{o} := \CC^g / (\ZZ^g + o \ZZ^g)$ is a CM abelian variety. Let $\theta_1,\ldots,\theta_g$ be its holomorphic periods as defined in Convention~\ref{SubsectionConvention}. 

\begin{teo}\label{ThmPeriodsSiegel}
The periods of the (split) bi-$\IQbar$-structure on $T_{o}\mathfrak{H}_g$ defined in $\mathsection$\ref{SectionBiIQbarShimura} are $\theta_j\theta_{j'}/\pi$ with $1\le j \le j' \le g$.
\end{teo}

Our proof of Theorem~\ref{ThmPeriodsSiegel} uses the Kodaira--Spencer map and does not use the root space decompositions discussed in $\mathsection$\ref{SubsectionDecompositionGeomShimura} and $\mathsection$\ref{SubsectionDecompositionArithShimura}. At the end of this section, we will show in $\mathsection$\ref{SubsectionTwoBiIQbarDecomSame} that the computation with the Kodaira--Spencer map is compatible with the root space decomposition.

We have the following more precise formulation of Theorem~\ref{ThmPeriodsSiegel} 
%A closer way to look at Theorem~\ref{ThmPeriodsSiegel} is as follows. It uses 
in view of the language of split bi-$\IQbar$-structures introduced in the current paper.  By dimension comparison, the Kodaira--Spencer map \eqref{EqKS} applied to the universal abelian variety $\mathfrak{A}_g \to \mathbb{A}_g$ at the point $[o] \in \mathbb{A}_g(\IQbar)$ gives an isomorphism $\mathrm{KS}_{[o]}  \colon T_{[o]} \mathbb{A}_g \cong \mathrm{S}^2 \Lie(A_{o})$ over $\IQbar$. Thus for the uniformization $u \colon \mathfrak{H}_g \rightarrow \mathbb{A}_g$, the composite 
\[
\mathrm{KS}_o \colon T_o \mathfrak{H}_g \xrightarrow{\mathrm{d}u} T_{[o]}\mathbb{A}_g \xrightarrow[\mathrm{KS}_{[o]}]{\sim} \mathrm{S}^2 \Lie A_o
\]
 is an isomorphism of $\CC$-vector spaces. The split bi-$\IQbar$-structure on $\Lie A_o \cong \CC^g$ given by $\mathsection$\ref{SubsectionExampleCMAV} induces a natural split bi-$\IQbar$-structure on $\mathrm{S}^2 \Lie A_o$.

\begin{teobis}{ThmPeriodsSiegel}\label{ThmKSSplitIQbarStru}
$\mathrm{KS}_o$ induces an isomorphism of bi-$\IQbar$-structures between $T_o \mathfrak{H}_g \otimes \mathfrak{L}(1)$ and $\mathrm{S}^2 \Lie A_o$, where $\mathfrak{L}(1)$ is the Tate twist from Definition~\ref{DefnTateTwist}.
\end{teobis}

\subsection{Universal abelian variety and eigen-symplectic bases}\label{SubsectionEigensympBases}
By abuse of notation, denote by $\mathbb{A}_g$ be the fine moduli space of principally polarized abelian varieties of dimension $g$ with level-$\ell$-structure for some $\ell \ge 3$. Then there exists a universal abelian variety 
\[
f \colon \mathfrak{A}_g \to \mathbb{A}_g,
\]
which is an abelian scheme of relative dimension $g$ defined over $\IQbar$ and is endowed with a principal polarization $\lambda \colon \mathfrak{A}_g \xrightarrow{\sim} \mathfrak{A}_g^{\mathrm{t}}$; here $ \mathfrak{A}_g^{\mathrm{t}} \to \mathbb{A}_g$ is the dual abelian scheme.

For each $\tau \in \mathfrak{H}_g$, denote by $A_\tau$ the abelian variety parametrized by $[\tau]$. If $[\tau]=[\tau'] \in \mathbb{A}_g(k)$ (for an algebraically closed field $k \subseteq \CC$), then $A_{\tau}$ and $A_{\tau'}$ are $k$-isomorphic as polarized abelian varieties. 
All discussions in $\mathsection$\ref{SectionDeRhamBettiKodaira} apply to $k = \IQbar$, $(\pi \colon \cA \rightarrow S) = (f\colon \mathfrak{A}_g \to \mathbb{A}_g)$ and $(u \colon \tilde{S} \to S^{\mathrm{an}}) = (u \colon \mathfrak{H}_g \to \mathbb{A}_g^{\mathrm{an}})$. In particular, one has 
\begin{itemize}
\item a symplectic global basis $\{\omega_1,\ldots,\omega_g,\eta_1,\ldots,\eta_g\}$ of $\cH^1_{\mathrm{dR}}(\cA_{\mathfrak{H}_g} /\mathfrak{H}_g) \to \mathfrak{H}_g$ as in Construction~\ref{LemmaGlobalBasisSympDeRham}, such that when evaluated at $o$ we get precisely the $\omega_j(o)$'s as above Theorem~\ref{ThmPeriodsSiegel};
\item a symplectic global basis $\{\gamma_1,\ldots,\gamma_{2g}\}$ as in \eqref{EqLocalSystemIsomBasis} of the vector bundle $\mathbb{V}(\cA_{\mathfrak{H}_g}/\mathfrak{H}_g)$ over $\mathfrak{H}_g$ (whose fiber over each $\tau \in \mathfrak{H}_g$ equals $H_1(A_\tau,\CC)$).%\footnote{Under the notation from $\mathsection$\ref{SectionCompCoh}, $A_{\tau}$ was denoted by $\cA_{[\tau]}$.}.%, such that each $\gamma_j(o)$ is an eigenvector for the action of $E$ on $H_1(A_{o},\CC)$.
\end{itemize}

\subsection{Realization of $\mathfrak{H}_g$}\label{SubsectionRealizationHg}
Let
\[
\left(\Lambda,\Psi\right) := \left( \ZZ^{2g}, 2\pi i \begin{bmatrix} 0 & I_g \\ -I_g & 0 \end{bmatrix}\right),
\]
where $\Psi$ is viewed as a symplectic form on $\Lambda_{\CC}^\vee$. Moreover, let $E$ act on $\Lambda_{\QQ}^\vee$ via the isomorphism $\Lambda_{\QQ}^\vee \cong H^1(A_{o},\QQ)$.

Let $\mathfrak{H}_g^\vee$ be the parametrizing space of Lagrangians in $\Lambda_{\CC}^\vee$, \textit{i.e.}
\[
\mathfrak{H}_g^\vee = \left\{ [W \subseteq \Lambda_{\CC}^\vee] : \dim_{\CC} W = g, ~\Psi(W,W) = 0 \right\}.
\]
Then $\mathfrak{H}_g^\vee$ has a natural $\IQbar$-structure, for which a point $[W \subseteq \Lambda_{\CC}^\vee] \in \mathfrak{H}_g^\vee$ is a \textit{$\IQbar$-point} if and only if $W$ descends to $\IQbar$, \textit{i.e.} there exists a $\IQbar$-subspace $V \subseteq \Lambda^\vee_{\IQbar}$ such that $W = V \otimes_{\IQbar}\CC$.

\begin{contr}\label{ConstructionIotaSiegel}
The map \eqref{EqXIntoFlag} for the Siegel case (in particular $X = \mathfrak{H}_g$) factors through  an injective map $\iota \colon \mathfrak{H}_g \to \mathfrak{H}_g^\vee$ with the following property: For any $\tau \in \mathfrak{H}_g$, the entries of $\tau$ are in $\IQbar$ if and only if $\iota(\tau)$ is a $\IQbar$-point.\footnote{In fact $\mathfrak{H}_g^\vee$ is the compact dual of $\mathfrak{H}_g$, with $\iota$ the map in the Borel Embedding Theorem.}
\end{contr}
This construction of $\iota$ is important. It implies that the $\IQbar$-structure on $\mathfrak{H}_g$ constructed in $\mathsection$\ref{SubsectionGeometricIQbarShimura}, denoted by $\mathfrak{H}_{g,\IQbar}$, is equivalent to the following simple construction: 
$\mathfrak{H}_g$ is a semi-algebraic open subset of $\left\{ \tau \in \mathrm{Mat}_{g\times g}(\CC) : \tau = \tau^{\!{\intercal}} \right\} = \left\{ \tau \in \mathrm{Mat}_{g\times g}(\IQbar) : \tau = \tau^{\!{\intercal}} \right\} \otimes_{\IQbar} \CC$, and a point $\tau \in \mathfrak{H}_g$ is a $\IQbar$-point if all entries of $\tau$ are in $\IQbar$. %In $\mathsection$\ref{SubsectionRealizationHg}, we will show that this coincides with the $\IQbar$-structure on $\mathfrak{H}_g$ defined in .

\begin{proof}
The vector bundle of Betti cohomology in family $\mathbb{V}(\cA_{\mathfrak{H}_g}/\mathfrak{H}_g)^\vee$ over $\mathfrak{H}_g$ can be trivialized by the dual basis $\{\gamma_1^\vee,\ldots,\gamma_{2g}^\vee\}$ in a similar way as \eqref{EqLocalSystemIsomBasis}
\begin{equation}
\Lambda_{\CC}^\vee \times \mathfrak{H}_g \xrightarrow{\sim} \mathbb{V}(\cA_{\mathfrak{H}_g}/\mathfrak{H}_g)^\vee, \quad ((k_1,\ldots,k_{2g}), \tau) \mapsto \sum_{j=1}^{2g} k_j \gamma_j^\vee(\tau).
\end{equation}
We work under this identification. Then the comparison between de Rham and Betti cohomologies \eqref{EqComparisonDeRhamBetti} (applied to the current situation) becomes
\[
\beta \colon \cH^1_{\mathrm{dR}}(\cA_{\mathfrak{H}_g} /\mathfrak{H}_g) \xrightarrow{\sim} \Lambda_{\CC}^\vee \times \mathfrak{H}_g.
\]
For each $\tau\in\mathfrak{H}_g$, set $\beta_{\tau} \colon H^1_{\mathrm{dR}}(A_{\tau}) \xrightarrow{\sim} \Lambda_{\CC}^\vee$ to be the restriction of this comparison over $\tau$. 

We claim that  $\beta_{\tau}(\Omega_{A_\tau})$ is a Lagrangian. By \eqref{EqImageOmegaUnderComparison}, we have 
\begin{equation}\label{EqBetaOmegaTau}
\beta_{\tau}(\Omega_{A_{\tau}})  = \{(\bx, \tau  \bx) : \bx \in \CC^g\}
\end{equation}
under the basis $\{\gamma_1^\vee(\tau),\ldots,\gamma_{2g}^\vee(\tau)\}$ of $\Lambda_{\CC}^\vee$. 
 Since $\dim \beta_{\tau}(\Omega_{A_\tau}) = g$, it suffices to prove  $\Psi(\beta_{\tau}(\Omega_{A_\tau}), \beta_{\tau}(\Omega_{A_\tau})) = 0$.  This follows from the direct computation $(\bx, \tau  \bx) \begin{bmatrix} 0 & I_g \\ -I_g & 0 \end{bmatrix} (\bx, \tau  \bx)^{\!^\intercal} = 0$. Hence we are done.

Now we are ready to define
\begin{equation}\label{EqDefnSiegelInDual}
\iota \colon \mathfrak{H}_g \to \mathfrak{H}_g^\vee, \quad \tau \mapsto [\beta_{\tau}(\Omega_{A_\tau}) \subseteq \Lambda_{\CC}^\vee].
\end{equation}
%We prove that $\beta_{\tau}(\Omega_{A_\tau})$ is indeed a Lagrangian of $\Lambda_{\CC}^\vee$.

It is clear that \eqref{EqXIntoFlag} factors through $\iota$ and that $\iota$ is injective. 

By \eqref{EqBetaOmegaTau}, $\tau$ has entries in $\IQbar$ if and only if $\beta_{\tau}(\Omega_{A_{\tau}})$ is defined over $\IQbar$. Hence the entries of $\tau$ are in $\IQbar$ if and only if $\iota(\tau)$ is a $\IQbar$-point. We are done.%, \textit{i.e.} $i(\tau)$ is a $\IQbar$-point.
\end{proof}

\subsection{Proof of Theorem~\ref{ThmPeriodsSiegel} and Theorem~\ref{ThmKSSplitIQbarStru}}%{Comparison of the $\IQbar$-structures on the tangent spaces related to the moduli space}
From
\[
\xymatrix{
\mathfrak{H}_g \ar[d]_u \ar[r]^{\iota} & \mathfrak{H}_g^\vee , & o \ar@{|->}[d] \ar@{|->}[r] & \iota(o) = [\beta_o(\Omega_{A_o}) \subseteq \Lambda_{\CC}^\vee] \\
\mathbb{A}_g & & [o]
}
\]
we get
\[
\xymatrix{
T_{o} \mathfrak{H}_g \ar[d]_-{\mathrm{d}u} \ar[r]^-{\mathrm{d}\iota} & T_{\iota(o)} \mathfrak{H}_g^\vee \\
T_{[o]} \mathbb{A}_g.
}
\]
The Kodaira--Spencer map \eqref{EqKS} applied to $\mathfrak{A}_g \to \mathbb{A}_g$ at the point $[o] \in \mathbb{A}_g(\IQbar)$ gives $\mathrm{KS}_{[o]}  \colon T_{[o]} \mathbb{A}_g \to \mathrm{S}^2 \Lie(A_{o})$. In this particular case (of moduli space), this map is injective; moreover both sides have the same dimension $g(g+1)/2$. Hence $\mathrm{KS}_{[o]}$ is an isomorphism of $\IQbar$-vector spaces. To sum it up, we have 
\begin{equation}\label{EqKSSiegel}
\mathrm{KS}_{[o]}  \colon T_{[o]} \mathbb{A}_g \cong \mathrm{S}^2 \Lie(A_{o}) \subseteq \Lie(A^{\mathrm{t}}_{o}) \otimes \Lie(A_{o})
\end{equation}
and $\mathrm{KS}_{[o]}$ is defined over $\IQbar$.

By knowledge on the tangent space of Grassmannian, we have
\begin{equation}\label{EqTangentSpaceHgDual}
T_{\iota(o)} \mathfrak{H}_g^\vee \subseteq \Hom\left(\beta_{o}(\Omega_{A_{o}}), \Lambda_{\CC}^\vee/\beta_{o}(\Omega_{A_{o}}) \right) = \frac{\Lambda_{\CC}^\vee}{\beta_{o}(\Omega_{A_{o}})} \otimes  \beta_{o}(\Omega_{A_{o}})^\vee.
\end{equation}
Recall from \eqref{EqCommutativeDiagramCMAbVar} the map $\beta_{o} \colon H^1_{\mathrm{dR}}(\cA_{o}) \to H^1(\cA_{o},\CC) = \Lambda_{\CC}^\vee$ and the induced maps
\begin{equation}\label{Eqlielielielielie}
\overline{\beta}_{o} \colon \Lie(A^{\mathrm{t}}_{o}) \to \Lambda_{\CC}^\vee/\beta_{o}(\Omega_{A_{o}}), \quad \beta_{o}^\vee \colon \beta_{o}(\Omega_{A_{o}})^\vee \to \Lie A_{o} = \Omega_{A_{o}}^\vee.
\end{equation}

Apply \eqref{EqKSTwoComparison} to $(\tilde S, \tilde s) = (\mathfrak{H}_g, o)$, and notice that the bottom arrow in this case is precisely $\mathrm{d}\iota \colon T_o\mathfrak{H}_g \to T_{\iota(o)}\mathfrak{H}_g^\vee$ composed with \eqref{EqTangentSpaceHgDual}. Thus 
$\mathrm{d}u \circ (\mathrm{d}\iota)^{-1}$ is the restriction of $\overline{\beta}_{o}^{-1} \otimes \beta_{o}^\vee$.

By Lemma~\ref{LemmaCMAbVarDiagonal}, $\overline{\beta}_{o}^{-1}$ is diagonalizable under $\IQbar$-bases to be $\mathrm{diag}(\theta_1/\pi,\ldots,\theta_g/\pi)$. By Remark~\ref{RemarkCMAbVarDiagonal}, $\beta_{o}^\vee$ is diagonalizable under $\IQbar$-bases to be $\mathrm{diag}(\theta_1,\ldots,\theta_g)$. Hence under suitable $\IQbar$-bases of $T_{\iota(o)} \mathfrak{H}_g^\vee$ and $T_{[o]} \mathbb{A}_g$, the matrix of the linear map $\mathrm{d}u \circ (\mathrm{d}\iota)^{-1} $ is $\mathrm{diag}(\theta_j \theta_{j'}/\pi)_{1\le j \le j' \le g}$. 

This yields Theorem~\ref{ThmPeriodsSiegel}, because the geometric $\IQbar$-structure on $T_o\mathfrak{H}_g$ is by definition given by the $\IQbar$-structure on $T_{\iota(o)}\mathfrak{H}_g^\vee$.

We can do better. The computation above and \eqref{EqKSTwoComparison} applied to $(\tilde S, \tilde s) = (\mathfrak{H}_g, o)$ yield the following assertion: $\mathrm{KS}_{[o]}$ maps a $1$-dimensional bi-$\IQbar$-subspace of $T_{[o]}\mathbb{A}_g$ with period  $\theta_j\theta_{j'}/\pi$ to a $1$-dimensional bi-$\IQbar$-subspace of $\mathrm{S}^2\Lie(A_o)$ with period $\theta_j\theta_{j'}$. This proves Theorem~\ref{ThmKSSplitIQbarStru}. We are done.
\qed

\begin{rem}\label{RemarkSymplecticSpaceDecompositionLagrangian}
By general knowledge on symplectic vector spaces, for the Lagrangian $\beta_{o}(\Omega_{A_{o}})$ of $\Lambda_{\CC}^\vee$ we have a decomposition  $\Lambda_{\CC}^\vee = \beta_{o}(\Omega_{A_{o}}) \bigoplus \beta_{o}(\Omega_{A_{o}})^\vee$ given as follows: take a $\IQbar$-basis $\{x_1,\ldots,x_g\}$ of $\beta_{o}(\Omega_{A_{o}})$, and then one has a dual $\IQbar$-basis for a complement defined by $\Psi(x_j,y_l)= 2\pi i \delta_{jl}$, and this complement is canonically isomorphic over $\CC$ to $\beta_{o}(\Omega_{A_{o}})^\vee$. Over $\IQbar$, this decomposition then becomes $\Lambda_{\IQbar}^\vee = \beta_o(\Omega_{A_o}) \bigoplus 2\pi i \beta_o(\Omega_{A_o})^\vee$. 
Hence we have a canonical isomorphism 
\begin{equation}\label{EqQbarDualQuotient2pii}
2\pi i \beta_{o}(\Omega_{A_{o}})^\vee = \Lambda_{\CC}^\vee/\beta_{o}(\Omega_{A_{o}})
\end{equation}
 defined over $\IQbar$.
\end{rem}

\vskip 0.5em
We close this subsection with the following (slightly more precise) formulation of Theorem~\ref{ThmPeriodsSiegel}, which follows immediately from the computation above. Denote by $\mathfrak{H}_{g,\IQbar}$ the $\IQbar$-structure on $\mathfrak{H}_g$ explained below Construction~\ref{ConstructionIotaSiegel}. Consider the tangent of the uniformization $u$ at the CM point $o$
\[
\psi := \mathrm{d}u \colon T_{o} \mathfrak{H}_g \to T_{[o]} \mathbb{A}_g.
\]

\begin{teobisbis}{ThmPeriodsSiegel}\label{ThmPeriodsSiegelBis}
Under a suitable $\IQbar$-basis of $T_{o} \mathfrak{H}_g = T_{o} \mathfrak{H}_{g,\IQbar} \otimes_{\IQbar}\CC$ and a suitable $\IQbar$-basis of $T_{[o]} \mathbb{A}_g = T_{[o]} \mathbb{A}_{g,\IQbar}\otimes_{\IQbar}\CC$, the matrix of $\psi$ is $\mathrm{diag}(\theta_j \theta_{j'}/\pi)_{1\le j \le j' \le g}$.
\end{teobisbis}

\subsection{In relation to the root space decomposition}\label{SubsectionTwoBiIQbarDecomSame}
Let $\bT$ be a maximal torus which contains $\mathrm{MT}(o)$. 
The last sentence of Remark~\ref{RemarkCMAbVarDiagonal} yields 
 a $\ZZ$-basis $\{\epsilon_0^*,\epsilon_1^*,\ldots,\epsilon_g^*\}$ of $X^*(\bT)$ such that $\epsilon_j^*({\beta_o}) = \theta_j$ for each $j \in \{1,\ldots,g\}$ and $\epsilon_0^*$ is a character of $\Gm$. 
The root system $\Phi(\bT, \GSp_{2g}) \subseteq X^*(\bT)$ defined in \cite[13.18]{BorelLinear-Algebrai} is $\{ \epsilon_j^*+\epsilon_{j'}^* : 1 \le j \le j' \le g \} \bigcup \{ \epsilon_j^*-\epsilon_{j'}^* : 1 \le j < j'\le g\}$. The set $\Phi_M^+$ defined above Proposition~\ref{PropDecompositionGeomIQbar} is $\{ \epsilon_j^*+\epsilon_{j'}^* : 1 \le j \le j' \le g \}$.

\begin{teo}\label{ThmDecompositionIntoRootSpacesSiegel}
The $1$-dimensional bi-$\IQbar$ subspace of $T_o \mathfrak{H}_g$ associated with $\epsilon_j^*+\epsilon_{j'}^*$ has period $\theta_j\theta_{j'}/\pi$.
\end{teo}
\begin{proof} All vector spaces in the proof are $\CC$-spaces unless otherwise stated. 
We start by defining an action of $\bT(\QQ)$ on $\Lie A_o = \Omega_{A_o}^\vee$. By \eqref{EqBetaOmegaTau}, we have $\beta_o(\Omega_{A_o}) = \{(o \bx,  \bx): \bx \in \CC^g\}$ under the basis $\{\gamma_{g+1}^\vee(o),\ldots,\gamma_{2g}^\vee(o),\gamma_1^\vee(o),\ldots,\gamma_g^\vee(o)\}$ of $H^1(A_o,\CC)$. % Thus $\beta_o^\vee(\Omega_{A_o}^\vee) = \beta_o(\Omega_{A_o})^\vee  = \CC^{2g} / \{(-o \bx', \bx'):\bx' \in \CC^g\}$.
%\{ (o \cdot \bx' , \bx') : \bx' \in \CC^g \}$.
%Thus under the basis $\{\gamma_{g+1}^\vee(o),\ldots,\gamma_{2g}^\vee(o),\gamma_1^\vee(o),\ldots,\gamma_g^\vee(o)\}$, 
The action of $\GSp_{2g}(\RR)^+$ on $\mathfrak{H}_g$ is defined by $\begin{bmatrix} A & B \\ C & D \end{bmatrix}\!\tau  = (A\tau+B)(C\tau+D)^{-1}$. By the discussion at the end of $\mathsection$\ref{SubsectionNotationSplitBiIQbarShimura}, $\bT(\RR)o = o$. So for each  $\begin{bmatrix} A & B \\ C & D \end{bmatrix} \in \bT(\RR)$, we have $(Ao+B) (C o+D)^{-1}= o$, and hence  
%On the other hand, any element $t = \begin{bmatrix} A & B \\ C & D \end{bmatrix} \in \GSp_{2g}(\RR)$ acts on $\mathfrak{H}_g$ by $t \cdot \tau = (A\tau+B)(C\tau+D)^{-1}$. So $(Ao+B) (C o+D)^{-1}= o$ for all  $\begin{bmatrix} A & B \\ C & D \end{bmatrix} \in \bT(\RR)$. Hence
\[
\begin{bmatrix} A & B \\ C & D \end{bmatrix} \begin{bmatrix} o\bx \\ \bx \end{bmatrix} = \begin{bmatrix} (Ao+B) \bx \\ (Co+D) \bx \end{bmatrix} = \begin{bmatrix} o\mathbf{y} \\ \mathbf{y} \end{bmatrix} \in \beta_o(\Omega_{A_o})
\]
with $\mathbf{y} =  (Co+D) \bx$. This defines an action of $\bT(\QQ)$ on $\Lie A_o$, and hence an action of $\bT(\QQ)$ on $\mathrm{S}^2 \Lie A_o$. Notice that the Kodaira--Spencer map $\mathrm{KS}_o$ is $\bT(\QQ)$-equivariant because every morphism on the bottom line of \eqref{EqKSTwoComparison1}, when restricted to $o$, is $\bT(\QQ)$-equivariant. % Hence we are done. 
Thus the $1$-dimensional bi-$\IQbar$ subspace of $T_o\mathfrak{H}_g$ associated with $\epsilon_j^*+\epsilon_{j'}^*$ is the eigenspace of $\epsilon_j^*+\epsilon_{j'}^*$ for the action of $\bT(\QQ)$ on $\mathrm{S}^2 \Lie A_o$.

Recall the bi-$\IQbar$ structure on $\Lie A_o$ defined in $\mathsection$\ref{SubsectionExampleCMAV}. 
By the first part of Remark~\ref{RemarkCMAbVarDiagonal} and the choice of $\epsilon_j^*$, the eigenspace of $\epsilon_j^*$ for the action of $\bT(\QQ)$ on $\Lie A_o$ is the $1$-dimensional bi-$\IQbar$ subspace of $\Lie A_o$ associated with $\theta_j$. Hence we are done by the conclusion of the previous paragraph.
\end{proof}

Theorem~\ref{ThmDecompositionIntoRootSpacesSiegel} yields the following immediate corollary. By a \textit{root space}, we mean a $1$-dimensional bi-$\IQbar$-subspace of $T_o \mathfrak{H}_g$ associated with some $\epsilon_j^* + \epsilon_{j'}^*$.
\begin{cor}\label{CorDecompositionIntoRootSpacesSiegel}
The following statements are equivalent:
\begin{enumerate}
\item[(i)] $\frac{\theta_j\theta_{j'}}{\theta_k\theta_{k'}} \not\in\IQbar^*$ for all distinct pairs $\{j,j'\}$ and $\{k,k'\}$;
\item[(ii)] each $1$-dimensional  bi-$\IQbar$-subspace of $T_o \mathfrak{H}_g$ is a root space;
\item[(iii)] each bi-$\IQbar$-subspace of $T_o \mathfrak{H}_g$ is the direct sum of root spaces.
\end{enumerate}
\end{cor}
In particular, if there exists a special subvariety (\textit{i.e.} a connected sub-Shimura variety) $S$ of $\mathbb{A}_g$ passing through $[o]$ such that $T_{[o]}S$ is not the direct sum of root spaces of $T_{[o]} \mathbb{A}_g$, then the holomorphic periods of $A_o$ satisfy some elementary non-trivial quadratic relation as defined in Definition~\ref{DefnElementaryNonTrivial}. 
We shall elaborate this phenomenon in a subsequent work \cite[$\mathsection$9]{GUHodgeCycle} by the first- and second-named authors.

\section{Shimura curves and Hilbert modular varieties}\label{SectionShimuraCurveHilbert}
Let $S$ be a connected Shimura variety associated with a connected Shimura datum $(\bG, X)$. Let $S_{\bH}$ be a connected Shimura subvariety associated with $(\bH,X_{\bH})$. We have seen that $X$ can be endowed with a split bi-$\IQbar$-structure in $\mathsection$\ref{SectionBiIQbarShimura}, and the restriction of this split bi-$\IQbar$-structure on $X_{\bH}$ gives precisely the bi-$\IQbar$-structure defined for $X_{\bH}$ in $\mathsection$\ref{SectionBiIQbarShimura}. In particular, each period of $X_{\bH}$ is a period of $X$ by Proposition~\ref{PropBiQbarEasyProp}.(iii).

In this section, we see two aspects of this discussion. The first is about Shimura curves, for which we use Corollary~\ref{PropProductPeriodTranscendental} (which is a consequence of WAST) and our Theorem~\ref{ThmPeriodsSiegel} to give a new proof of a conjecture of Lang. The second is to compute the periods of Hilbert modular varieties.

\subsection{Shimura curves}
Lang \cite{Lang} raised the following transcendence question in uniformization theory. Let $C$ be a smooth projective algebraic curve of genus $> 1$ defined over $\IQbar$. Suppose that the universal holomorphic covering map
\begin{equation}\label{EqQuestionLang}
\varphi \colon E_{\rho} := \{z \in \CC : |z| < \rho\} \rightarrow C^{\mathrm{an}}
\end{equation}
is normalized in such a way that $\varphi(0) \in C(\IQbar)$ and that $\varphi'(0) \in \IQbar$. Is the covering radius $\rho$ then a transcendental number?

This question was answered affirmatively by Cohen and Wolfart \cite{CW} when $C$ is a Shimura variety of dimension $1$ and $\varphi(0)$ is a CM point. Apart from W\"{u}stholz's result on transcendental numbers, their proof relies on some hard computation of Shimura \cite[Thm.1.2, Thm.7.1, Thm.7.6]{ShimuraRelation}.

We hereby give a new and easier proof using the framework of this paper which does not use Shimura's computation. More precisely we use the bi-$\IQbar$-structure on the Hermitian symmetric space and our Theorem~\ref{ThmPeriodsSiegel}, in addition to the consequence of W\"{u}stholz prsented in our paper (Corollary~\ref{PropProductPeriodTranscendental}).

\begin{prop}\label{PropLangQuestion}
The question of Lang mentioned in the paragraph of \eqref{EqQuestionLang} has an affirmative answer if $C$ is a Shimura variety of dimension $1$ and $\varphi(0)$ is a CM point.
\end{prop}

\begin{proof}
By general theory of Shimura varieties, $C$ is of abelian type, \textit{i.e.} there exists a connected Shimura subvariety $C'$ of $\mathbb{A}_g$ (for some $g\ge 1$) of dimension $1$ together with a finite map $C' \rightarrow C$. Moreover, it is known that this finite map is defined over $\IQbar$, and hence it suffices to prove the result with $C$ replaced by $C'$.

Now let us explain how the uniformization $\varphi$ is obtained from the Shimura setting. Take $o \in \mathfrak{H}_g$ which maps to $\varphi(0)$ under the uniformization $\mathfrak{H}_g \rightarrow \mathbb{A}_g$. Let $A_o$ be the CM abelian variety parametrized by $\varphi(0)$, and let $\theta_1,\ldots,\theta_g$ be its holomorphic periods. 

Write $\cD_g$ for the bounded realization of $\mathfrak{H}_g$ based on $o$.

As a sub-Shimura variety of $\mathbb{A}_g$ of dimension $1$, $C = \Gamma\backslash \cD$ where $\cD = \{z \in \CC : |z|<1\}$ is the Poincar\'{e} unit disk and $\Gamma$ is a subgroup of $\Sp_{2g}(\ZZ)$. The Shimura sub-datum with which $C$ is associated can be recovered as follows. Let $\bH'$ be the neutral component of the Zariski closure of $\Gamma$ in $\GSp_{2g}$, then it is known that $\cD = \bH'(\RR)^+o$. Now let  $\bH = \mathbb{G}_m \cdot \bH' < \GSp_{2g}$. Then $\cD = \bH(\RR)^+o \subseteq \GSp_{2g}(\RR)^+o = \cD_g$, and thus we have the sub-Shimura datum $(\bH,\cD)$ of $(\GSp_{2g},\cD_g)$. Now $C$ is associated with $(\bH,\cD)$, and we have a uniformization $u_C \colon \cD \rightarrow C = \Gamma\backslash \cD$.

Now $\varphi$ from \eqref{EqQuestionLang} is the composite of $\rho^{-1} \cdot \colon E_{\rho} \to \cD$ and $u_C$. Thus $\varphi'(0) = \rho^{-1} u_C'(0)$. Since $\varphi'(0) \in \IQbar$ by assumption, we have that $\rho = u_C'(0)$ up to $\IQbar^*$. Hence it suffices to prove that $u_C'(0)$ is a transcendental number.

 By Theorem~\ref{ThmPeriodsSiegel} (and the formulation Theorem~\ref{ThmPeriodsSiegelBis}) and Proposition~\ref{PropBiQbarEasyProp}.(iii), we have $u_C'(0) = \theta_j\theta_{j'}/\pi$ for some $j$ and $j'$. Thus $u_C'(0)$ is a transcendental number by Corollary~\ref{PropProductPeriodTranscendental}. We are done.
\end{proof}

\subsection{Hilbert modular varieties}\label{SubsectionHilbertModularVariety}

Let $F$ be a totally real field with $[F:\QQ] = g$. Let $\mathbf{G}:= \Gm \cdot \mathrm{Res}_{F/\QQ}\mathrm{SL}_2$. Then $\mathbf{G}^{\mathrm{der}}(\QQ) = \mathrm{SL}_2(F)$, and $\mathbf{G}^{\mathrm{der}}(\RR) = \prod_{\sigma \colon F \to \RR} \mathrm{SL}_2(\RR)$. For the Siegel upper half plane $\mathfrak{H} = \{\tau \in \CC: \mathrm{Im}\tau > 0\}$, $G:=\mathbf{G}(\RR)^+$ acts on $\mathfrak{H}^g$ componentwise. The pair $(\mathbf{G},\mathfrak{H}^g)$ is a connected Shimura datum. For each congruence subgroup $\Gamma$ of $\mathbf{G}(\QQ)$, the connected Shimura variety $\Gamma \backslash \mathfrak{H}^g$ is, by Baily--Borel, the analytification of an algebraic variety, which furthermore is defined over $\IQbar$. This Shimura variety, which we denote by $\mathcal{H}_F$, is called the \textit{Hilbert modular variety} defined by $\mathbf{G}$ (or, more simply, associated with $F$). Then $\mathcal{H}_F$ is a Shimura subvariety of $\mathbb{A}_g$.

Let $\cO_F$ be the ring of integers of $F$.  
An abelian scheme $\cA \to S$ of relative dimension $g$ is said to \textit{have RM by $F$} if there exists an injective ring homomorphism $\iota \colon R_F \hookrightarrow \mathrm{End}_S(\cA)$ (for some order $R_F$ in $\cO_F$) such that $\Lie(\cA/S)$ is a locally free $R_F\otimes \cO_S$-module. It is known that the Hilbert modular variety $\mathcal{H}_F$  parametrizes abelian varieties with RM by $F$ and some extra structures.

The goal of this section is to prove the following theorem. Let $[ o ] \in \mathcal{H}_F(\IQbar)$ be a CM point. Let $A_o$ be the CM variety parametrized by $[o]$, and let $\theta_1,\ldots,\theta_g$ be the holomorphic periods of $A_o$ as defined in the introduction (Convention 1.1).

Let $o \in \mathfrak{H}^g$ be a point such that $o \mapsto [o]$. The natural inclusion $\mathfrak{H}^g \subseteq \mathfrak{H}_g$ induces $T_o \mathfrak{H}^g \subseteq T_o \mathfrak{H}_g$. Thus the split bi-$\IQbar$-structure on $T_o \mathfrak{H}_g$ defined in $\mathsection$\ref{SectionBiIQbarShimura} gives a split bi-$\IQbar$-structure on $T_o \mathfrak{H}^g$. 

\begin{teo}
Under suitable $\IQbar$-bases of $T_{ o } \mathfrak{H}^g = T_o \mathfrak{H}_{\IQbar}^g \otimes_{\IQbar} \CC$ and $T_{[ o ]} \mathcal{H}_F = T_{[o]} \mathcal{H}_{F,\IQbar}\otimes_{\IQbar} \CC$, the matrix of the linear map $\mathrm{d}u$ is $\frac{1}{\pi}\mathrm{diag}(\theta_j^2)_{j\in \{1,\ldots,g\}}$. In particular, the periods of the split bi-$\IQbar$-structure on $T_{ o } \mathfrak{H}^g$ are $\theta_j^2 / \pi$ for $j \in \{1,\ldots,g\}$.
\end{teo}

\begin{proof} We write for simplicity $\mathcal{H}$ for $\mathcal{H}_F$. 

The following holds true by the modular interpretation of $\mathcal{H}$ above. For a suitable $\Gamma$, the Hilbert modular variety $\mathcal{H}$ is a fine moduli space with the following property: the universal family $\cA \to \mathcal{H}$ satisfies that $\Lie(\cA/\mathcal{H})$ is a locally free $R_F\otimes \cO_S$-module. Hence $\Lie(A_{ o })$ is a free $R_F$-module. But $\Lie (A_o)$ has dimension $g$ and $R_F$ is a $\ZZ$-module of rank $g$. so $\Lie (A_o)$ has rank $1$ as an $R_F$-module. We write $\Lie(A_{ o }) = R_F  \mathbf{f}$ with $\mathbf{f}$ a $\IQbar$-element of $\Lie (A_o)$.

As $\mathcal{H} \subseteq \mathbb{A}_g$, the Kodaira--Spencer map \eqref{EqKSSiegel} gives, in combination with $A_o^{\mathrm{t}} \cong A_o$, 
\[
T_{[ o ]}\mathcal{H} \subseteq T_{[ o ]}\mathbb{A}_g \xrightarrow{\mathrm{KS}_{[ o ]}} \Lie(A_{ o }) \otimes \Lie(A_{ o }).
\]
The $R_F$-action on each fiber of $\Lie(\cA/\mathcal{H})$ (and of $\Lie (\cA^{\mathrm{t}}/\mathcal{H})$) then implies $\mathrm{KS}_{[ o ]}(T_{[ o ]}\mathcal{H}) \subseteq \Lie(A_{ o }) \otimes_{R_F} \Lie(A_{ o })$. By comparing dimensions of both sides, we have equality. Hence we have the isomorphism defined over $\IQbar$
\[
\mathrm{KS}_{[ o ]} \colon T_{[ o ]}\mathcal{H} \xrightarrow{\sim} \Lie(A_{ o }) \otimes_{R_F} \Lie(A_{ o }).
\]
Similarly, by using \eqref{EqTangentSpaceHgDual} and taking into consideration of \eqref{EqQbarDualQuotient2pii} we get
\[
T_{ o }\mathfrak{H}^g = 2\pi i  \beta_{o}(\Omega_{A_{o}})^\vee  \otimes_{R_F} \beta_{o}(\Omega_{A_{o}})^\vee.
\]
Similarly to $\Lie (A_o)$, $\beta_o(\Omega_{A_o})^\vee$ is a free $R_F$-module of rank $1$. 
Write $\beta_o(\Omega_{A_o})^\vee = R_F \mathbf{e}$ with $\mathbf{e}$ a $\IQbar$-element of 
$\beta_o(\Omega_{A_o})^\vee$. The following diagram commutes, with all horizontal maps defined over $\IQbar$:
\[
\xymatrix@R+1pc@C+5.5pc{
T_{ o }\mathfrak{H}^g \ar[r]^-{=} \ar[d]_-{\mathrm{d}u} & 2\pi i  \beta_{o}(\Omega_{A_{o}})^\vee \otimes_{R_F} \beta_{o}(\Omega_{A_{o}})^\vee \ar[r]^-{2\pi i f_1 \mathbf{e} \otimes f_2\mathbf{e} \mapsto 2\pi i f_1f_2\mathbf{e}}_-{\sim}  \ar[d]^-{\frac{1}{2\pi i}\beta_o^\vee \otimes \beta_o^\vee}  & 2\pi i \beta_{o}(\Omega_{A_{o}})^\vee \ar[d]^-{\frac{1}{2\pi i}(\beta_o^\vee)^2} \\
T_{[ o ]}\mathcal{H} \ar[r]^-{\mathrm{KS}_{[ o ]}}_-{\sim} & \Lie(A_{ o }) \otimes_{R_F} \Lie(A_{ o })  \ar[r]^-{f_1 \mathbf{f} \otimes f_2\mathbf{f} \mapsto f_1f_2\mathbf{f}}_-{\sim} & \Lie(A_{ o }).
}
\]
Thus the conclusion follows from the first part of Remark~\ref{RemarkCMAbVarDiagonal}.
%Hence we are done.
\end{proof}

\section{Some general discussion on CM points}\label{SectionDiscussionCM}
In this section, we gather some preliminary results regarding the Mumford--Tate group of CM abelian varieties and on CM points in $\mathbb{A}_g$. They will be used in later discussions.

\subsection{Mumford--Tate group of CM abelian varieties}

Let $A$ be a CM abelian variety, and let $\bT_o := \mathrm{MT}(A)$. Let $E:= \mathrm{End}(A)\otimes_{\ZZ} \QQ$.

If $A$ has no square factors, then $E =  E_1\times \cdots \times E_k$  is a product of CM fields. In this case, there exists an element $\iota \in E$ such that $\bar{\iota} = - \iota$. Then $E$ can be endowed with the $\QQ$-symplectic form
\[
\langle x, y \rangle := \mathrm{Tr}_{E/\QQ}(\bar{x}\iota y).
\]
This makes $(E, \langle, \rangle) \cong ( \QQ^{2g}, \begin{bmatrix} 0 & I_g \\ -I_g & 0 \end{bmatrix})$ into a symplectic space. 
Set $\mathrm{GU}_E$ to be the subgroup of $\GSp_{2g}$ generated by $\Gm = Z(\GSp_{2g})$ and
\[
\mathrm{U}_E := \{ x \in \mathrm{Res}_{E/\QQ}\GmF{E} : x \bar{x} = 1\}.
\]

\begin{lem}\label{LemmaNoSquareFactorRegularMT}
Assume $A$ has no square factors. 
Then $\mathrm{GU}_E$ is a maximal torus of $\GSp_{2g}$ and equals $Z_{\GSp_{2g}}(\bT_o)$.
\end{lem}
\begin{proof} For each $j \in \{1,\ldots, k\}$, let $F_j$ be the largest totally real subfield of $E_j$. Set $F:=F_1 \times \cdots \times F_k$. Then
\[
\mathrm{U}_E = \ker(\mathrm{Nm} \colon \mathrm{Res}_{E/\QQ}\GmF{E} \rightarrow \mathrm{Res}_{E/\QQ}\GmF{F}).
\]
Hence $\mathrm{U}_E$ is a torus. Moreover, $\sum_{j=1}^k [F_j :\QQ] = g$. So $\dim \mathrm{U}_E = g$. 

Therefore  $\mathrm{GU}_E$ is a torus of dimension $g+1$. But $\mathrm{rk}\GSp_{2g} = g+1$. So $\mathrm{GU}_E$ is a maximal torus of $\GSp_{2g}$.

Now let us prove $\mathrm{GU}_E = Z_{\GSp_{2g}}(\bT_o)$. Since $\bT_o < \mathrm{GU}_E$ and $\mathrm{GU}_E$ is a torus, we have  $\mathrm{GU}_E < Z_{\GSp_{2g}}(\bT_o)$. Hence it suffices to prove that $Z_{\GSp_{2g}}(\bT_o)$ is a torus.%Assume $[o] \in \mathbb{A}_g(\IQbar)$ parametrizes $A$, and let $o \in \mathfrak{H}_g$ be such that $o \mapsto [o]$. %Then $Z_{\GSp_{2g}}(\bT_o)$ is a reductive group because its extension to $\RR$ contains $o(\SSS)$. By Lemma~\ref{LemmaMTAndCompactR} and since $\QQ$ is dense in $\RR$, we have $Z_{\GSp_{2g}}(\bT_o)(\RR) < \mathrm{Stab}_{\GSp_{2g}(\RR)}(o)$. 

Write $V= H_1(A,\QQ)$, which is a $\QQ$-Hodge structure of type $(-1,0)+(0,-1)$ and whose Mumford--Tate group is $\bT_o$. Then for any $\alpha \in \mathrm{End}(V_{\QQ})$, we have $\alpha \in \mathrm{End}_{\QQ\text{-HS}}(V)$ if and only if $\alpha$ commutes with all elements in $\bT_o(\QQ)$.
Thus  $\mathrm{End}_{\QQ\text{-HS}}(V) = \Lie Z_{\GSp_{2g}}(\bT_o)$. Observe that $\mathrm{End}_{\QQ\text{-HS}}(V)$ is abelian because $A$ has no square factors. So $Z_{\GSp_{2g}}(\bT_o)$ is a torus. We are done.
\end{proof}

\subsection{Weyl points}

Let $[o] \in \mathbb{A}_g(\IQbar)$ be a CM point, and let $A_o$ be the associated CM abelian variety. Assume that $A_o$ is simple. Then $E:=\mathrm{End}^0(A_o)$ is a CM field. Let $F$ be the maximal totally real subfield of $E$, then $[F:\QQ] =g$ and $[E:F]=2$.

Write $E^c$ and $F^c$ for the Galois closures of $E$ and $F$ in $\IQbar$. Then $\Gal(F^c/\QQ) < \mathfrak{S}_g$ and $\Gal(E^c/\QQ) < (\ZZ/2\ZZ)^g \rtimes \mathfrak{S}_g$. 

The point $[o]$ is called a \textit{Weyl point} (or \textit{Galois generic}) if $\Gal(E^c/\QQ) = (\ZZ/2\ZZ)^g \rtimes \mathfrak{S}_g$.

\begin{prop}\label{PropWeylPoint}
For each Weyl point $[o] \in \mathbb{A}_g(\IQbar)$, there are precisely $3$ special subvarieties of $\mathbb{A}_g$ which pass through $[o]$: $\{[o]\}$, $\mathbb{A}_g$, and the Hilbert modular variety defined by $\Gm \cdot \mathrm{Res}_{F/\QQ}\mathrm{SL}_2$.\footnote{For readers who are not familiar with this terminology, we refer to $\mathsection$\ref{SubsectionHilbertModularVariety} for the definition.}
\end{prop}
\begin{proof}
The result is clearly true if $g = 1$. From now on, assume $g\ge 2$.

Let $S$ be a special subvariety of $\mathbb{A}_g$, which passes through $[o]$, associated with the connected Shimura datum $(\bH,X_{\bH})$. Assume $S \not= \{[o]\}$.% and $S \not= \mathbb{A}_g$.

Let $\bT_o := \mathrm{MT}([o])$. Then $\bT_o<\bH$, $\bT_o$ is a maximal torus of $\GSp_{2g}$, and the Galois closure of the splitting field of $\bT_o$ is $E^c$. There exists a maximal torus $\bT_o'$ of $\Sp_{2g}$ such that $\bT_o = \Gm \cdot \bT_o'$, where $\Gm = Z(\GSp_{2g})$. In particular $\dim \bT_o' = g$.%, because $G=Z(G)G^{\mathrm{der}}$ is an almost direct product. 

%The splitting field of $\bT_o'$ is a subfield of $E^c$, \textit{i.e.} $\bT_o'_{E^c} \cong \GmF{E^c}^g$.

%Since $o \in X_H$, we have $\bT_o < \bH$.

We claim that $\bT_o' \subseteq \bH^{\mathrm{der}}$. Indeed, $\bT_o' \cap \bH^{\der}$ is a non-trivial subtorus of $\bT_o'$, and hence $\bT_o' = (\bT_o' \cap \bH^{\der})\bT_o''$ for some subtorus $\bT_o''$ of $\bT_o'$. Let $r_1 := \dim (\bT_o' \cap \bH^{\der})$ and $r_2 := \dim \bT_o''$. Let $L_1^c$ (resp. $L_2^c$) be the Galois closure of the splitting field of $\bT_o' \cap \bH^{\der}$ (resp. of $\bT_o''$). Then $\Gal(L_i^c/\QQ) < (\ZZ/2\ZZ)^{r_i}\rtimes \mathfrak{S}_{r_i}$ for $i \in \{1,2\}$. Since $E^c$ is the Galois closure of the splitting field of $\bT_o'$, we then have that
\[
\Gal(E^c/\QQ) < \prod_{i=1}^2 (\ZZ/2\ZZ)^{r_i}\rtimes \mathfrak{S}_{r_i}.
\]
But $\Gal(E^c/\QQ) = (\ZZ/2\ZZ)^g \rtimes \mathfrak{S}_g$ since $[o]$ is Weyl. Notice that $r_1+r_2 =g$ and $r_1 > 0$. So $r_1 = g$ and $r_2 = 0$. Therefore $\bT_o' = \bT_o' \cap \bH^{\mathrm{der}}$, and hence $\bT_o' < \bH^{\mathrm{der}}$.

The upshot is that $\bT_o'$ is a maximal torus of $\bH^{\mathrm{der}}$. Hence $\mathrm{rk}\bH^{\mathrm{der}} = g$. For the root system of $(\bT_o',\bH^{\mathrm{der}})$, the Weyl group of $\bH^{\mathrm{der}}$ is then $(\ZZ/2\ZZ)^g \rtimes \mathfrak{S}_g$.

We claim that $\bH^{\mathrm{der}}$ is a $\QQ$-simple algebraic group. Assume not, then we can get a contradiction to $\Gal(E^c/\QQ) = (\ZZ/2\ZZ)^g \rtimes \mathfrak{S}_g$ with a similar argument as for $\bT_o' < \bH^{\mathrm{der}}$.

The upshot is that $\bH^{\mathrm{der}} = \mathrm{Res}_{F'/\QQ}\tilde{\bH}$ for some $F'$ totally real and $\tilde{\bH}$ simply-connected. Thus
\[
\bH_{\CC} = \oplus_{\sigma \colon F' \to \RR} \tilde{\bH}_{\sigma,\CC}.
\]
For $r := [F':\QQ]$, we then have $r \cdot \mathrm{rk}(\bH_{\sigma,\CC}) = g$.%\footnote{$\mathrm{rk}(\bH_{\sigma,\CC}) $ also equals $\mathrm{rk} \Phi_{\sigma}$ and $\mathrm{rk} \Delta_{\sigma}$.}

Let $\Phi$ be the root system of $\bH_{\CC}$, and let $\Phi_{\sigma}$ be the root system of $\bH_{\sigma,\CC}$. Then $\Phi = \coprod_{\sigma \colon F' \to \RR} \Phi_{\sigma}$. Then
\[
\mathrm{Aut}(\Phi) = \left( \prod_{\sigma} \mathrm{Aut}(\Phi_{\sigma}) \right) \rtimes \mathfrak{S}_r.
\]
By the theory of root systems, $\mathrm{Aut}(\Phi_{\sigma}) = W_{\sigma}\rtimes \mathrm{Aut}(\Delta_{\sigma})$, where $W_{\sigma}$ is the Weyl group of $H_{\sigma,\CC}$, and $\mathrm{Aut}(\Delta_{\sigma})$ is the transformation group of the Dynkin diagram $\Delta_{\sigma}$ associated with $\bH_{\sigma,\CC}$. Then $W_{\sigma}$ is a Weyl group of type $A_r$, $B_r$, $C_r$, or $D_r$ with $r$ dividing $g$. We may therefore have $W_{\sigma}=\mathfrak{S}_{r+1}$, $W_{\sigma}=(\ZZ/2\ZZ)^r \rtimes \mathfrak{S}_r$ or $W_{\sigma}=(\ZZ/2\ZZ)^{r-1} \rtimes \mathfrak{S}_r$. Moreover, 
 $\mathrm{Aut}(\Delta_{\sigma})$ equals $1$, $\ZZ/2\ZZ$, or $\mathfrak{S}_3$ (in which case $\bH_{\sigma,\CC}$ is of type $D_3$).

Assume $g \ge 4$. 
Notice that $\mathrm{Aut}(\Phi)$ contains the Weyl group of $\bH^{\mathrm{der}}$, which is $(\ZZ/2\ZZ)^g \rtimes \mathfrak{S}_g$. So we must have $r = 1$ or $r = g$. The same holds true for $g = 2$ and $g = 3$ because $r|g$.

If $r = 1$, then $F' = \QQ$ and $\bH^{\mathrm{der}}$ is a simple-connected simple group of rank $g$. Then it is of type $A_g$, $B_g$, $C_g$, or $D_g$. Since the Weyl group of $\bH^{\mathrm{der}}$ is $(\ZZ/2\ZZ)^g \rtimes \mathfrak{S}_g$ and $\bH^{\mathrm{der}}$ is a subgroup of $\Sp_{2g}$, we conclude that $\bH^{\mathrm{der}} = \Sp_{2g}$. And hence $S = \mathbb{A}_g$.

If $r = g$, then $[F':\QQ] = g \ge 2$ and $\mathrm{rk} \tilde{\bH} = 1$. So $\tilde{\bH}$ is a simple simply-connected group of type $A_1$. Therefore $\tilde{\bH}_{\IQbar}$ is either $\mathrm{SL}_2$ or $\mathrm{PSL}_2$. As $\bH^{\mathrm{der}} = \mathrm{Res}_{F'/\QQ}\tilde{\bH}$ is a subgroup of $\Sp_{2g}$, we have  $\tilde{\bH}_{\IQbar} = \mathrm{SL}_2$. Hence $\tilde{\bH}$ is a form of $\mathrm{SL}_2$ over $F'$. Hence $\tilde{\bH}$ is either $\SL_{2,F'}$ or a quaternion algebra over $F'$ which is non-split at some real place $\sigma$ of $F'$. We exclude the second case as follows. Recall that $\mathrm{rk}\bH^{\mathrm{der}} = g$. Since $\mathrm{rk}\bH \le \mathrm{rk}\GSp_{2g} = g+1$, we then have $\bH = \Gm \cdot \bH^{\mathrm{der}} = \Gm \cdot \mathrm{Res}_{F'/\QQ}\tilde{\bH}$. Thus $\tilde{\bH}$ is split at some real place $\sigma'$ of $F'$ by (SV3) of the definition of Shimura data. Write $V$ for the $\QQ$-vector space of dimension $2g$ on which $\GSp_{2g}$ naturally acts. If $\tilde{\bH}$ is a quaternion over $F'$ which is non-split at $\sigma$, then the action of $\bH_{\sigma,\RR}$ on $V_{\sigma,\RR}$ has weight $0$, while the action of $\bH_{\sigma',\RR}$ on $V_{\sigma',\RR}$ has weight $1$. This is impossible by definition of Shimura data. Hence 
$\tilde{\bH} = \SL_{2,F'}$. Thus $S$ is the Hilbert modular variety defined by the totally real field $F'$ of degree $g$. By the modular interpretation at the beginning of $\mathsection$\ref{SubsectionHilbertModularVariety}, an order of $\cO_{F'}$ is thus contained in the endomorphism ring of $A_o$. This obliges $F' = F$. Hence we are done.
\end{proof}

\begin{cor}\label{CorWeylPointDense}
Let $[o]$ be a CM point in $\mathbb{A}_g(\IQbar)$. Set
\small
\begin{align*}
\Sigma:= \{[z] \in \mathbb{A}_g(\IQbar) : &~ [z] \text{ is CM,  the only special subvariety of } \\
&~ \mathbb{A}_g\text{ passing through }[o]\text{ and }[z]\text{ is }\mathbb{A}_g\}.
\end{align*}
\normalsize
Then $\Sigma$ is dense in $\mathbb{A}_g^{\mathrm{an}}$ in the usual topology.
\end{cor}
\begin{proof}
Let $\cW$ be the set of all Weyl points in $\mathbb{A}_g$. It is known that $\cW$ is dense in $\mathbb{A}_g^{\mathrm{an}}$ in the usual topology; see \cite[Prop.2.1]{CUEquiTore}.

Consider $\mathrm{End}^0(A_o)$. There are only finitely many totally real fields of degree $g$ which are contained in $\mathrm{End}^0(A_o)$. For each such totally real field $F'$, one can associate a Hilbert modular variety $\mathcal H_{F'}$ obtained from $\Gm\cdot \mathrm{Res}_{F'/\QQ}\SL_2$. Then $\cW \setminus \bigcup_{F'} \mathcal H_{F'}(\IQbar)$ is dense in  $\mathbb{A}_g^{\mathrm{an}}$ in the usual topology.

It remains to show that $\cW \setminus \bigcup_{F'} \mathcal H_{F'}(\IQbar)$ is contained in $\Sigma$. Indeed, 
for any $[z] \in \cW \setminus \Sigma$ with $F$ the maximal totally real subfield of $\mathrm{End}^0(A_z)$,  Proposition~\ref{PropWeylPoint} implies that $[o]$ is contained in the Hilbert modular variety defined by $\Gm \cdot \mathrm{Res}_{F/\QQ}\SL_2$, and hence $F$ is a subfield of $\mathrm{End}^0(A_o)$ and has degree $g$. So we are done.
\end{proof}

\section{Analytic Subspace Conjecture and its consequence on transcendence}\label{SectionAnalyticSubspaceConj}

Let $S$ be a connected Shimura variety associated with the Shimura datum $(\bG,X)$, and let $u \colon X \rightarrow S^{\mathrm{an}}$ be the uniformization. Endow $X$ with the  \textit{$\IQbar$-structure} as in $\mathsection$\ref{SubsectionGeometricIQbarShimura}. Fix a special point $[o] \in S(\IQbar)$ and some $o \in u^{-1}([o])$. Then $o$ is a  $\IQbar$-point of $X$. The goals of this section are to formulate an analogue of W\"{u}stholz's Analytic Subgroup Theorem for the holomorphic tangent space $T_o(X)$, which we call the \textit{Analytic Subspace Conjecture}, and to show how this conjecture gives an affirmative answer to Question~\ref{QuestionNoQuadraticRelationIntro}.

\subsection{Analytic Subspace Conjecture}
Inspired by the reformulations of WAST for tori (Theorem~\ref{ThmWustholzAlgTori}) and for CM abelian varieties (Theorem~\ref{ThmWustholzCMAbVar}), we make the following conjecture, which is precisely the analogous statement for Shimura varieties.

Let $T_{{o}}X$ and $T_{[o]} S$  be the tangent space defined over $\CC$. Set 
\[
\psi := \mathrm{d}u \colon T_{{o}}X \to T_{[o]} S.
\]
Then the Harish--Chandra realization of $X$ makes $X$ into a bounded symmetric domain $\cD \subseteq T_o X$ centered at $o$. By abuse of notation, we also write $u \colon \cD \rightarrow S$.

Recall that $T_o X$ has a bi-$\IQbar$-structure defined in $\mathsection$\ref{SectionBiIQbarShimura}.

\begin{conj}\label{ConjAnalyticSubspace}
Let $z \in \cD$ be such that $[z]:=u(z) \in S(\IQbar)$. %Set $V_z$ to be the smallest
Let $V$ be a sub-vector space of $T_{[o]} S_{\IQbar}$ with $\psi(z) \in V \otimes \CC$. Then $V \supseteq \psi(V')$ for some bi-$\IQbar$-subspace $V' \subseteq T_o X$ which contains $z$.

If moreover $\mathrm{MT}(o)$ is a maximal torus, then we can take $V' = T_{[o]} S'_{\IQbar}$ for some Shimura subvariety $S'$ of $S$ which contains $[z]$ and $[o]$.
\end{conj}
%\Ginlinecomment{Need to assume $\bG^{\mathrm{der}}$ simple. Otherwise Example~9.2 is a ``counterexample'' because we have a smaller group than $\GSp_4$ (which is ``$\SL_2\times \SL_2$'').}

The linear map $\psi$ was computed for the Siegel case, \textit{i.e.} $S = \mathbb{A}_g$, $(\bG,X) = (\GSp_{2g}, \mathfrak{H}_g)$ and $\cD = \cD_g$. %Let $[o] \in \mathbb{A}_g(\IQbar)$ be a CM point,  parametrizing a CM abelian variety $A_o$. We have computed in 
Indeed by Theorem~\ref{ThmPeriodsSiegelBis}, under suitable $\IQbar$-bases, the map $\psi$ is the diagonal matrix $\frac{1}{\pi} \mathrm{diag}(\theta_j\theta_{j'})_{1\le j \le j' \le g}$, where the $\theta_j$'s are the holomorphic periods of the CM abelian variety $A_o$ parametrized by $[o]$. Note   that in this case, if $\mathrm{MT}(o)$ is a maximal torus, then $A_o$ has no square factors.%Note that the assumption of  $Z_{\GSp_{2g}}(\bT_o)$ being a maximal torus is equivalent to that $A_o$ has no square factors; see Corollary~\ref{CorNoSquareFactorRegularMT}. 

\subsection{Consequence on quadratic relations of holomorphic periods}
Let $A$ be a CM abelian variety of dimension $g$ defined over $\IQbar$ with no square factors.
 Let $\theta_1, \ldots, \theta_g$ be its holomorphic periods as defined in Convention~\ref{SubsectionConvention}.
\begin{prop}\label{PropConsequenceAnalyticSubspaceQuadraticRelations}
Assume that Conjecture~\ref{ConjAnalyticSubspace} holds true for $S = \mathbb{A}_g$, the point $[o] \in \mathbb{A}_g(\IQbar)$ parametrizing $A$, and any $V$ of dimension $\dim S-1$. Then Question~\ref{QuestionNoQuadraticRelationIntro} has an affirmative answer for $A$.
%\[
%\dim_{\IQbar} \sum_{1\le j \le j' \le g}\IQbar \theta_j\theta_{j'} = \frac{g(g+1)}{2}.
%\]
\end{prop}

\begin{proof}
Each $\theta_j$ is well-defined up to $\IQbar^*$. By abuse of notation, we write $\theta_j \in \CC$ for a representative. It is known by WAST that $\theta_j \not\in \IQbar$ for each $j \in \{1,\ldots,g\}$.

For the uniformization $u \colon \mathfrak{H}_g \rightarrow \mathbb{A}_g$, take ${o} \in u^{-1}([o])$ and the derivative of $u$ at ${o}$
\[
\psi \colon T_{{o}}\mathfrak{H}_g \rightarrow T_{[o]} \mathbb{A}_g.
\]

We start the proof by fixing $\IQbar$-bases of $T_{{o}}\mathfrak{H}_g$ and of $T_{[o]} \mathbb{A}_g$.  We have seen in Theorem~\ref{ThmPeriodsSiegelBis} that the matrix for $\psi$ is $\mathrm{diag}(\theta_j\theta_{j'}/\pi)_{1\le j \le j' \le g}$ under suitable $\IQbar$-bases $\{f_{jj'}\}_{1\le j \le {j'} \le g}$  of $T_{{o}}\mathfrak{H}_{g,\IQbar}$ and $\{e_{jj'}\}_{1\le j \le {j'} \le g}$  of $T_{[o]}\mathbb{A}_{g,\IQbar}$.  Moreover by Theorem~\ref{ThmDecompositionIntoRootSpacesSiegel}, these $\IQbar$-bases can be chosen such that each $\CC f_{jj'} = \CC e_{jj'}$ is a root space.

Write $\cD_g$ for the bounded realization of $\mathfrak{H}_g$ based at ${o}$. Then $o \in \cD_{g,\IQbar} \subseteq T_{{o}}\mathfrak{H}_{g,\IQbar}$ and $o$ is the origin. For each $z \in \cD_g$, write $[z]$ for the image of $z$ under $\cD_g \cong \mathfrak{H}_g \xrightarrow{u} \mathbb{A}_g$.

%Consider the set $G(\QQ){o}$; it is known that $G(\QQ){o}$ is contained in $\cD_{g,\IQbar}$ and is dense. 
\vskip 0.5em

\noindent\textbf{Claim: } There exists a $\IQbar$-basis $\{f_{jj'}\}_{1\le j \le {j'} \le g}$  of $T_{{o}}\mathfrak{H}_{g,\IQbar}$ as above such the point $z:= \sum f_{jj'}$ satisfies that  $[z] \in \mathbb{A}_g(\IQbar)$.

\vskip 0.5em

Let us prove this claim. Set
\begin{align*}
\tilde{\Sigma}:= \{z \in \cD_g :& ~ z \text{ is CM,  the only special subvariety of } \\
& ~ \mathbb{A}_g\text{ passing through }[o]\text{ and }[z]\text{ is }\mathbb{A}_g\}.
\end{align*}
Then  $\tilde{\Sigma}$ is dense in $\cD_g$ in the usual topology by Corollary~\ref{CorWeylPointDense}. 
%Let $\{f_{jj'}\}_{1\le j \le {j'} \le g}$ be an arbitrary $\IQbar$-eigenbasis of $T_{{o}}\mathfrak{H}_{g,\IQbar}$. 
Since $\tilde{\Sigma}$ is dense in $\cD_g$, there exists $z = \sum k_{jj'} f_{jj'} \in \tilde{\Sigma}$ with $k_{jj'} \not= 0$ for each $j, j'$. Moreover as each CM point is $\IQbar$ in $T_o \mathfrak{H}_g$, we have that $k_{jj'} \in \IQbar$. Thus $z = \sum k_{jj'} f_{jj'} \in \tilde\Sigma$ with each $k_{jj'} \in \IQbar^*$. Now the claim holds true by replacing each $f_{jj'}$ by $k_{jj'}f_{jj'}$.% In order words, $z=(1,\ldots,1)$ under the $\IQbar$-basis $\{f_{jj'}\}_{1\le j \le {j'} \le g}$ for the geometric $\IQbar$-structure.

\vskip 0.5em

Finally, set
\[
e_{jj'} := \frac{\psi(f_{jj'})}{\theta_j\theta_{j'}/\pi}.
\]
By Theorem~\ref{ThmPeriodsSiegelBis} $\{e_{jj'}\}_{1\le j \le j' \le g}$ is a $\IQbar$-eigenbasis of $T_{[o]} \mathbb{A}_{g,\IQbar}$.

\vskip 0.5em

Now assume
\begin{equation}\label{EqQuadraticRelation}
\sum_{1\le j \le j'\le g} c_{jj'} \theta_j \theta_{j'} = 0
\end{equation}
with $c_{jj'} \in \IQbar$. 

 Define
\[
V:= \left\{\sum a_{jj'} e_{jj'} \in T_{[o]}\mathbb{A}_{g,\IQbar}: \sum c_{jj'}a_{jj'} = 0 \right\};
\]
it is a $\IQbar$-subspace of $T_{[o]}\mathbb{A}_{g,\IQbar}$ because each $c_{jj'} \in \IQbar$. 
A direct computation shows that $\psi(z) = \psi(\sum f_{jj'}) = \frac{1}{\pi}\sum \theta_j \theta_{j'} e_{jj'}$. So  $\psi(z) \in V\otimes \CC$ by \eqref{EqQuadraticRelation}.%:= \left\{\sum a_{jj'} e_{jj'} \in T_{[o]}\mathbb{A}_{g,\IQbar}: \sum c_{jj'}a_{jj'} = 0 \right\},   %Hence by minimality of $V_z$, we have $V_z \subseteq V \not= T_{[o]} \mathbb{A}_{g,\IQbar}$.

Hence we can apply Conjecture~\ref{ConjAnalyticSubspace} to the point $z\in \cD_g$ and $V$. So there exists a bi-$\IQbar$-subspace $V'$ of $T_{{o}}\mathfrak{H}_{g,\IQbar}$
containing  $z = \sum f_{jj'}$ such that $\psi(V')\subset V$..

If $\mathrm{MT}(o) = \mathrm{MT}(A)$ is a maximal torus, then the ``moreover'' part of Conjecture~\ref{ConjAnalyticSubspace} furthermore implies that we can take $V' = T_{[o]}S'_{\IQbar}$ for some special subvariety $S'$ of $\mathbb{A}_g$. But then by choice of $z$, we have $S' = \mathbb{A}_g$. So $V = T_{[o]}\mathbb{A}_{g,\IQbar}$ and hence $c_{jj'} = 0$ for all $j,j'$.

In general, there exists an equivalence relation among the pairs in $\{1,\ldots,g\}^2$ defined as follows: $(j,j') \sim (k,k')$ if and only if $\theta_j\theta_{j'} \cong \theta_k\theta_{k'}$. Let $\{1,\ldots,g\}^2 = \coprod_s J_s$ be the partition into equivalence classes, and for each $s$ let us fix a pair $(j_s,j'_s) \in J_s$. Then \eqref{EqQuadraticRelation} can be rewritten such that $c_{jj'} = 0$ if $(j,j') \not= (j_s,j'_s)$ for any $s$. Then our goal is to prove that $c_{j_sj'_s} = 0$ for each $s$. With this choice 
\[
V:= \left\{\sum a_{jj'} e_{jj'} \in T_{[o]}\mathbb{A}_{g,\IQbar}: \sum_{s} c_{j_sj_s'}a_{j_sj_s'} = 0 \right\}.
\]
%as $\sum_s c_{j_sj'_s} \theta_{j_s}\theta_{j'_s} = 0$. 

By Proposition~\ref{PropBiQbarEasyProp}, the bi-$\IQbar$-subspace $V'$ equals $\bigoplus_s V'_s$ for some $V'_s\subseteq \bigoplus_{(j,j') \in J_s} \CC f_{jj'}$. Since $z = \sum f_{jj'} \in V'$, we have $\CC\cdot \sum_{(j,j') \in J_s} f_{jj'} \subseteq V'_s$ for each $s$. Therefore  the subspace $ \psi(V') \supset V$ contains the point  $e_{j_sj'_s} + \sum_{(j,j') \in J_s \setminus \{(j_s,j'_s)\}} \frac{\theta_{j}\theta{j'}}{\theta_{j_s}\theta_{j'_{s}}}e_{jj'}$ for each $s$. But from the last paragraph we assumed $c_{jj'} = 0$ for any $(j,j') \in J_s \setminus \{(j_s,j'_s)\}$. So $c_{j_sj'_s} = 0$ by definition of $V$. Hence we are done.
\end{proof}

\section{Grothendieck's Period Conjecture and consequences on CM periods}\label{SectionGroPeriodConj}
This section is rather independent of the rest of the paper. We recall the basic versions of Grothendieck's Period Conjecture, and explain some consequences on the algebraic relations among $\pi$ and holomorphic periods of CM abelian varieties. We show that all non-trivial such algebraic relations are generated by a fixed set of monomials (Proposition~\ref{PropCMAlgebraicRelationMonomial}).% of degree $\ge 3$.

\subsection{Statement of the conjecture}
Let $X$ be a smooth projective irreducible variety defined over $\IQbar$. Let
\[
\beta \colon H^*_{\mathrm{dR}}(X) \otimes_{\IQbar} \CC \xrightarrow{\sim} H^*(X,\QQ)\otimes_{\QQ}\CC
\]
be the de Rham--Betti comparison. Under suitable bases, $\beta$ is a matrix in $\mathrm{GL}_N(\CC)$. Write $(\beta_{jj'})_{1 \le j,j'\le N}$ for this matrix.

For each $n \in\mathbb{N}$, each algebraic cycle in $X^n$ induces $\IQbar$-polynomial relations, all homogenous of degree $n$, for the $\beta_{jj'}$'s. Indeed, for $n = 1$, an algebraic cycle of $X$ of codimension $d$ gives an element in $H^{2d}_{\mathrm{dR}}(X) \cap \beta^{-1}(H^{2d}(X,\QQ))$. 
%a Hodge cycle in $X$ is a non-zero element in $H^{d,d}(X) \cap \beta^{-1}(H^{2d}(X,\QQ))$. \Ginlinecomment{What is this??? On the de Rham side.} 
Hence there exists a vector $\mathbf{v} = \begin{bmatrix} v_1 & \cdots  & v_N \end{bmatrix}^{\intercal} \in \IQbar^N$ such that $\beta\mathbf{v} \in \QQ^N$. 
 Writing it out, we obtain $N$ linear relations over $\IQbar$ among the $\beta_{jj'}$'s. %\Ginlinecomment{Pb: These relations should not have constant terms. But here there is one. How to get rid of it?} 
For general $n$, the assertion follows from K\"{u}nneth's Formula.

Notice that among such relations, some of them are products of polynomials of lower degrees. This notably happens for elements in $\mathrm{End}(X)$, which are cycles in $X^2$ but give linear relations. Indeed, for each $\alpha \in \mathrm{End}(X)$, we have the following commutative diagram
\[
\xymatrix{
H^*_{\mathrm{dR}}(X) \otimes_{\IQbar} \CC \ar[r]^-{\beta} \ar[d]_-{\alpha^*} & H^*(X,\QQ)\otimes_{\QQ}\CC \ar[d]^-{\alpha^*} \\
H^*_{\mathrm{dR}}(X) \otimes_{\IQbar} \CC \ar[r]^-{\beta} & H^*(X,\QQ)\otimes_{\QQ}\CC.
}
\]
Under suitable bases, the $\alpha^*$ on the left is an $N\times N$-matrix with entries in $\IQbar$, and the $\alpha^*$ on the right is an $N\times N$-matrix with entries in $\QQ$. Thus this commutative diagram gives linear relations over $\IQbar$ among the $\beta_{jj'}$'s.

\begin{conj}[Grothendieck, strong version] 
Consider the $\IQbar$-variety $\mathrm{GL}_N$. Set $I$ to be the ideal of the subvariety $\beta^{\IQbar\text{-Zar}}$.\footnote{Thus $\beta$ is the generic point of this subvariety.} Then $I$ is generated by the polynomial relations obtained from all algebraic cycles in all powers of $X$.
\end{conj}
Let us take a closer look at this conjecture. Write $I = \bigoplus_{j \ge 1} I_j$ the homogenous decomposition. A consequence of this conjecture for $I_1$ is that $I_1$ is generated by relations obtained from elements in $\mathrm{End}(X)$. In case of commutative algebraic groups, this consequence can be deduced from W\"{u}stholz's Analytic Subgroup Theorem. Little is known for $j \ge 2$.

There is also a weaker version of this conjecture, sometimes also known as Grothendieck's Period Conjecture. Let $G_{\mathrm{mod}}$ be the motivic Galois group defined by Nori \cite{Nori} and Ayoub \cite{Ayoub}. In case of the first cohomology for abelian varieties, $G_{\mathrm{mod}}$ is precisely the Mumford--Tate group  by Andr\'{e} \cite{AndreMotivicGaloisGroupMTGroup}.%Deligne \cite[Thm.1]{De1980} (Hodge = absolute Hodge).
\begin{conj}[Grothendieck, weak version]
$\mathrm{trdeg}_{\IQbar}(\beta_{jj'})_{1\le j,j'\le N} = \dim G_{\mathrm{mod}}$.
\end{conj}
The strong version implies the weak version. Conversely, the weak version implies the strong version \textit{if} the universal period torsor is connected. We refer to \cite[7.5.2.2.Prop]{AndreMotifs} for more details.

\subsection{Consequences on holomorphic CM periods}
Let $A$ be a CM abelian variety over $\IQbar$ with no square factors. Let $\theta_1,\ldots,\theta_g$ be its holomorphic periods. %We see two consequences of the weak version of Grothendieck's Period Conjecture.

Let $\bT_0 := \mathrm{MT}(A)$. Then $\bT_0$ is a subtorus of $\GSp_{2g}$; it contains $\Gm$, the center of $\mathrm{GSp}_{2g}$. Thus the multiplier $\mathrm{GSp}_{2g} \rightarrow \Gm$, $\alpha \mapsto \det(\alpha)^{1/g}$, is non-trivial when restricted to $\bT_0$. Hence there is a non-constant homomorphism $X^*(\Gm) \rightarrow X^*(\bT_0)$. Let $\epsilon_0^*$ be the image of $1$ of this homomorphism.%; a more precise choice will be made below.

By Lemma~\ref{LemmaNoSquareFactorRegularMT}, $\bT := Z_{\GSp_{2g}}(\bT_0)$ is a maximal torus of $\GSp_{2g}$. The last sentence of Remark~\ref{RemarkCMAbVarDiagonal} yields 
 a $\ZZ$-basis $\{\epsilon_0^*,\epsilon_1^*,\ldots,\epsilon_g^*\}$ of $X^*(\bT)$ satisfying the following property:  For de Rham--Hodge comparison $\beta \in \mathrm{GSp}_{2g}(\CC)$, %, write $\bar{\beta}$ for the image of $\beta$ under the natural quotient $\mathrm{GSp}_{2g} \to G$. 
we have $\epsilon_j^*({\beta}) = \theta_j$ for each $j \in \{1,\ldots,g\}$. %\Ginlinecomment{Add a short explanation?}% We also choose $\epsilon_0^*$ such that $\epsilon_0^*(\beta) = 2\pi i$.%Hence $(\epsilon_i^*+\epsilon_j^*)({\beta}) = \theta_i\theta_j$ for each $1 \le i<j\le g$.

\begin{prop}
%Suppose $A$ has no square factors.
Assume the weak version of Grothendieck's Period Conjecture for $A$. Then each $\theta_j\theta_{j'}$ is a transcendental number.
\end{prop}
\begin{proof} 
Consider the root system $\Phi(\bT, \GSp_{2g}) \subseteq X^*(\bT)$ defined  in \cite[13.18]{BorelLinear-Algebrai}; it contains $2\epsilon_j^*$, $\epsilon_j^*+\epsilon_{j'}^*$ and $\epsilon_j^*-\epsilon_{j'}^*$ for all $1 \le j < j'\le g$.% \Ginlinecomment{This is for $\mathrm{Sp}_{2g}$. For $\mathrm{GSp}_{2g}$???}

As $Z_{\GSp_{2g}}(\bT_0)$ is solvable, $\bT_0$ is a semi-regular torus by \cite[13.1.Prop]{BorelLinear-Algebrai}. So $\bT_0 \not\subseteq \ker \chi$ for any $\chi \in \Phi(\bT, \GSp_{2g})$ by \cite[13.2.Prop]{BorelLinear-Algebrai}. In particular
\begin{equation}\label{Eq}
(\epsilon_j^*+\epsilon_{j'}^*)|_{\bT_0} \not\equiv 0 \quad \text{ for all }1\le j < j' \le g.
\end{equation}

By our choice of $\epsilon_j^*$, we have $(\epsilon_j^*+\epsilon_{j'}^*)({\beta}) = \theta_j\theta_{j'}$ for each $1 \le j < j' \le g$.

Now comes the step where we assume the weak version of Grothendieck's Period Conjecture for $A$. Then ${\beta}$ is the generic point of $\bT_0$. Hence \eqref{Eq} implies $(\epsilon_j^*+\epsilon_{j'}^*)({\beta}) \not\in \IQbar$ for all $1 \le j < j' \le g$. We are done.% Then the conclusion follows from the last paragraph.
\end{proof}

%The techniques used in this proof can be adapted to prove the following stronger result.
\begin{prop}\label{PropCMAlgebraicRelationMonomial}
%Suppose $A$ has no square factors.
Assume the weak version of Grothendieck's Period Conjecture for $A$. Then 
there exist monomials over $\IQbar$ in $\pi$ and the $\theta_j$'s (with $j \in \{1,\ldots,g\}$) such that each algebraic relations over $\IQbar$ among these numbers are generated by these monomials.
\end{prop}
\begin{proof}
Set $l := \dim \bT - \dim \bT_0$. 

The inclusion $\bT_0 < \bT$ induces a group homomorphism $X^*(\bT) \rightarrow X^*(\bT_0)$, which is furthermore surjective. The kernel of this homomorphism, denoted by $N$, is a free-$\mathbb{Z}$-module of rank $l$. Let $\{a_{j0}\epsilon_0^*+\cdots + a_{jg}\epsilon_g^*: j = 1,\ldots, l\}$ be a set of generators of $N$. Then
\begin{equation}\label{EqT0T}
\bT_0 = \{ t \in \bT : \prod_{k=0}^g \epsilon_k^*(t)^{a_{jk}} = 1 \text{ for all }j =1,\ldots, l\}.
\end{equation}

Denote by $(t_0,t_1,\ldots,t_g)$ the coordinates of $\mathbb{A}_{\IQbar}^{g+1}$. The image of $\bT_{\IQbar} \xrightarrow{(\epsilon_0^*,\ldots,\epsilon_g^*)} \mathbb{G}_{\mathrm{m},\IQbar}^{g+1} \subseteq \mathbb{A}_{\IQbar}^{g+1}$ is an open $\IQbar$-subvariety. % as $\PP^{g+1}$ if we identify $\mathbb{A}_{\IQbar}^{g+1}$ with the affine chart $U := \{[x_0:x_1:\cdots :x_g : 1] : x_j \in \IQbar\text{ for }j \in \{0,1,\ldots,g\}\}$ of $\PP^{g+1}$. 
For each $j \in \{1,\ldots,l\}$, set $P_j := \prod_{k=0}^g t_k^{a_{jk}} -  1$. %; it is the homogeneous polynomial obtained from $\prod_{k=0}^g (x_k/x_{g+1})^{a_{jk}} -1 $. 

By \eqref{EqT0T} , as a $\IQbar$-subvariety we have
\[
\bT_0 = \bT \cap Z(P_1,\ldots, P_l).
\]
As $l = \dim \bT - \dim \bT_0$ and $\bT_{\IQbar}$ is open in $\PP_{\IQbar}^{g+1}$, by Hilbert Nullstellensatz, each polynomial $P \in \IQbar[t_0,\ldots,t_g]$ which vanishes identically on $\bT_0$ is in the radical of the ideal $(P_1,\ldots,P_l)$, which  is precisely $(P_1,\ldots,P_l)$ since the kernel of $X^*(\bT) \rightarrow X^*(\bT_0)$ is torsion free.%, so the ideal $(P_1,\ldots,P_l)$ equals its radical.

Now let us study the algebraic relations over $\IQbar$ among $2\pi i$, $\theta_1,\ldots,\theta_g$. Assume $P \in \IQbar[t_0,\ldots,t_g]$ is a polynomial such that $P(2\pi i,\theta_1,\ldots,\theta_g) = 0$. %Let $P \in \IQbar[x_0,\ldots,x_{g+1}]$ be the homogeneous polynomial associated with $Q$ (where $t_j = x_j/x_{g+1}$). Then we have $P(2\pi i,\theta_1,\ldots,\theta_g,1) = 0$. 
Since $\epsilon_0^*(\beta) = 2\pi i$ and $\epsilon_j^*(\beta) = \theta_j$ for each $j \in \{1,\ldots,g\}$, we have $P(\epsilon_0^*(\beta), \epsilon_1^*(\beta) ,\ldots, \epsilon_g^*(\beta)) = 0$.

Now comes the step where we assume the weak version of Grothendieck's Period Conjecture for $A$. Then ${\beta}$ is the generic point of $\bT_0$. Thus from the last paragraph we have $P|_{\bT_0} \equiv 0$. Hence $P$ is in the ideal $(P_1,\ldots,P_l)$. Therefore each algebraic relation among $\pi$, $\theta_1,\ldots,\theta_g$ is generated by the $l$-monomials
\[
(2 \pi i)^{a_{j0}} \theta_1^{a_{j1}} \cdots \theta_g^{a_{jg}} \in \IQbar, \quad j \in \{1,\ldots, l\}.
\]
We are done.% \Ginlinecomment{Nullstellensatz: All polynomials defininig $T$ as a subvariety of $T_{\max}$ should be generated by these relations!!!}
\end{proof}

This proposition yields the following immediate corollary.
\begin{cor}
The weak version of Grothendieck's Period Conjecture for $A$ gives an affirmative answer to Question~\ref{QuestionNoQuadraticRelationIntro} for $A$.
\end{cor}

\bibliographystyle{alpha}

\end{document}